\pdfoutput=1
\documentclass[10pt,reqno]{amsart}

\usepackage[a4paper,left=20mm,right=20mm,top=15mm,bottom=15mm,marginpar=20mm]{geometry}

\usepackage{amsmath}
\usepackage{amssymb}
\usepackage{amsthm}
\usepackage[latin1]{inputenc}
\usepackage{eurosym}
\usepackage[dvips]{graphics}
\usepackage{graphicx}
\usepackage{epsfig}
\usepackage{hyperref}
\usepackage{dsfont}
\usepackage[nocompress, space]{cite}
\usepackage{color}
\usepackage[displaymath,mathlines]{lineno}
\usepackage{comment}
\usepackage{float}
\usepackage{subcaption}

\usepackage[normalem]{ulem}

\allowdisplaybreaks

\usepackage{hyperref}

\usepackage{ifthen}

\makeindex

\renewcommand{\div}{\operatorname{div}}

\newcommand{\Hh}{{\overline{H}}}

\def\leq{\leqslant}
\def\geq{\geqslant}

\numberwithin{equation}{section}

\newtheoremstyle{thmlemcorr}{10pt}{10pt}{\itshape}{}{\bfseries}{.}{10pt}{{\thmname{#1}\thmnumber{
#2}\thmnote{ (#3)}}}
\newtheoremstyle{thmlemcorr*}{10pt}{10pt}{\itshape}{}{\bfseries}{.}\newline{{\thmname{#1}\thmnumber{
#2}\thmnote{ (#3)}}}
\newtheoremstyle{defi}{10pt}{10pt}{\itshape}{}{\bfseries}{.}{10pt}{{\thmname{#1}\thmnumber{
#2}\thmnote{ (#3)}}}
\newtheoremstyle{remexample}{10pt}{10pt}{}{}{\bfseries}{.}{10pt}{{\thmname{#1}\thmnumber{
#2}\thmnote{ (#3)}}}
\newtheoremstyle{ass}{10pt}{10pt}{}{}{\bfseries}{.}{10pt}{{\thmname{#1}\thmnumber{
A#2}\thmnote{ (#3)}}}

\theoremstyle{thmlemcorr}
\newtheorem{theorem}{Theorem}
\numberwithin{theorem}{section}
\newtheorem{lemma}[theorem]{Lemma}

\theoremstyle{thmlemcorr*}
\newtheorem{theorem*}{Theorem}
\newtheorem{lemma*}[theorem]{Lemma}
\newtheorem{corollary*}[theorem]{Corollary}
\newtheorem{proposition*}[theorem]{Proposition}
\newtheorem{problem*}[theorem]{Problem}
\newtheorem{conjecture*}[theorem]{Conjecture}

\theoremstyle{defi}

\newtheorem{hyp}{Assumption}

\theoremstyle{remexample}
\newtheorem{remark}[theorem]{Remark}
\newtheorem{example}[theorem]{Example}

\newtheorem{pro}[theorem]{Proposition}

\theoremstyle{ass}

\newtheorem*{notations*}{Notation}

\begin{document}

\title[Hessian Riemannian and Effective Hamiltonians and Mather measures]{The Hessian Riemannian flow and Newton's method for Effective Hamiltonians and Mather measures}

\author{Diogo A. Gomes}
\address[D. A. Gomes]{
        King Abdullah University of Science and Technology (KAUST), CEMSE Division, Thuwal 23955-6900. Saudi Arabia, and  
        KAUST SRI, Center for Uncertainty Quantification in Computational Science and Engineering.}
\email{diogo.gomes@kaust.edu.sa}
\author{Xianjin Yang}
\address[X. Yang]{
        King Abdullah University of Science and Technology (KAUST), CEMSE Division, Thuwal 23955-6900. Saudi Arabia, and  
        KAUST SRI, Center for Uncertainty Quantification in Computational Science and Engineering.}
\email{xianjin.yang@kaust.edu.sa}


\thanks{
        The authors were supported by King Abdullah University of Science and Technology (KAUST) baseline funds and KAUST OSR-CRG2017-3452.
}
\date{\today}

\begin{abstract}
Effective Hamiltonians arise in several problems, including homogenization of Hamilton--Jacobi equations, nonlinear control systems, Hamiltonian dynamics, and Aubry--Mather theory. In Aubry--Mather theory, related objects, Mather measures, are also of great importance. Here, we combine ideas from mean-field games with the Hessian Riemannian flow to compute effective Hamiltonians and Mather measures simultaneously. {We prove the convergence of the Hessian Riemannian flow in the continuous setting. For the discrete case, we give both the existence and the convergence of the Hessian Riemannian flow.} In addition, we explore a variant of Newton's method
that greatly improves the performance of the Hessian Riemannian flow.  In our numerical experiments, we see that our algorithms preserve the non-negativity of Mather measures and are more stable than {related} methods in problems that are close to singular. Furthermore, our method also provides a way to approximate stationary MFGs.
\end{abstract}

\maketitle

\section{Introduction}
Let $\mathbb{T}^d$ be the unit $d$-dimensional torus. Given $P\in\mathbb{R}^d$ and a {smooth} Hamiltonian $H: \mathbb{T}^d\times\mathbb{R}^d\rightarrow\mathbb{R}$, { which is coercive in the second variable,} {the effective Hamiltonian $\overline{H}(P)$ is the unique real number for which there is a periodic viscosity solution $u: \mathbb{T}^d\rightarrow \mathbb{R}$ of the Hamilton--Jacobi equation}
\begin{equation}
\label{HJ}
H(x, P+D_xu)=\overline{H}(P), \quad x\in \mathbb{T}^d.
\end{equation}
This problem, sometimes called {\em the cell problem} \cite{LPV}, appears in several applications, including homogenization of Hamilton--Jacobi equations \cite{LPV}, front propagation \cite{majda1994large}, Bloch wave-form expansion and WKB approximation of the Schr\"{o}dinger equation \cite{evans2004towards, EGom1}, homogenization of an integral function \cite{weinan1991class}, Aubry--Mather theory \cite{jauslin1999forced, weinan1999aubry}, nonlinear control systems \cite{capuzzo2001rate}, Hamiltonian dynamics \cite{weinan1999aubry, EGom1, EGom2, gomes2003regularity, evans2004towards} and in the study of the long-time behavior of Hamilton--Jacobi equations \cite{barles2000large}. 

For $H(x,p)$ continuous, periodic in $x$ and coercive in $p$, a well-known result in \cite{LPV} gives 
the existence and  uniqueness of $\overline{H}(P)$. However, explicit solutions of \eqref{HJ} are hard to find. Thus,  efficient numerical algorithms are of great interest. {Here, we consider convex  Hamiltonians that satisfy additional growth bounds. More precisely, we suppose that $H$ satisfies: 
\begin{hyp}
	\label{HStronglyConvex}
	The Hamiltonian, $H: \mathbb{T}^d\times \mathbb{R}^d\rightarrow \mathbb{R}$,  is strictly convex in $p$. More precisely, there exists a constant $\rho>0$ such that
	\begin{equation*}
	H_{p_ip_j}(x,p)\xi_i\xi_j\geq \rho |\xi|^2
	\end{equation*}
	for {$p,\xi\in \mathbb{R}^d$ and $x\in \mathbb{T}^d$. }
\end{hyp}
\begin{hyp}
	\label{GrowthBounds}
	{The Hamiltonian $H$ satisfies }
	\begin{align*}
	\begin{split}
	|D_p^2H(x,p)| &\leq C,\\
	|D_{x,p}^2H(x,p)| &\leq C(1+|p|),
	\end{split}
	\end{align*}
and
	\begin{align*}
	|D_x^2H(x,p)| &\leq C(1+|p|^2)
	\end{align*}
	for some constant $C>0$.
\end{hyp}
}

As discussed in Section \ref{previouswork},  {related} methods, except for the one in \cite{falcone2008variational},  compute only $\overline{H}$ and $u$.  However, in Aubry--Mather theory, in addition to $\overline{H}$ and $u$, it is also critical to compute related objects, 
Mather measures, see \cite{Ma, Mane2} or the survey \cite{biryuk2010introduction}. {Given a Tonelli Lagrangian   $L$ and $P\in \mathbb{R}^d$,  a {\em Mather measure} is a probability measure  $\mu\in \mathcal{P}\left(\mathbb{T}^d\times \mathbb{R}^d\right)$  that minimizes}
\begin{equation}
\label{MatherProblem}
\int_{\mathbb{T}^d\times\mathbb{R}^d}\left(L(x,v)+P\cdot v\right)d\mu(x,v),
\end{equation}
among all probability measures that satisfy the following {holonomic} constraint
\begin{equation*}
\int_{\mathbb{T}^d\times\mathbb{R}^d}\left(v\cdot \nabla\varphi\right) d\mu=0, \quad \forall  \varphi\in {C}^1(\mathbb{T}^d). 
\end{equation*}
Let $H$ be the Legendre transform of $L$,  $$H(x,p)=\sup\limits_{v}\left(p\cdot v-L(x,v)\right).$$ If $\overline{H}$ and $u$ solve \eqref{HJ}, the infimum of \eqref{MatherProblem} is $-\overline{H}(P)$, $\mu$ is supported on the graph $\left(x, -D_pH(x,P+D_xu(x))\right)$ and {is uniquely determined by its $x$-projection, denoted by $m$, which is a weak solution of}
\begin{equation}
\label{FokkerPlank}
{-\div(mD_pH(x,P+D_xu))=0, \quad x\in \mathbb{T}^d.}
\end{equation}

{Here, we seek to compute $u, \overline{H},$ and {the projected Mather measure $m$  simultaneously}. More precisely, we want to solve numerically for $(u,m,\overline{H})$ the system
	\begin{equation}
	\label{HJFPTogether}
	\begin{cases}
	H(x,P+D_xu)=\overline{H}(P), & \text{in}\ \mathbb{T}^d,\\
	-\div(mD_pH)=0, & \text{in}\ \mathbb{T}^d. 
	\end{cases}
	\end{equation} 
	Another motivation to solve $(u,m,\overline{H})$ together is that \eqref{HJFPTogether} resembles the following first-order stationary mean-field game (MFG): 
	\begin{equation}
	\label{MFG}
	\begin{cases}
	H(x,P+D_xu)=\Hh(P)+g(m), & \text{in}\ \mathbb{T}^d, \\
	-\div(mD_pH)=0, & \text{in}\ \mathbb{T}^d,  
	\end{cases}
	\end{equation}
	where $g: \mathbb{R}\rightarrow \mathbb{R}$ is a given increasing  function. Mean-field games model the behavior of rational and indistinguishable agents in a large population {(see  \cite{ll1, ll2})}. In \eqref{MFG}, 
	$u$ determines the cost for an agent at $x\in\mathbb{T}^d$, $m$ is a probability density that gives the agents' distributions, and $g$ determines the interaction between agents. When $g=0$, \eqref{MFG} reduces to \eqref{HJFPTogether}, $\Hh(P)$ is the effective Hamiltonian, and $m$ represents the projected Mather measure. 
	
	As pointed out in Example \ref{UniquenessExample} of Section \ref{section2}, the projected Mather measure in \eqref{HJFPTogether} is not unique. However, from the theory of MFGs, the uniqueness of $m$ in \eqref{MFG} is guaranteed if $g$ is increasing. Thus, we use \eqref{MFG} as an approximation to \eqref{HJFPTogether}}.

More precisely, we consider {the} following MFG that  arises in the study of  entropy penalized Mather measures  \cite{evans2003some}: 
\begin{equation}
\label{ApproxiMFG}
\begin{cases}
H(x,P+D_xu^k)=\overline{H}^k\left(P\right)+\frac{1}{k}\ln m^k \\
-\div(m^kD_pH(x,P+D_xu^k))=0,
\end{cases}
\end{equation}
where $k>0$ is an integer and  
\begin{equation}
\label{HBark}
\overline{H}^k(P)=\frac{1}{k}\ln \left(\int_{\mathbb{T}^d}e^{kH(x,P+D_xu^k)}dx\right).
\end{equation}
Under Assumptions \ref{HStronglyConvex} and \ref{GrowthBounds}, as $k\rightarrow \infty$, {$\overline{H}^k(P)$ converges to $\overline{H}(P)$}, a subsequence of $u^k$ converges to a viscosity subsolution of \eqref{HJ}, and if, up to a subsequence, $m^k$ converges to $m$, then $m$ is  the projected Mather measure \cite{evans2003some}. The convergence of $u^k$ and $m^k$ is not guaranteed, see for example the discussion in \cite{GIMY}. Here, we develop a numerical algorithm to solve \eqref{ApproxiMFG} and study numerically the convergence of $u^k$ and $m^k$ as $k\rightarrow+\infty$.

To solve \eqref{ApproxiMFG}, we construct the Hessian Riemannian flow that preserves the non-negativity of $m$. More precisely, we consider the system of PDEs:
\begin{align} 
\label{ApproxHessianMonotoneFlow}
\begin{bmatrix}
\dot{\boldsymbol{m}}\\\dot{\boldsymbol{u}}
\end{bmatrix}=-\begin{bmatrix}
\boldsymbol{m}\left(-H(x,P+D_x\boldsymbol{u})+{\boldsymbol{\overline{H}}(P)}+\frac{1}{k}\ln \boldsymbol{m}\right)\\
-\div(D_pH(x,P+D_x\boldsymbol{u})\boldsymbol{m})
\end{bmatrix},
\end{align}
where
\begin{equation}
\label{BoldHBarFirstAppear}
{\boldsymbol{\overline{H}}(P)}=\frac{\int_{\mathbb{T}^d}\left(\boldsymbol{m}H(x,P+D_x\boldsymbol{u})-\frac{1}{k}\boldsymbol{m}\ln \boldsymbol{m}\right)dx}{\int_{\mathbb{T}^d}\boldsymbol{m}}.
\end{equation}
{We note that   $\boldsymbol{m}$, $\boldsymbol{u}$, and $\boldsymbol{\overline{H}}(P)$ depends on the choice of $k$. }
In Section \ref{section2}, we establish the following convergence theorem. {For the notation, we refer the reader to the end of Section \ref{previouswork}.}
\begin{theorem}
	\label{ConvergenceApproxiU}
	Suppose that Assumptions \ref{HStronglyConvex} and  \ref{GrowthBounds} hold, and that   \eqref{ApproxHessianMonotoneFlow} {has a solution {$\left(\boldsymbol{m}, \boldsymbol{u}\right)$} such that}  {$\left(\boldsymbol{m}, \boldsymbol{u}\right)\in {C}^1\left([0,\infty);C^1_+(\mathbb{T}^d)\times C^2(\mathbb{T}^d)\right)$}. {Let} $(m^*,u^*)$ {be} the  smooth solution of \eqref{ApproxiMFG}. Then, there exists a sequence $\{t_i\}$ such that   $\boldsymbol{u}(t_i)\rightarrow u^*$ in  $W^{1,2}\left({\mathbb{T}^d}\right)$ as $i\rightarrow +\infty$. Moreover, we have $\boldsymbol{u}(t)\rightarrow u^*$ in $L^2\left(\mathbb{T}^d\right)$ and $\boldsymbol{m}(t)\rightarrow m^*$ in $L^1(\mathbb{T}^d)$, as $t\rightarrow \infty$.
\end{theorem}
{Note  that we do not prove that a solution to \eqref{ApproxHessianMonotoneFlow} exists. Instead, in Section \ref{NumericalScheme}, we discretize \eqref{ApproxHessianMonotoneFlow} in space 
and obtain a system of ODEs:
\begin{align}
\label{DynDiscreteHessianRiemannianFlow}
\begin{bmatrix}
\dot{M}\\\dot{U}
\end{bmatrix}=-
\overline{F}\begin{bmatrix}
M\\U
\end{bmatrix},
\end{align}
where $\left(M,U\right)\in \mathbb{R}^N\times \mathbb{R}^N$, $N$ is the number of grid points, and $\overline{F}$ is defined in 
\eqref{overlineF}. 
 Then, we show the existence and    convergence of the flow in \eqref{DynDiscreteHessianRiemannianFlow}, as stated in the next theorem. In the limit $t\rightarrow+\infty$, we obtain a stationary solution to a discretized version of \eqref{ApproxiMFG}.}  
\begin{theorem}
\label{discreteExistenceConvergence}
Suppose that Assumptions \ref{hypGConvex}-\ref{hypumass} hold (see Section \ref{NumericalScheme}).  Then, {\eqref{DynDiscreteHessianRiemannianFlow} admits a unique solution  $\left(M(t),U(t)\right)=\left(m_1(t),\dots,m_N(t),u_1(t),\dots,u_N(t)\right)$  on $[0,+\infty)$}.  {Let $(M^*,U^*)$ solve {\eqref{DMFG-EH} (see Section \ref{NumericalScheme})}, where $M^*=(m_1^*,\dots,m_N^*)$ and $U^*=(u_1^*,\dots,u_N^*)$.
Then}, for each $1\leq j \leq N$, as $t\rightarrow \infty$, we have $$u_j(t)\rightarrow u_j^* \qquad \text{and}\qquad 
m_j(t)\rightarrow m^*_j.$$
\end{theorem}
{In Section \ref{section4}, we explore the connection of the previous flow with  
a variant of Newton's method that is equivalent to the Crank-Nicolson scheme for \eqref{DynDiscreteHessianRiemannianFlow}.
Numerical results and performance comparisons follow in Section \ref{section6}. Particularly, our methods are stable for problems that are nearly singular.} 

{We note that in Aubry--Mather theory,  Mather measures and effective Hamiltonians have also been studied in non-convex settings \cite{2016arXiv160507532G, Cagnetti2011}, where there are still many open problems. 
In particular, fast algorithms for non-convex 
problems would be extremely interesting. However, our methods 
do not seem to apply in that 
setting. Finally, we observe that our algorithm also serves as an alternative numerical algorithm for solving stationary MFGs. 
}

\section{Previous work}
\label{previouswork}
{Several} authors studied and proposed numerical methods for
the computation of effective Hamiltonians. Here, we give a brief overview of the various approaches in the literature. 

Two approaches described in \cite{qian2003two} use the asymptotic behavior of Hamilton--Jacobi equations  to compute $\overline{H}$. The first approach, called small-$\delta$ method,  {introduces a positive parameter $\delta$ and considers} the stationary equation
\begin{equation}
\label{SmallDelta}
\delta u_\delta + H(x,P+D_xu_\delta)=0, \quad x\in \mathbb{T}^d.
\end{equation}
According to \cite{LPV}, $-\delta u_\delta$ converges uniformly to $\overline{H}(P)$ on $\mathbb{R}^d$ as $\delta\rightarrow 0$. Thus, we can choose a small $\delta$ and solve \eqref{SmallDelta} numerically to get an approximation for $\overline{H}(P)$.
The second method, called in \cite{qian2003two} the large-$T$ method, uses a large-time approximation
\begin{equation}
\label{LargeT}
\begin{cases}
u_t+H(x,P+D_xu)=0 &\quad \text{in} \quad \mathbb{T}^d\times(0,\infty), \\
u=v				  &\quad \text{in} \quad \mathbb{T}^d\times\{t=0\},
\end{cases}
\end{equation}
where $v$ is a continuous, periodic function. Under suitable assumptions, \eqref{LargeT} has a unique viscosity solution on $\mathbb{T}^d\times [0,T]$, see \cite{souganidis1985existence}, and  \cite{qian2003two} established that $-u(x,t)/t\rightarrow \overline{H}(P)$ for a general, not necessarily convex,  Hamiltonian, $H$.

Alternatively, the effective Hamiltonian can be computed using a representation formula that arises
as a dual problem of an infinite-dimensional linear programming problem \cite{EGom3, G}. 
This is the idea used in  \cite{gomes2004computing}, where  $\overline{H}(P)$ is computed through the formula,
\begin{equation*}
\overline{H}(P)=\inf_{\phi\in C^1\left(\mathbb{T}^d\right)}\sup_xH\left(x,P+D_x\phi\right),
\end{equation*}
by discretizing the spatial variable and solving the minimax problem. 

The preceding approaches are slow from the computational point of view. Thus, significant efforts have been devoted to developing fast algorithms. These include solving a homogenization problem directly \cite{oberman2009homogenization, luo2011new} and employing a Newton-type method \cite{cacace2016generalized} to solve \eqref{HJ}.

{In \cite{oberman2009homogenization, luo2011new}, given a function $f$,  the authors of} \cite{luo2011new} considered the oscillatory equation
\begin{equation*}
\begin{cases}
H(Du^\epsilon,\frac{x}{\epsilon})=f(x) \quad x\in \Omega\char`\\ \{0\}\subset \mathbb{R}^d, \\
u^\epsilon(0)=0.
\end{cases}
\end{equation*}
Then, the value of $f$ at point $x_0$, which is close enough to the minimum of $u^\epsilon-P\cdot x$, yields an approximation of $\overline{H}(P)$ \cite{luo2011new}.

The generalized Newton method in \cite{cacace2016generalized} uses a novel approach to compute the effective Hamiltonian. There, {\eqref{HJ} is discretized directly into a nonlinear system  $F(X)=0$,} where $X$ encodes a discretized version of $u$ and $\overline{H}$. Then, the resulting system is solved by the Newton method.

The focus of the preceding methods is the computation of the effective Hamiltonian and the viscosity solution. Mather measures do not play a role. In contrast, the variational method in \cite{evans2003some} approximates 
the projected Mather measure and the effective Hamiltonian by
\begin{equation*}
m^k=e^{k\left(H(P+D_xu^k,x)-\overline{H}^k(P)\right)},
\end{equation*}
\begin{equation}
\label{ApproximatedEH}
{\overline{H}^k(P)=\frac{1}{k}\ln\left(\int_{\mathbb{T}^d}e^{kH(x,P+D_xu^k)dx}\right),}
\end{equation}
where $k\in\mathbb{N}$ and $u^k$ is the minimizer of 
\begin{equation}
\label{VariationalFunctional}
I_k[u^k]=\int_{\mathbb{T}^d}e^{kH(x,P+D_xu^k)}dx
\end{equation}
subject to
\begin{equation*}
\int_{\mathbb{T}^d}u^kdx=0.
\end{equation*}
We observe that  \eqref{ApproxiMFG} is the Euler-Lagrange equation corresponding to the functional in  \eqref{VariationalFunctional}. 
If $H(x,p)$ satisfy Assumptions \ref{HStronglyConvex} and \ref{GrowthBounds},
the results in \cite{evans2003some} imply that  $\overline{H}^k(P)\rightarrow \overline{H}\left(P\right)$ as $k\rightarrow \infty$.
{Inspired by this, the authors in \cite{falcone2008variational} propose a numerical method solving} the Euler-Lagrange equation of \eqref{VariationalFunctional}  by finite-difference methods and gets $\overline{H}^k\left(P\right)$ using \eqref{ApproximatedEH}. Numerical experiments in \cite{falcone2008variational} show that this approximation is more efficient than the algorithm in \cite{gomes2004computing} but with less accuracy. However, as pointed out in \cite{falcone2008variational}, this scheme is unstable when $k$ is too large for a fixed mesh. In contrast, our methods seem to be {more stable}, as illustrated in Section \ref{section6}.
\begin{notations*}
	We use $|\cdot|$ to represent the $l^2$-norm of a matrix or a vector, and $\|\cdot\|$ to represent the $L^2$-norm of a function. 
	{Denote by $C^{1}_*(\mathbb{T}^d)$ and $C^{1}_+(\mathbb{T}^d)$, respectively, the spaces of nonnegative and strictly positive functions in $C^{1}(\mathbb{T}^d)$}. {For a Banach space $Y$, the set {${C}^1\left([0,+\infty);Y\right)$} is the space} of continuous {differentiable} functions in $t\in [0,+\infty)$, with values in $Y$. For $f, g\in L^2\left(\mathbb{T}^d\right)$, {the standard $L^2$ inner product $\left\langle f, g \right\rangle$  is}  $\int_{\mathbb{T}^d}fg$. Besides, we also denote the inner product of two vectors in a Euclidean space by $\left\langle \cdot, \cdot \right\rangle$. {We identify the $d$-dimensional torus $\mathbb{T}^d$ with $[0,1]^d$.}  Finally, {we denote by $\mathbb{R}^N_+$  the subset of vectors in $\mathbb{R}^N$ with positive components.} 
\end{notations*}
\section{MFGs and Effective Hamiltonians}
\label{section2}
To solve the cell problem and compute the projected Mather measure, we combine
 \eqref{HJ} and \eqref{FokkerPlank} into the system
\begin{equation}
\label{MFGWithHBar}
\begin{cases}
H(x,P+D_xu)=\overline{H}, \\
-\div(D_pH(x,P+D_xu)m)=0,
\end{cases}
\end{equation}
where $m \geq 0$ is a probability measure. Taking into account that
\begin{equation*}
\overline{H}=\int_{\mathbb{T}^d}H(x,P+D_xu) dx,
\end{equation*}
we define {$F:  C_*^{1}(\mathbb{T}^d)\times C^{2}(\mathbb{T}^d)\rightarrow C^1(\mathbb{T}^d)\times C(\mathbb{T}^d)$} as follows:
\begin{align}
\label{OpF}
F\begin{bmatrix}
m\\u
\end{bmatrix}=\begin{bmatrix}
-H(x,P+D_xu)+\int_{\mathbb{T}^d}H(x,P+D_xu)dx\\
-\div(D_pH(x,P+D_xu)m)
\end{bmatrix}.
\end{align}
We notice that if $u$ is the viscosity solution of \eqref{MFGWithHBar}, so is $u+C$, where $C$ is an arbitrary constant. So,  {to normalize our solutions,} we require $\int_{\mathbb{T}^d}u=0$. Hence, our goal is to solve 
\begin{equation}
\label{MFGEH}
F(m,u)=0, \quad\text{subject to }\int_{\mathbb{T}^d}m=1, \int_{\mathbb{T}^d}u=0.
\end{equation}

The previous equation may not have a solution {$(m,u)$} in {$C_*^{1}\left(\mathbb{T}^d\right)\times C^{2}\left(\mathbb{T}^d\right)$}. For example, $m$ may be singular. We tackle this matter by introducing various approximation procedures.
First, we attempt to use a monotone flow as in \cite{almulla2017two} to approximate the solution of \eqref{MFGEH}. However, we observe that this flow may not preserve the non-negativity of $m$. This leads us to introduce the Hessian Riemannian flow. Under the assumption of the existence of a solution to \eqref{MFGWithHBar} with $m>0$, we 
prove the convergence for $u$. Unfortunately, the convergence for $m$ may not hold due to the non-uniqueness of solutions of \eqref{MFGWithHBar} and the possibility of $m$ vanishing. Hence, we add an entropy penalization term to the Hessian Riemannian flow that gives both the positivity and the convergence for $m$. 

\subsection{The monotone flow}
A way to compute the solution of \eqref{MFGEH} is the monotone flow method introduced in \cite{almulla2017two}. First, we recall that the operator $F$ defined in \eqref{OpF} is monotone provided $H\left(x,p\right)$ is convex in $p$; that is, 
for {$\left(m,u\right), \left(\theta, v\right)\in C^{1}_*(\mathbb{T}^d)\times C^{2}(\mathbb{T}^d)$}, $\int_{\mathbb{T}^d}m=1$ and
$\int_{\mathbb{T}^d}\theta=1$, $F$ satisfies
{
\begin{equation}
	\label{MonotonicityFbarFormula}
	\left\langle F\begin{bmatrix}
	m\\u
	\end{bmatrix} - F\begin{bmatrix}
	\theta\\v
	\end{bmatrix}, \begin{bmatrix}
	m\\u
	\end{bmatrix} - \begin{bmatrix}
	\theta\\v
	\end{bmatrix}\right\rangle \geq 0,
\end{equation}
where we use the  ${L^2\left(\mathbb{T}^d\right)\times L^2\left(\mathbb{T}^d\right)}$ inner product.}
The monotonicity of $F$ suggests the monotone flow,
\begin{equation}
\label{TheorMonotoneFlow}
\begin{bmatrix}
\dot{\boldsymbol{m}}\\\dot{\boldsymbol{u}}
\end{bmatrix}=-F\begin{bmatrix}
\boldsymbol{m}\\\boldsymbol{u}
\end{bmatrix},
\end{equation}
where {$\left(\boldsymbol{m}, \boldsymbol{u}\right)\in {C}^1\left([0,\infty);C^1(\mathbb{T}^d)\times C^{2}(\mathbb{T}^d)\right)$}
to approximate in the limit $t\to \infty$ the stationary solutions. This approach is suggested by the following reasoning. 
If $\left(\boldsymbol{m}, \boldsymbol{u}\right)$, $\left(\widetilde{\boldsymbol{m}}, \widetilde{\boldsymbol{u}}\right)$ solve \eqref{TheorMonotoneFlow} and $\int_{\mathbb{T}^d}\boldsymbol{m}=\int_{\mathbb{T}^d}\widetilde{\boldsymbol{m}}=1$, we have
\begin{equation}
\label{MonotoneContraction}
\frac{d}{dt}\left(\|\boldsymbol{u}-\widetilde{\boldsymbol{u}}\|^2+\|\boldsymbol{m}-\widetilde{\boldsymbol{m}}\|^2\right)=-2\left\langle
F\begin{bmatrix}
\boldsymbol{m}\\\boldsymbol{u}
\end{bmatrix}-F\begin{bmatrix}
\widetilde{\boldsymbol{m}}\\
\widetilde{\boldsymbol{u}}
\end{bmatrix}, 
\begin{bmatrix}
\boldsymbol{m}\\\boldsymbol{u}
\end{bmatrix}-\begin{bmatrix}
\widetilde{\boldsymbol{m}}\\
\widetilde{\boldsymbol{u}}
\end{bmatrix}
 \right\rangle \leq 0,
\end{equation}
provided $t\geq 0, \boldsymbol{m}\geq 0$ and $\boldsymbol{\widetilde{m}}\geq 0$. Thus, if $\left(m^*, u^*\right)$ solves \eqref{MFGEH}. Then, $\left(m^*, u^*\right)$ also solves \eqref{TheorMonotoneFlow}, since $\frac{d}{dt}m^*=\frac{d}{dt}u^*=0$. Thus, if we suppose further that  {$\boldsymbol{m} \in {C}^1\left([0,\infty);C^1(\mathbb{T}^d)\right)$, $\boldsymbol{u}\in {C}^1\left([0,\infty);C^2(\mathbb{T}^d)\right)$}, $\left(\boldsymbol{m}, \boldsymbol{u}\right)$ solves \eqref{TheorMonotoneFlow},  $\int_{\mathbb{T}^d}\boldsymbol{m}=1$, $\boldsymbol{m}\geq 0$, $\int_{\mathbb{T}^d}\boldsymbol{u}=0$ and $m^*\geq 0$, we have
\begin{equation*}
\frac{d}{dt}\left(\|\boldsymbol{u}-u^*\|^2+\|\boldsymbol{m}-m^*\|^2\right)=-2\left\langle
F\begin{bmatrix}
\boldsymbol{m}\\\boldsymbol{u}
\end{bmatrix}-F\begin{bmatrix}
m^*\\
u^*
\end{bmatrix}, 
\begin{bmatrix}
\boldsymbol{m}\\\boldsymbol{u}
\end{bmatrix}-\begin{bmatrix}
m^*\\
u^*
\end{bmatrix}
\right\rangle \leq 0, 
\end{equation*}
according to \eqref{MonotoneContraction}. 
In this case, \eqref{TheorMonotoneFlow} defines a contraction in the region where $\boldsymbol{m}$ is non-negative. 

However, there are several issues with this monotone flow approach. First, we do not know if it is globally defined. Besides, the projected Mather measure may be singular. Finally, the convergence is not guaranteed either. In Example \ref{NegativityOfMonotoneflow} below, we show that the monotone flow may not preserve the non-negativity of $\boldsymbol{m}$. Hence, \eqref{TheorMonotoneFlow} may not give a global contraction. 
\begin{example}
\label{NegativityOfMonotoneflow}
Let $d=1$. We set $H(x,p)=\frac{p^2}{2}+\sin\left(2\pi x\right)$ and $P=0$. Then, the monotone flow in \eqref{TheorMonotoneFlow} becomes
\begin{equation}
\label{NegativityOfMonotoneflowEq}
\begin{bmatrix}
\dot{\boldsymbol{m}}\\\dot{\boldsymbol{u}}
\end{bmatrix}=\begin{bmatrix}
\frac{\boldsymbol{u}_x^2}{2}+\sin\left(2\pi x\right)-\int_{0}^{1}\frac{\boldsymbol{u}_x^2}{2}dx\\
\left(\boldsymbol{m}\boldsymbol{u}_x\right)_x
\end{bmatrix}.
\end{equation}
Let $\left(m_0,0\right)$ to be the initial point and $\int_{0}^{1}m_0dx=1$. It is easy to check that $\left( \boldsymbol{m}, \boldsymbol{u}\right)=\left(m_0+\sin\left(2\pi x\right)t, 0\right)$ is the solution for \eqref{NegativityOfMonotoneflowEq}. However, $\boldsymbol{m}(t)$ becomes negative in some regions as $t\rightarrow +\infty$.  
\end{example}

Another reason why the convergence may fail is that the solution of \eqref{MFGWithHBar} may not be unique, as the next example illustrates. 
\begin{example}
\label{UniquenessExample}
	Let $d=2$ and  $H(x,p)=\frac{|p|^2}{2}$. Then,  $D_pH(x,p)=p$. Let $x=(x_1,x_2)$. We choose $P=(1,0)$. Accordingly, \eqref{MFGWithHBar} becomes
	\begin{equation}
	\label{UniquenessExampleEq}
	\begin{cases}
	\frac{|P+D_xu|^2}{2}=\overline{H}, \\
	-\div\left(m\left(P+D_xu\right)\right)=0.
	\end{cases}
	\end{equation}
	It is easy to see that $\overline{H}=\frac{1}{2}$,  $u=0$ and $m=f(x_2)$, where $f$ is any function that depends only on the second component of $x$ solving \eqref{UniquenessExampleEq}. Thus, $m$ is not unique. 
\end{example}

To guarantee the non-negativity of $m$ in the monotone flow, we use the Hessian Riemannian gradient flow introduced in  \cite{alvarez2004hessian}. 

\subsection{The Hessian Riemannian gradient flow}
In \cite{alvarez2004hessian}, Alvarez et al. considered the constrained minimization problem
$$\min\{f(x)\ |\ x\in\overline{E},\  Ax=b\},$$ where $\overline{E}$ is the closure of an open, nonempty, convex set $E\subset \mathbb{R}^n$, $A \in \mathbb{R}^{m\times n}$ with $m\leq n$, $b\in \mathbb{R}^n$, and $f\in {C}^1(\mathbb{R}^n)$. {To solve this problem, the authors introduced a Riemannian metric  $g$  derived from the Hessian matrix $\nabla^2 h$ of a Legendre-type convex function \cite{alvarez2004hessian} $h$ on $E$. Then, they used the steepest descent flow to generate trajectories in the relative interior of the feasible set $\mathcal{F}:=E\cap\{x\ |\ Ax=b\}$. } 
In the steepest descent method, the authors sought a trajectory $x(t)$ solving
\begin{align}
	\label{SDM}
	\begin{cases}
		\dot{x}+\nabla_Hf_{|\mathcal{F}}(x)=0,\\
		x(0)=x^0\in\mathcal{F},
	\end{cases}
\end{align}
where $\nabla_Hf_{|\mathcal{F}}(x)$ is the projection w.r.t. $g$ of the gradient of $f$ into the admissible directions. 
According to \cite{alvarez2004hessian}, \eqref{SDM} is well-posed. Moreover, this steepest descent flow never leaves the admissible set and leads to a local minimum. 

A similar idea can be used for monotone operators and lead us to the Hessian Riemannian flow. 

\subsection{The Hessian Riemannian flow}
To guarantee the non-negativity of $m$, we introduce the Hessian Riemannian flow. {More precisely, we define the  convex function {$h: C^1_*(\mathbb{T}^d)\times C^{2}(\mathbb{T}^d)\rightarrow\mathbb{R}$}  such that}
\begin{equation*}
h(m,u)=\int_{\mathbb{T}^d}m\ln m + \frac{1}{2}u^2dx.
\end{equation*}
{The Hessian of $h$, $\nabla^2h$, evaluated at $(m,u)$ is defined,  for any $\left(\mu_1,v_1\right), \left(\mu_2,v_2\right)\in C_*^1\left(\mathbb{T}^d\right)\times C^2\left(\mathbb{T}^d\right)$, by
	\begin{equation*}
	\nabla^2h(m,u)\left[\begin{pmatrix}
	\mu_1\\ v_1
	\end{pmatrix}, \begin{pmatrix}
	\mu_2\\ v_2
	\end{pmatrix}\right]=\int_{\mathbb{T}^d}\left(\mu_1,v_1\right)\mathcal{H}\begin{pmatrix}
	\mu_2\\ v_2
	\end{pmatrix}dx,
	\end{equation*}
where $\mathcal{H}$ is
\begin{equation*}
\mathcal{H}=\begin{bmatrix}
\frac{1}{m} & 0\\
0 & 1\\
\end{bmatrix}.
\end{equation*}
Considering the projection onto $\int_{\mathbb{T}^d}m=1$, we redefine the function $F$ in \eqref{OpF} as
\begin{align*}
F\begin{bmatrix}
m\\u
\end{bmatrix}=\begin{bmatrix}
-H(x,P+D_xu)+\frac{\int_{\mathbb{T}^d}H(x,P+D_xu)m 
}{\int_{\mathbb{T}^d}m}\\
-\div(D_pH(x,P+D_xu)m)
\end{bmatrix}.
\end{align*}
Then, we consider the Hessian Riemannian flow, 
\begin{equation*}
	\begin{bmatrix}
		\dot{\boldsymbol{m}}\\\dot{\boldsymbol{u}}
	\end{bmatrix}=-(\mathcal{H})^{-1} F\begin{bmatrix}
	\boldsymbol{m}\\\boldsymbol{u}
\end{bmatrix},
\end{equation*}
}which can be rewritten as 
\begin{align}
	\label{ContinuousHMF}
	\begin{bmatrix}
		\dot{\boldsymbol{m}}\\\dot{\boldsymbol{u}}
	\end{bmatrix}=-\begin{bmatrix}
	\boldsymbol{m}\left(-H(x,P+D_x\boldsymbol{u})+\frac{\int_{\mathbb{T}^d}H(x,P+D_x\boldsymbol{u})\boldsymbol{m}}{\int_{\mathbb{T}^d}\boldsymbol{m}}\right)\\
	-\div(D_pH(x,P+D_x\boldsymbol{u})\boldsymbol{m})
\end{bmatrix}.
\end{align}
The mass of $\boldsymbol{m}$ is preserved by this flow, because
\begin{equation}
\label{Contm}
	\int_{\mathbb{T}^d}\dot{\boldsymbol{m}}=-\int_{\mathbb{T}^d}\boldsymbol{m}\left(-H(x,P+D_x\boldsymbol{u})+\frac{\int_{\mathbb{T}^d}H(x,P+D_x\boldsymbol{u})\boldsymbol{m}}{\int_{\mathbb{T}^d}\boldsymbol{m}}\right)=0.
\end{equation}
{The positivity of $\boldsymbol{m}$ follows by the existence of solutions to the flow in  \eqref{ContinuousHMF}, the continuity of $H$, and \eqref{Contm}. More precisely, assume that $(\boldsymbol{m},\boldsymbol{u})$ is a solution to \eqref{ContinuousHMF} in ${C}^1\left([0,\infty);C^1(\mathbb{T}^d)\times C^2(\mathbb{T}^d)\right)$. Then, the first equation of \eqref{ContinuousHMF} gives
\begin{align*}
\frac{d}{dt}\ln \boldsymbol{m}=H(x,P+D_x\boldsymbol{u})-\frac{\int_{\mathbb{T}^d}H(x,P+D_x\boldsymbol{u})\boldsymbol{m}}{\int_{\mathbb{T}^d}\boldsymbol{m}}.
\end{align*}
The continuity of $H$, the boundedness of $(\boldsymbol{m},\boldsymbol{u})$, and \eqref{Contm} imply that the right-hand side of the prior equation is bounded. Then, for any given $t\in [0,+\infty)$, $\ln \boldsymbol{m}(t)$ is bounded. Hence, $\boldsymbol{m}(t)$ is positive. 

Next, we have the following convergence result.
}
\begin{pro}
	\label{ConvergenceU}
	Suppose that Assumption \ref{HStronglyConvex} holds and that {$\left(\boldsymbol{m}, \boldsymbol{u}\right)\in {C}^1\left([0,\infty);C^1_+(\mathbb{T}^d)\times C^2(\mathbb{T}^d)\right)$} is the solution of \eqref{ContinuousHMF}.  Assume further that $\int_{\mathbb{T}^d}\boldsymbol{m}(0)=1$ and $\int_{\mathbb{T}^d}\boldsymbol{u}(0)=0$. Moreover, suppose that $(m^*,u^*)$ solves \eqref{MFGEH} and {$(m^*,u^*)\in C^1_+(\mathbb{T}^d)\times C^2(\mathbb{T}^d)$}. Then, there exists a sequence, $\{t_i\}$, such that
	\begin{equation*}
	\lim_{i\rightarrow +\infty}\int_{\mathbb{T}^d}\left|D_xu^*-D_x\boldsymbol{u}(t_i)\right|^2m^*dx\rightarrow 0.
	\end{equation*}
	In addition, 
	\begin{equation*}
	\int_{\mathbb{T}^d}m^*\ln m^*\leq  \int_{\mathbb{T}^d}m^*\ln\boldsymbol{m}(t)+C,
	\end{equation*}
	where $C$ is a constant.
\end{pro}
\begin{proof}
	We notice that, if $\int_{\mathbb{T}^d}\boldsymbol{u}(0)=0$, we have $\int_{\mathbb{T}^d}\boldsymbol{u}(t)=0$, since
	\begin{equation*}
	\frac{d}{dt}\int_{\mathbb{T}^d}\boldsymbol{u}(t)=\int_{\mathbb{T}^d}\div(D_pH(x,P+D_x\boldsymbol{u})\boldsymbol{m})dx=0,
	\end{equation*}
	by the periodicity of $\boldsymbol{u}(t)$ and $\boldsymbol{m}(t)$. 
	
	For the convergence, we define a Lyapunov function for $t>0$,
	\begin{equation}
	\label{defphit}
	\phi(t)=\int_{\mathbb{T}^d}m^*\ln m^*-\boldsymbol{m}(t)\ln\boldsymbol{m}(t)-\left(1+\ln\boldsymbol{m}(t)\right)\left(m^*-\boldsymbol{m}(t)\right)dx+\frac{1}{2}\left\|u^*-\boldsymbol{u}(t)\right\|^2.
	\end{equation}
	Because  $\int_{\mathbb{T}^d}m^*= \int_{\mathbb{T}^d}\boldsymbol{m}(t)=1$, $\phi(t)$ can be simplified as
	\begin{equation*}
	\phi(t)=\int_{\mathbb{T}^d}m^*\ln\frac{m^*}{\boldsymbol{m}(t)}dx+\frac{1}{2}\left\|u^*-\boldsymbol{u}(t)\right\|^2.
	\end{equation*}
	
	We know that $\phi(t)\geq 0$ since the mapping $z\mapsto z\ln z$ is convex for all $z\geq 0$. Next, {by differentiating $\phi(t)$ in time, using the fact that $\dot{m^*}=0, \dot{u^*}=0$, and $\int_{\mathbb{T}^d}{\dot{\boldsymbol{m}}}=0$, and integrating by parts, we get}
	\begin{align*}
	\begin{split}
	&\frac{d}{dt}\phi(t)\\
	=&\int_{\mathbb{T}^d}-\frac{\dot{\boldsymbol{m}}}{\boldsymbol{m}}\left(m^*-\boldsymbol{m}\right)dx-\left\langle \dot{\boldsymbol{u}}, u^*-\boldsymbol{u}\right\rangle\\
	\leq&\int_{\mathbb{T}^d}\left(\frac{\dot{m^*}}{m^*}-\frac{\dot{\boldsymbol{m}}}{\boldsymbol{m}}\right)\left(m^*-\boldsymbol{m}\right)dx+\left\langle \dot{u^*}-\dot{\boldsymbol{u}}, u^*-\boldsymbol{u}\right\rangle
	\end{split}\\
	=&-\int_{\mathbb{T}^d} \left(H\left(x,P+D_xu^*\right)-H\left(x,P+D_x\boldsymbol{u}\right)-{\left\langle D_pH\left(x,P+D_x\boldsymbol{u}\right), D_xu^*-D_x\boldsymbol{u}\right\rangle}\right)\boldsymbol{m}\\
	&-\int_{\mathbb{T}^d} \left(H\left(x,P+D_x\boldsymbol{u}\right)-H\left(x,P+D_xu^*\right)-{\left\langle D_pH\left(x,P+D_xu^*\right), D_x\boldsymbol{u}-D_xu^*\right\rangle}\right)m^*\\
	\leq& -\rho\int_{\mathbb{T}^d}\left|D_xu^*-D_x\boldsymbol{u}\right|^2(m^*+\boldsymbol{m})dx,
	\end{align*}
	where we apply Assumption \ref{HStronglyConvex} in the last inequality. Then, we have
	\begin{equation}
	\label{LyapunovIneq}
	\frac{d}{dt}\phi(t)+\rho\int_{\mathbb{T}^d}\left|D_xu^*-D_x\boldsymbol{u}\right|^2(m^*+\boldsymbol{m})dx\leq 0.
	\end{equation}
	Hence, $\boldsymbol{u}(t)$ is bounded in $L^2$ and $\int_{\mathbb{T}^d}\left|D_xu^*-D_x\boldsymbol{u}\right|^2(m^*+\boldsymbol{m})dx \in L^1\left(\left[0,+\infty\right)\right)$. By Lemma \ref{ConvergenceOfL1} below, we know that 
	\begin{align*}
	0=&\liminf_{t\rightarrow +\infty}\int_{\mathbb{T}^d}\left|D_xu^*-D_x\boldsymbol{u}\right|^2(m^*+\boldsymbol{m})dx.
	\end{align*}
	{So, there is a sequence $\{t_i\}$ such that }
	\begin{equation*}
	\lim_{i\rightarrow +\infty}\int_{\mathbb{T}^d}\left|D_xu^*-D_x\boldsymbol{u}(t_i)\right|^2m^*dx\rightarrow 0.
	\end{equation*}
	Besides, integrating \eqref{LyapunovIneq} from $0$ to $t$, we have
	\begin{equation*}
	\int_{\mathbb{T}^d}m^*\ln \frac{m^*}{\boldsymbol{m}(t)}\leq \phi(0).
	\end{equation*}
	So,
	\begin{equation*}
	\int_{\mathbb{T}^d}m^*\ln m^*-\phi(0)\leq  \int_{\mathbb{T}^d}m^*\ln\boldsymbol{m}(t).
	\end{equation*}
\end{proof}
\begin{lemma}
\label{ConvergenceOfL1}
Suppose that $g(t): [0,+\infty)\rightarrow [0,+\infty)$ is  continuous and $\int_0^{+\infty}g(t)<+\infty$. Then, we have
\begin{equation*}
\liminf_{t\rightarrow+\infty}g(t)=0.
\end{equation*}
\end{lemma}
\begin{proof}
	Suppose that  $\liminf_{t\rightarrow+\infty}g(t)\not=0$. {Then, we can find $t_0\geq 0$ and $\epsilon>0$ such that, for any $t>t_0$, we have $g(t)>\epsilon$.} This contradicts the fact that $\int_0^{+\infty}g(t)<+\infty$.
\end{proof}

Unfortunately, the convergence of $\boldsymbol{m}$ for \eqref{ContinuousHMF} may not hold since solutions of $F(m,u)=0$ may not be unique and $m^*$ may fail to be positive as shown in Example \ref{UniquenessExample}. This observation motivates us to introduce an entropy penalization that we discuss next.  

\subsection{Entropy penalization}
To obtain uniqueness for the projected Mather measure, we consider the entropy penalized model given by \eqref{ApproxiMFG}. 
Combining \eqref{HBark} and the first equation of \eqref{ApproxiMFG}, we get
$$m^k=e^{k\left(H(x,P+D_xu^k)-\overline{H}^k(P)\right)}.$$
Thus, $\overline{H}^k(P)$ can be rewritten as
\begin{equation*}
\overline{H}^k(P)=\frac{\int_{\mathbb{T}^d}\left(m^kH(x,P+D_xu^k)-\frac{1}{k}m^k\ln m^k\right)dx}{\int_{\mathbb{T}^d}m^k}.
\end{equation*}

{
According to \cite{evans2003some}, under Assumptions \ref{HStronglyConvex} and \ref{GrowthBounds}, for each $k$, {there exists a unique solution $(u^k,m^k, \overline{H}^k)\in C^\infty(\mathbb{T}^d)\times C^\infty\left(\mathbb{T}^d\right)\times\mathbb{R}$ to \eqref{ApproxiMFG}.}  Besides, 
\begin{equation*}
\overline{H}(P)=\lim_{k\rightarrow\infty}\overline{H}^k(P).
\end{equation*}
By passing to a subsequence, there exists a function $u$ such that
\begin{equation*}
u^k\rightarrow u \quad \text{uniformly on $\mathbb{T}^d$},
\end{equation*}
and, for each $1\leq q <\infty$,
\begin{equation*}
D_xu^k\rightharpoonup D_xu \quad \text{weakly in $L^q\left(\mathbb{T}^d;\mathbb{R}^d\right)$}.
\end{equation*}
Furthermore, if there exists a probability measure $m$ such that
$m^k\rightharpoonup m$ weakly as measures on $\mathbb{T}^d$, then $m$ is a projected Mather measure.  
Moreover,
\begin{equation*}
{H(x, P+D_xu)\leq \overline{H}(P) \quad \text{a.e. in $\mathbb{T}^d$}.}
\end{equation*}
Thus, $u$ is a subsolution for \eqref{HJ}.  The method in \cite{evans2003some} does not give the convergence of $u^k$ and $m^k$. Here, we study the algorithm for solving \eqref{ApproxiMFG} and examine how the sequences, $u^k$ and $m^k$, behave numerically in Section \ref{section6}. Solving \eqref{ApproxiMFG} is also interesting in itself since the algorithm gives another way to solve stationary MFGs. 
}

The monotone flow for \eqref{ApproxiMFG} may not preserve the mass of $\boldsymbol{m}$. Instead, we explore its Hessian Riemannian flow, which is given in \eqref{ApproxHessianMonotoneFlow}. We notice that the mass of  $\boldsymbol{m}$ is constant since 
\begin{equation*}
\int_{\mathbb{T}^d}\dot{\boldsymbol{m}}=0.
\end{equation*}

{Next, we give a lemma, which is used to prove Theorem \ref{ConvergenceApproxiU} latter. 
\begin{lemma}
	\label{EntropyInequality}
	{Suppose that $a,\epsilon \in \mathbb{R}$, $a>0$, and $0<\epsilon<1$, then, for any $z>0$, we have}
	$$a\ln \frac{a}{z}-(a-z)-a\left(\epsilon+\ln(1-\epsilon)\right)\geq \epsilon\left|z-a\right|.$$
\end{lemma}
\begin{proof}
	When $z\geq a$, we define
	$$g(z)=a\ln \frac{a}{z}-(a-z)-a\left(\epsilon+\ln(1-\epsilon)\right)-\epsilon(z-a).$$
	So, we have
	$$\frac{d g(z)}{dz}=-\frac{a}{z}+1-\epsilon.$$
	Thus, $g$ achieves its minimum when $z=\frac{a}{1-\epsilon}$.
	Since $g\left(\frac{a}{1-\epsilon}\right)=0$, {we conclude that when $z\geq a$, $g(z)\geq 0$.}
	
	Similarly, when $z<a$, we define
	$$f(z)=a\ln \frac{a}{z}-(a-z)-a\left(\epsilon+\ln(1-\epsilon)\right)+\epsilon(z-a).$$
	We differentiate $f$ with respect to $z$ and get
	$$\frac{d f(z)}{dz}=-\frac{a}{z}+1+\epsilon.$$
	Thus, $f$ achieves its minimum at $z=\frac{a}{1+\epsilon}$.
	Evaluating $f$ at $z=\frac{a}{1+\epsilon}$, we obtain
	$$f\left(\frac{a}{1+\epsilon}\right)=a\left(\ln\left(\frac{1+\epsilon}{1-\epsilon}\right)-2\epsilon\right) \geq 0.$$
	So, $f(z)\geq 0$ when $z<a$. 
	
	Therefore, {we conclude that, for any $z>0$,}
	$$a\ln \frac{a}{z}-(a-z)-a\left(\epsilon+\ln(1-\epsilon)\right)\geq \epsilon\left|z-a\right|.$$
\end{proof}
}

Then, we prove the convergence for both $\boldsymbol{m}$ and $\boldsymbol{u}$, 
\begin{proof}[Proof (of Theorem \ref{ConvergenceApproxiU})]
	{The argument is an adaptation to the proof of Proposition \ref{ConvergenceU}. As before, we have $\int_{\mathbb{T}^d}\boldsymbol{u}(t)=0$. Let $\phi$ be as in \eqref{defphit}.}
	 Differentiating $\phi$ w.r.t. $t$, we get
	\begin{align*}
	\begin{split}
	\frac{d}{dt}\phi(t)
	\leq&\int_{\mathbb{T}^d}\left(\frac{\dot{m^*}}{m^*}-\frac{\dot{\boldsymbol{m}}}{\boldsymbol{m}}\right)\left(m^*-\boldsymbol{m}\right)dx+\left\langle \dot{u^*}-\dot{\boldsymbol{u}}, u^*-\boldsymbol{u}\right\rangle\\
	=&-\int_{\mathbb{T}^d} \left(H\left(x,P+D_xu^*\right)-H\left(x,P+D_x\boldsymbol{u}\right)-D_pH\left(x,P+D_x\boldsymbol{u}\right)\left(D_xu^*-D_x\boldsymbol{u}\right)\right)\boldsymbol{m}\\
	&-\int_{\mathbb{T}^d} \left(H\left(x,P+D_x\boldsymbol{u}\right)-H\left(x,P+D_xu^*\right)-D_pH\left(x,P+D_xu^*\right)\left(D_x\boldsymbol{u}-D_xu^*\right)\right)m^*\\
	&-\frac{1}{k}\int_{\mathbb{T}^d}\left(\ln m^*-\ln \boldsymbol{m}\right)\left(m^*-\boldsymbol{m}\right)\\
	\leq& -\rho\int_{\mathbb{T}^d}\left|D_xu^*-D_x\boldsymbol{u}\right|^2(m^*+\boldsymbol{m})-\frac{1}{k}\int_{\mathbb{T}^d}\left(\ln m^*-\ln \boldsymbol{m}\right)\left(m^*-\boldsymbol{m}\right),
	\end{split}
	\end{align*}
	where we use Assumption \ref{HStronglyConvex} in the last inequality. Thus, $\phi(t)$ is decreasing, $\boldsymbol{u}(t)$ is bounded in $L^2$,  and $\int_{\mathbb{T}^d}m^*\ln m^*-\phi(0)\leq  \int_{\mathbb{T}^d}m^*\ln\boldsymbol{m}(t)$. Besides, we conclude that $$\rho\int_{\mathbb{T}^d}\left|D_xu^*-D_x\boldsymbol{u}\right|^2m^*+\frac{1}{k}\left(\ln m^*-\ln \boldsymbol{m}\right)\left(m^*-\boldsymbol{m}\right)dx \in L^1\left(\left[0,+\infty\right)\right).$$ Then, by Lemma \ref{ConvergenceOfL1}, we have
	\begin{align*}
	0=&\liminf_{t\rightarrow\infty}\int_{\mathbb{T}^d}\rho\left|D_xu^*-D_x\boldsymbol{u}\right|^2m^*+\frac{1}{k}\left(\ln m^*-\ln \boldsymbol{m}\right)\left(m^*-\boldsymbol{m}\right)dx.
	\end{align*}
	Thus, {we have a sequence $\{t_i\}$  satisfying}
	\begin{equation*}
	\int_{\mathbb{T}^d}|D_xu^*-D_x\boldsymbol{u}(t_i)|^2m^*\rightarrow 0.
	\end{equation*}
	Since $m^*$ is strictly positive on $\mathbb{T}^d$, we have
	\begin{equation*}
	\int_{\mathbb{T}^d}|D_xu^*-D_x\boldsymbol{u}(t_i)|^2\rightarrow 0.
	\end{equation*}
	By the Poincar\'{e} inequality, we obtain
	\begin{equation*}
	\|u^*-\boldsymbol{u}(t_i)\|^2\leq C\|D_xu^*-D_x\boldsymbol{u}(t_i)\|^2\rightarrow 0.
	\end{equation*}
	Thus, we conclude that 
	\begin{equation*}
	\|u^*-\boldsymbol{u}(t_i)\|_{W^{1,2}\left({\mathbb{T}^d}\right)}\rightarrow 0.
	\end{equation*}
	Also, we have
	\begin{equation*}
	\int_{\mathbb{T}^d}\left(\ln m^*-\ln \boldsymbol{m}(t_i)\right)\left(m^*-\boldsymbol{m}(t_i)\right)dx\rightarrow 0.
	\end{equation*}
	Since 
	$$m^*\ln\frac{m^*}{\boldsymbol{m}(t)} -\left(m^*-\boldsymbol{m}(t)\right)\leq \left(\ln m^*-\ln \boldsymbol{m}(t_i)\right)\left(m^*-\boldsymbol{m}(t_i)\right),$$
	we get 
	\begin{equation}
	\label{ConvEntropy}
	\int_{\mathbb{T}^d}m^*\ln m^*-\boldsymbol{m}(t_i)\ln\boldsymbol{m}(t_i)-\left(1+\ln\boldsymbol{m}(t_i)\right)\left(m^*-\boldsymbol{m}(t_i)\right)dx\rightarrow 0.
	\end{equation}
	So, we have $\phi(t_i)\rightarrow 0$. Since $\phi$ is decreasing, we have
	$
	\lim\limits_{t\rightarrow +\infty}\phi(t)=0.
$
	Accordingly, it follows that 
	\begin{equation*}
	\boldsymbol{u}(t)\rightarrow u^* \quad \text{in $L^2\left(\mathbb{T}^d\right)$}.
	\end{equation*}
	Besides, by rewriting \eqref{ConvEntropy}, we obtain
	\begin{align}
	\label{EntropyLimit}
	\lim\limits_{t\rightarrow \infty}\int_{\mathbb{T}^d}\left(m^*\ln \frac{m^*}{\boldsymbol{m}(t)}-(m^*-\boldsymbol{m}(t))\right)=0.
	\end{align}
	{By Lemma \ref{EntropyInequality}}, we get, for any $0<\epsilon<1$,
	\begin{align*}
	&\epsilon\int_{\mathbb{T}^d}\left|m^*-\boldsymbol{m}(t)\right|
	\leq \int_{\mathbb{T}^d} \left(m^*\ln \frac{m^*}{\boldsymbol{m}(t)}-(m^*-\boldsymbol{m}(t))\right)-\left(\epsilon+\ln(1-\epsilon)\right)\int_{\mathbb{T}^d}m^*.
	\end{align*}
	Then, using \eqref{EntropyLimit}, we obtain
	$$\lim\limits_{t\rightarrow \infty}\epsilon\int_{\mathbb{T}^d}\left|m^*-\boldsymbol{m}(t)\right|\leq -\left(\epsilon+\ln(1-\epsilon)\right)\int_{\mathbb{T}^d}m^*.$$
	So, we have
	\begin{align}
	\label{limitL1}
	\lim\limits_{t\rightarrow \infty}\int_{\mathbb{T}^d}\left|m^*-\boldsymbol{m}(t)\right|\leq \left(-1-\frac{\ln\left(1-\epsilon\right)}{\epsilon}\right)\int_{\mathbb{T}^d}m^*.
	\end{align}
	Because \eqref{limitL1} holds for any $\epsilon\in (0,1)$, {we consider the limit $\epsilon \rightarrow 0$ and get}
	\begin{align*}
	\lim\limits_{t\rightarrow \infty}\int_{\mathbb{T}^d}\left|m^*-\boldsymbol{m}(t)\right| = 0.
	\end{align*}
\end{proof}

\section{A numerical scheme for the Hessian Riemannian flow}
\label{NumericalScheme}
{Let $\mathbb{T}_\Delta^d$ be a uniform grid on $\mathbb{T}^d$ and let  $(x_i)_{i=1}^{N}$ be the vector of grid points. We approximate $u$ and $m$ on $\mathbb{T}_\Delta^d$ by  $U=(u_1,\dots,u_N)$ and $M=(m_1,\dots,m_N)$.} In addition, we impose periodicity of $u$ and $m$ using a straightforward convention; for $d=1$, we set $u_0=u_N$ and $m_0=m_N$. Our difference scheme for $H$ is 
\begin{align}
\label{G(U)}
G(U)=\left(G_1(U),\dots,G_N(U)\right)^T, \ \text{where}\ G_i(U) \approx H(x_i,P+D_xu(x_i)).
\end{align}
{Here is an example of a possible discretization on $H$.
\begin{example}
Let $d=1$, $p\in \mathbb{R}$, and $x\in \mathbb{T}^1$. Define the Hamiltonian $H: \mathbb{T}^1\times\mathbb{R}\rightarrow \mathbb{R}$ as
\begin{equation}
\label{QuadHami}
H(x,p)=\frac{\left|p\right|^2}{2}-\sin(2\pi x).
\end{equation}
{Thus, given $P\in \mathbb{R}$, \eqref{MFGWithHBar} becomes
	\begin{align} 
	\label{1dsyspr}
	\begin{cases}
	\frac{\left(P+u_x\right)^2}{2}-\sin(2\pi x)=\overline{H},\\
	-\left(m\left(P+u_x\right)\right)_x=0.
	\end{cases}
	\end{align}
	Let $(m,u,\overline{H})$ solve the prior system. 
	In the discretized setting, we consider equidistributed grid points on $[0,1]$, given by the $(N+1)$-dimensional vector  $X=\{ x_1,\dots,x_N\}=\{ \frac{1}{N},\dots, 1\}$.} Let $(M,U)=( m_1,\dots,m_N, u_1,\dots,u_N)$ be the approximation of $(m,u)$ on $\mathbb{T}_\Delta^d$. Then, we approximate $\left|P+u_x\right|$ at $x_i$ by
\begin{align*}
\sqrt{\left[{\min}\left\{P+\frac{u_{i+1}-u_{i}}{h},0\right\}\right]^2+\left[{\max}\left\{P+\frac{u_{i}-u_{i-1}}{h},0\right\}\right]^2}.
\end{align*}
Accordingly, we approximate $H(x_i,P+u_x)$ by
\begin{equation}
\label{DefGi}
G_i(U)=\frac{1}{2}\left[\left[{\min}\left\{P+\frac{u_{i+1}-u_{i}}{h},0\right\}\right]^2+\left[{\max}\left\{P+\frac{u_{i}-u_{i-1}}{h},0\right\}\right]^2\right] - \sin(2\pi x_i).
\end{equation}
\end{example}
} 

{Next, we study the discretization of \eqref{ApproxiMFG}.}
Let $\mathcal{L}_U:\mathbb{R}^N\rightarrow \mathbb{R}^N$ be the linearized operator of $G$ at $U\in \mathbb{R}^N$ and  $\mathcal{L}_U^*$ its adjoint operator. We define  $\widetilde{F}: \mathbb{R}^{N}_+\times \mathbb{R}^N\rightarrow \mathbb{R}^{N}\times \mathbb{R}^N$ as 
\begin{align}
\label{WidtideF}
\begin{split}
\widetilde{F}\begin{bmatrix}
M\\U
\end{bmatrix}=\begin{bmatrix}
-G_1(U)+{\widetilde{H}^k(P)}+\frac{1}{k}\ln m_1\\
\dots\\
-G_N(U)+{\widetilde{H}^k(P)}+\frac{1}{k}\ln m_N\\
\left(\mathcal{L}_U^*M\right)_1\\
\dots\\
\left(\mathcal{L}_U^*M\right)_N
\end{bmatrix}.
\end{split}
\end{align}
Then, the space-discretized version of \eqref{ApproxiMFG} is
\begin{align}
\label{DMFG-EH}
\begin{cases}
\widetilde{F}\begin{bmatrix}
M\\U
\end{bmatrix}=0,\\
m_i>0, \frac{1}{N}\sum \limits_{i=1}^{N}m_i=1,\\
{\widetilde{H}^k(P)}=\frac{\sum \limits_{i=1}^{N}\left(m_iG_i(U)-\frac{1}{k}m_i\ln m_i\right)}{\sum \limits_{i=1}^{N}m_i},
\end{cases}
\end{align}
where ${\widetilde{H}^k: \mathbb{R}^d \rightarrow \mathbb{R}}$ is a numerical approximation of the effective Hamiltonian. {We note that $\widetilde{H}^k(P)$ depends on $P$ through $G(U)$ and on the choice of $k$.} 

{Then, one can follow the discussion in Proposition 3.5 of \cite{almulla2017two} to show that \eqref{DMFG-EH} admits a unique solution. {To ensure the monotonicity of $\widetilde{F}$, we require each component of $G(U)$ to be convex.  
		\begin{hyp}
			\label{hypGConvex}
			For each $1\leq i\leq N$, the map $U\mapsto G_i(U)$ is convex for $U\in \mathbb{R}^N$.
	\end{hyp} Then, we prove below in Lemma \ref{Monotonicity} that $\widetilde{F}$ is monotone.} }
\begin{lemma}
	\label{Monotonicity}
	Suppose that Assumption \ref{hypGConvex} holds. Let $\left(M,U\right)$ and  $\left(\Theta, V\right)$ be two vectors in $\mathbb{R}_+^N\times\mathbb{R}^N$, where $M=(m_1,\dots,m_N)^T, U=(u_1,\dots,u_N)^T, \Theta=(\theta_1,\dots,\theta_N)^T$ and $V=(v_1,\dots,v_N)^T$. Moreover, $\frac{1}{N}\sum \limits_{i=1}^{N}m_i=1$ and $\frac{1}{N}\sum \limits_{i=1}^{N}\theta_i=1$.
	The operator $\widetilde{F}$ in \eqref{WidtideF} satisfies
	$$\left\langle \widetilde{F}\begin{bmatrix}
	M\\U
	\end{bmatrix} - \widetilde{F}\begin{bmatrix}
	\Theta\\V
	\end{bmatrix}, \begin{bmatrix}
	M\\U
	\end{bmatrix} - \begin{bmatrix}
	\Theta\\V
	\end{bmatrix}\right\rangle \geq 0.$$
\end{lemma}

\begin{proof}
	Let $\left(M,U\right), \left(\Theta, V\right)$ be as above. We have
	\begin{align}
	\label{FWideMonotonicity}
		&\left\langle \widetilde{F}\begin{bmatrix}
	M\\U
	\end{bmatrix} - \widetilde{F}\begin{bmatrix}
	\Theta\\V
	\end{bmatrix}, \begin{bmatrix}
	M\\U
	\end{bmatrix} - \begin{bmatrix}
	\Theta\\V
	\end{bmatrix}\right\rangle \nonumber \\
	=& \sum  \limits_{j=1}^N\left(G_j\left(V\right)-G_j\left(U\right)\right)(m_j-\theta_j)+\sum \limits_{j=1}^N\left(\left(\mathcal{L}_U^*M\right)_j-\left(\mathcal{L}_V^*\Theta\right)_j\right)(u_j-v_j) \nonumber\\
	&+\frac{1}{k}\sum_{j=1}^N\left(\ln m_j-\ln\theta_j\right)(m_j-\theta_j) \nonumber \\
	&+\sum_{j=1}^{N}\left(\frac{\sum \limits_{i=1}^{N}\left(m_iG_i(U)-\frac{1}{k}m_i\ln m_i\right)}{\sum \limits_{i=1}^{N}m_i}-\frac{\sum \limits_{i=1}^{N}\left(m_iG_i(V)-\frac{1}{k}\theta_i\ln \theta_i\right)}{\sum \limits_{i=1}^{N}\theta_i}\right)\left(m_j-\theta_j\right) \nonumber \\
	=&\sum \limits_{j=1}^N\left(\left(G_j\left(V\right)-G_j\left(U\right)\right)m_j-\left(\mathcal{L}_U^*M\right)_j(v_j-u_j)\right) \nonumber \\
	&+\sum \limits_{j=1}^N\left(\left(G_j\left(U\right)-G_j\left(V\right)\right)\theta_j-\left(\mathcal{L}_V^*\Theta\right)_j(u_j-v_j)\right) \nonumber\\
	&+\frac{1}{k}\sum_{j=1}^N\left(\ln m_j-\ln\theta_j\right)(m_j-\theta_j),
	\end{align}
	taking into account that 
	\begin{align*}
	\begin{split}
	\sum_{j=1}^{N}\left(\frac{\sum \limits_{i=1}^{N}\left(m_iG_i(U)-\frac{1}{k}m_i\ln m_i\right)}{\sum \limits_{i=1}^{N}m_i}-\frac{\sum \limits_{i=1}^{N}\left(m_iG_i(V)-\frac{1}{k}\theta_i\ln \theta_i\right)}{\sum \limits_{i=1}^{N}\theta_i}\right)\left(m_j-\theta_j\right)=0,
	\end{split}
	\end{align*}
	because $\frac{1}{N}\sum \limits_{i=1}^{N}m_i=\frac{1}{N}\sum \limits_{i=1}^{N}\theta_i=1$. 
	Since $z\mapsto \ln z$ is increasing, we get
	\begin{align*}
	\begin{split}
	\frac{1}{k}\sum_{j=1}^N\left(\ln m_j-\ln\theta_j\right)(m_j-\theta_j)\geq 0.
	\end{split}
	\end{align*}
	Because $\mathcal{L}_U^*$ is the adjoint operator of $\mathcal{L}_U$, we have 
	\begin{align*}
	\sum_{j=1}^{N}\left(\mathcal{L}_U^*M\right)_j\left(v_j-u_j\right)=\left\langle\mathcal{L}_U^*M,V-U \right\rangle=\left\langle M,\mathcal{L}_U\left(V-U\right) \right\rangle=\sum_{j=1}^{N}\left(\mathcal{L}_U\left(V-U\right)\right)_jm_j.
	\end{align*}
	Thus, 
	\begin{align*}
	\begin{split}
	&\sum \limits_{j=1}^N\left(\left(G_j\left(V\right)-G_j\left(U\right)\right)m_j-\left(\mathcal{L}_U^*M\right)_j(v_j-u_j)\right)\\
	=&\sum \limits_{j=1}^N \left(
	G_j\left(V\right)-G_j\left(U\right)-\left(\mathcal{L}_U\left(V-U\right)\right)_j
	\right)m_j\geq 0,
	\end{split}
	\end{align*}
	because of the positivity of $m_j$ and of the convexity of $G_j$. Similarly,
	\begin{align}
	\label{AnotherGMonotonicity}
	\begin{split}
	\sum \limits_{j=1}^N\left(\left(G_j\left(U\right)-G_j\left(V\right)\right)\theta_j-\left(\mathcal{L}_V^*\Theta\right)_j(u_j-v_j)\right) \geq 0.
	\end{split}
	\end{align}
	Combining \eqref{FWideMonotonicity} and \eqref{AnotherGMonotonicity}, we conclude that 
	$$\left\langle \widetilde{F}\begin{bmatrix}
	M\\U
	\end{bmatrix} - \widetilde{F}\begin{bmatrix}
	\Theta\\V
	\end{bmatrix}, \begin{bmatrix}
	M\\U
	\end{bmatrix} - \begin{bmatrix}
	\Theta\\V
	\end{bmatrix}\right\rangle \geq 0.$$
\end{proof}

{Next, we construct the Hessian Riemannian flow corresponding to \eqref{DMFG-EH}, which is a discretization of \eqref{ApproxHessianMonotoneFlow} in space.}
Let  $\left(M,U\right)=(m_1,\dots,m_N,u_1,\dots,u_N)$, $(M^0,U^0)=(m^0_1,\dots,m^0_N,u^0_1,\dots,u^0_N)\in \mathbb{R}^N_+\times \mathbb{R}^N$, $\frac{1}{N}\sum \limits_{i=1}^{N}m^0_i=1$, and $\frac{1}{N}\sum \limits_{i=1}^{N}u^0_i=0$. Then, we define $\overline{F}: \mathbb{R}^N_+\times \mathbb{R}^N\rightarrow \mathbb{R}^N\times \mathbb{R}^N$ by
\begin{align}
\label{overlineF}
\overline{F}\begin{bmatrix}
M\\U
\end{bmatrix}=\begin{bmatrix}
m_1\left(-G_1(U)+\frac{\sum\limits_{i=1}^{N}\left(m_iG_i(U)-\frac{1}{k}m_i\ln m_i\right)}{\sum \limits_{i=1}^{N}m_i}+\frac{1}{k}\ln m_1\right)\\
\dots\\
m_N\left(-G_N(U)+\frac{\sum\limits_{i=1}^{N}\left(m_iG_i(U)-\frac{1}{k}m_i\ln m_i\right)}{\sum \limits_{i=1}^{N}m_i}+\frac{1}{k}\ln m_N\right)\\
\left(\mathcal{L}_U^*M\right)_1\\
\dots\\
\left(\mathcal{L}_U^*M\right)_N
\end{bmatrix}.
\end{align}
Accordingly, the Hessian Riemannian flow is
\begin{align}
\label{DynDiscreteHessianFlow}
\begin{cases}
\begin{bmatrix}
\dot{M}\\\dot{U}
\end{bmatrix}=-
\overline{F}\begin{bmatrix}
M\\U
\end{bmatrix},\\
M(0)=M^0, U(0)=U^0.
\end{cases}
\end{align}
{To prove the local existence of \eqref{DynDiscreteHessianFlow}, we need to assume that each partial derivative of $G_i(U)$, $1\leq i\leq N$, is locally Lipschitz. 
	\begin{hyp}
		\label{hypDGLocalLipschitz}
		Let $U\in \mathbb{R}^N$, $U=(u_1,\dots,u_N)^T$. For each $1\leq i, j\leq N$, $\partial_jG_i(U)$ is locally Lipschitz.
\end{hyp}}
Under Assumptions \ref{hypGConvex} and \ref{hypDGLocalLipschitz}, $\overline{F}$ is locally Lipschitz continuous on $\mathbb{R}_+^N\times\mathbb{R}^N$.
Moreover, since $\overline{F}$ depends only on $(M,U)$, we have local existence of the solution for \eqref{DynDiscreteHessianFlow}; that is, given ${(M^0,U^0)}\in \mathbb{R}_+^N\times \mathbb{R}^N$, the initial value problem in \eqref{DynDiscreteHessianFlow} { has a unique solution on $t\in (0, T)$ for some $0<T\leq+\infty$. Same as the discussion for \eqref{ContinuousHMF}, the existence of the solution to \eqref{DynDiscreteHessianFlow} implies the positivity of $M$. }

Next, we prove the boundedness of $(M,U)$ on $(0,T)$, which then implies $T=+\infty$. 
\begin{pro}
\label{LocalBoundedness}
{Suppose that Assumption \ref{hypGConvex} holds and that \eqref{DynDiscreteHessianFlow} has a solution  $\left(M,U\right) \in \mathbb{R}_+^N\times\mathbb{R}^N$ on $[0,T)$, where $T<+\infty$, $M=(m_1,\dots,m_N)^T$ and $U=(u_1,\dots,u_N)^T$.} Then,
\begin{equation*}
\left(\sum_{j=1}^{N}\left(m_j^2(t)+u_j^2(t)\right)\right)^{\frac{1}{2}}
\end{equation*}
is bounded as $t\rightarrow T$.
\end{pro}
\begin{proof}
Let $(M^0,U^0)=(m^0_1,\dots,m^0_N,u^0_1,\dots,u^0_N)\in \mathbb{R}^N_+\times \mathbb{R}^N$, $\frac{1}{N}\sum \limits_{i=1}^{N}m^0_i=1$ and $\frac{1}{N}\sum \limits_{i=1}^{N}u^0_i=0$. Since $M(0)=M^0$,
we have $\frac{1}{N}\sum \limits_{i=1}^{N}m_i(t)=1$. {In addition, due to $m_i(t)>0$, we have that $m_i(t)$ is bounded as $t\rightarrow T$.} Let $\left(M^*,U^*\right)=\left(m_1^*,\dots,m_N^*,u_1^*,\dots,u_N^*\right)$ be the solution of $\eqref{DMFG-EH}$.  

Define  
\begin{equation}
\label{PhiBar}
\overline{\phi}(t)=\frac{1}{N}\sum_{j=1}^{N}\left(m_j^*\ln\frac{m^*_j}{m_j(t)}\right)+\frac{1}{2N}\sum_{j=1}^{N}\left(u_j(t)-u_j^*\right)^2.
\end{equation}
By the convexity of the mapping $z\mapsto z\ln z$, $z\in \mathbb{R}$, we have 
$$m_j^*\ln m^*_j-m_j(t)\ln m_j(t)-\left(1+\ln m_j(t)\right)\left(m_j^*-m_j(t)\right)\geq 0.$$
Thus,
\begin{align}
\label{positivityofPhi}
\begin{split}
0\leq &\frac{1}{N}\sum_{j=1}^{N}\left(m_j^*\ln m^*_j-m_j(t)\ln m_j(t)-\left(1+\ln m_j(t)\right)\left(m_j^*-m_j(t)\right)\right)\\
=&\frac{1}{N}\sum_{j=1}^{N}\left(m_j^*\ln\frac{m^*_j}{m_j(t)}\right),
\end{split}
\end{align}
taking into account that $\frac{1}{N}\sum\limits_{j=1}^{N}\left(m_j^*-m_j(t)\right)
=0.$
So, we conclude that $\overline{\phi}(t)\geq 0$. Define 
$$\overline{H}_{U,M}=\frac{\sum \limits_{i=1}^{N}\left(m_iG_i(U)-\frac{m_i\ln m_i}{k}\right)}{\sum \limits_{i=1}^{N}m_i}$$
and 
$$\overline{H}_{U^*,M^*}=\frac{\sum \limits_{i=1}^{N}\left(m_i^*G_i(U^*)-\frac{m_i^*\ln m_i^*}{k}\right)}{\sum \limits_{i=1}^{N}m_i^*}.$$
Differentiating $\overline{\phi}$ and using $\dot{m_j^*}=\dot{u^*_j}=0$, we get
\begin{align*}
&\frac{d \overline{\phi}}{dt}
=\frac{1}{N}\sum_{j=1}^{N}\left(-\frac{m^*_j}{m_j}\dot{m}_j\right)+\frac{1}{N}\sum_{j=1}^{N}\left(\left(u_j-u_j^*\right)\dot{u}_j\right)\\
=& \frac{1}{N}\sum_{j=1}^{N}\left(\left(\frac{\dot{m}_j^*}{m_j^*}-\frac{\dot{m}_j}{m_j}\right)\left(m_j^*-m_j\right)\right)+\frac{1}{N}\sum_{j=1}^{N}\left(\left(u_j-u^*_j\right)\left(\dot{u}_j-\dot{u}^*_j\right)\right)\\
=&\frac{1}{N}\sum_{j=1}^{N}\left(\left(G_j(U^*)-\overline{H}_{U^*,M^*}-\frac{1}{k}\ln m_j^*-G_j(U)+\overline{H}_{U,M}+\frac{1}{k}\ln m_j\right)\left(m_j^*-m_j\right)\right)\\
&-\frac{1}{N}\sum_{j=1}^{N}\left(\left(u_j^*-u_j\right)\left(\left(\mathcal{L}_{U^*}^*M^*\right)_j-\left(\mathcal{L}_{U}^*M\right)_j\right)\right)\\
=&\frac{1}{N}\sum_{j=1}^{N}\left(\left(G_j(U^*)-G_j(U)\right)m_j^*-\left(u_j^*-u_j\right)\left(\mathcal{L}^*_{U^*}M^*\right)_j\right)\\
&+\frac{1}{N}\sum_{j=1}^{N}\left(-\left(G_j(U^*)-G_j(U)\right)m_j+\left(u_j^*-u_j\right)\left(\mathcal{L}_U^*M\right)_j\right)\\
&-\frac{1}{kN}\sum_{j=1}^{N}\left(\ln m^*_j-\ln m_j\right)\left(m_j^*-m_j\right),
\end{align*}
using, as before, the identity
\begin{equation*}
\frac{1}{N}\sum_{j=1}^{N}\left(\overline{H}_{U^*,M^*}-\overline{H}_{U,M}\right)\left(m_j^*-m_j\right)=0.
\end{equation*}
Thus, using the definition of $\mathcal{L}$ and $\mathcal{L}^*$, we have 
\begin{align}
\label{MonotonicityPhiBar}
\begin{split}
&\frac{d \overline{\phi}}{dt}+\frac{1}{N}\sum_{j=1}^{N}\left(\left(G_j(U)-G_j(U^*)-\left(\mathcal{L}_{U^*}\left(U-U^*\right)\right)_j\right)m^*_j\right)\\
&+\frac{1}{N}\sum_{j=1}^{N}\left(\left(G_j(U^*)-G_j(U)-\left(\mathcal{L}_U\left(U^*-U\right)\right)_j\right)m_j\right)\\
&+\frac{1}{kN}\sum_{j=1}^{N}\left(\left(\ln m_j-\ln m_j^*\right)\left(m_j-m_j^*\right)\right)= 0.
\end{split}
\end{align}
Due to the convexity of $G_j$ and the monotonicity of $z\mapsto \ln z$, $\overline{\phi}$ is decreasing. In addition, due to  \eqref{positivityofPhi}, we have $$\frac{1}{2N}\sum_{j=1}^{N}\left(u_j(t)-u_j^*\right)^2\leq \overline{\phi}(t)\leq  \overline{\phi}(0).$$
Therefore, we conclude that 
$\left(\sum_{j=1}^{N}\left(m_j^2(t)+u_j^2(t)\right)\right)^{\frac{1}{2}}$ is bounded as $t\rightarrow T$.
\end{proof}

Next, we prove \eqref{DynDiscreteHessianFlow} is well-posed.
\begin{pro}
	\label{GlobalExistence}
	Suppose that Assumptions \ref{hypGConvex} and  \ref{hypDGLocalLipschitz} hold, then \eqref{DynDiscreteHessianFlow} admits a unique solution on $[0,+\infty)$. 
\end{pro}
\begin{proof}
	Define 
	\begin{equation}
	\label{TM}
	T_M=\sup\{T>0\ |\  \exists!\  \text{solution}\  (M,U)\ \text{of \eqref{DynDiscreteHessianFlow} on}\ [0, T)\}.
	\end{equation}
	Since \eqref{DynDiscreteHessianFlow} has local existence, we know $T_M>0$. Suppose that $T_M<+\infty$. {Then, as $t\rightarrow T_M$,  $\sum\limits_{j=1}^{N}\left(m_j^2(t)+u_j^2(t)\right)$ is  bounded  on $[0, T_M)$. Let $\omega^0$ be the set of limit points of $\left(M,U\right)$ on $[0,T_M)$. Define $$\Omega=\left\{(M(t),U(t)): t\in[0, T_M)\right\}\cup \omega^0.$$}Since $\left(M,U\right)$ is bounded, we know that $\omega^0$ is nonempty and that $\Omega$ is compact. Thus, by Lemma \ref{KCC} below, $\Omega\subset \mathbb{R}_+^N\times\mathbb{R}^N$, we can extend $\left(M,U\right)$ beyond $T_M$. The extension contradicts with the finiteness of $T_M$. So, $T_M=+\infty$. 
\end{proof}
\begin{lemma}
	\label{KCC}
	Suppose that Assumtions \ref{hypGConvex} and \ref{hypDGLocalLipschitz} hold and that $\left(M(t),U(t)\right)$ is bounded on $[0, T_M)$, where $T_M$ is defined in \eqref{TM}. Assume $T_M<+\infty$. Define $\overline{\mathbb{R}}=\mathbb{R}\cup \{-\infty,+\infty\}$ and $\overline{\mathbb{R}}_+=\mathbb{R}_+\cup\{0,+\infty\}$. Let $$\Omega=\left\{(M(t),U(t)): t\in[0, T_M)\right\}\cup\omega^0,$$ where  $\omega^0\subset \overline{\mathbb{R}}_+^N\times\overline{\mathbb{R}}^N$, be the set of limit points of $\left(M(t),U(t)\right)$ on $[0,T_M)$. Then $\Omega\subset \mathbb{R}_+^N\times\mathbb{R}^N$.
\end{lemma}
\begin{proof}
	We prove $\Omega\subset \mathbb{R}_+^N\times\mathbb{R}^N$ by contradiction. Suppose that  $\Omega\not\subset \mathbb{R}_+^N\times\mathbb{R}^N$.  We can find a sequence $t_i$ such that $\left(M(t_i),U(t_i)\right)\rightarrow \left(M^*, U^*\right)$, where $t_i<T_M$, $t_i\rightarrow T_M$, and $\left(M^*,U^*\right)\in \overline{\left(\mathbb{R}_+^N\times\mathbb{R}^N\right)}\backslash \left(\mathbb{R}_+^N\times\mathbb{R}^N\right)$. Let $M(t)=\left(m_1(t),\dots,m_N(t)\right), U(t)=\left(u_1(t),\dots,u_N(t)\right).$ From \eqref{DynDiscreteHessianFlow}, we know that
	\begin{align*}
	\frac{d}{dt}\left(\ln m_j(t)\right)=G_j(U)-\frac{\sum \limits_{l=1}^{N}\left(m_lG_l(U)-\frac{1}{k}m_l\ln m_l\right)}{\sum \limits_{l=1}^{N}m_l}-\frac{1}{k}\ln m_j.
	\end{align*}
	Thus, 
	\begin{align}
	\label{ContraEq}
	\ln \frac{m_j(t_i)}{m_j(0)}+\frac{1}{k}\int_{0}^{t_i}\ln m_jds=\int_{0}^{t_i}\left(G_j(U)-\frac{\sum \limits_{l=1}^{N}\left(m_lG_l(U)-\frac{1}{k}m_l\ln m_l\right)}{\sum \limits_{l=1}^{N}m_l}\right)ds.
	\end{align}
	Since $(M(t_i),U(t_i))\rightarrow (M^*,U^*)$, the left-hand side of \eqref{ContraEq} converges to $-\infty$. However, by Proposition \ref{LocalBoundedness}, $(M^*,U^*)$ is bounded. So, the right-hand side of \eqref{ContraEq} is finite, which gives a contradiction.
\end{proof}
{Next, we study the convergence of the Hessian Riemannian flow in \eqref{DynDiscreteHessianFlow}. We require the convexity of $G$. 
\begin{hyp}  
	\label{hypGStrictlyConvexity} 
	{Define the operator $\Gamma: \{1,\dots,d\}\times \mathbb{R}\mapsto \mathbb{R}$ such that} $\Gamma(i,\cdot)$ is the forward difference for a given grid vertex in the direction, $i$. Let $U=\{u_1,\dots,u_N\}$ and $V=\{v_1,\dots,v_N\}$ be two sets of different values for the same grid on $\mathbb{T}^d$. Then, there exists a constant $\varrho$ such that
	\begin{align}
	\label{NumericalPoincare}
	\frac{1}{h^2}\sum_{i=1}^{d}\sum_{j=1}^{N}\left(\Gamma(i,u_j)-\Gamma(i,v_j)\right)^2\leq \varrho \sum_{j=1}^{N}\left(G_j(U)-G_j(V)-\nabla G_j(V)^T\left(U-V\right)\right).
	\end{align}
\end{hyp}
\begin{remark}
	In particular, for $d=1$, \eqref{NumericalPoincare} is reduced to 
	\begin{equation}
	\label{NumericalPoincare1D}
	\sum_{j=1}^{N}\left(\frac{u_j-u_{j+1}}{h}-\frac{v_j-v_{j+1}}{h}\right)^2\leq \varrho \sum_{j=1}^{N}\left(G_j(U)-G_j(V)-\nabla G_j(V)^T\left(U-V\right)\right).
	\end{equation}
\end{remark}
{The next proposition shows that the function $G$ defined in \eqref{DefGi} satisfies Assumption \ref{hypGStrictlyConvexity}.

	\begin{pro}
		$G_i, 1\leq i\leq N$ defined in \eqref{DefGi} satisfies \eqref{NumericalPoincare1D}. 
	\end{pro}
	\begin{proof}
		From the definition in \eqref{DefGi}, we notice that
		\begin{align}
		\label{VerifyGiEq1}
		\begin{split}
		&\sum_{j=1}^{N}\left(G_j(U)-G_j(V)-\left(\nabla G_j(V)\right)^T(U-V)\right)\\
		=&\left(\sum_{j=1}^{N}G_j(U)\right)-\left(\sum_{j=1}^{N}G_j(V)\right)-\left(\nabla\left(\sum_{j=1}^{N}G_j(V)\right)\right)^T\left(U-V\right).
		\end{split}
		\end{align}
		By the periodicity of $U$, i.e. $u_0=u_N$, we get
		\begin{align*}
		\begin{split}
		\sum_{j=1}^{N}G_j(U)=&\frac{1}{2}\sum_{j=1}^{N}\left(\frac{u_j-u_{j-1}}{h}+P\right)^2-\sum_{j=1}^{N}\sin(2\pi x_j)\\
		=&\frac{1}{2h^2}\sum_{j=1}^{N}\left(u_j-u_{j-1}\right)^2+\frac{1}{2}P^2N-\sum_{j=1}^{N}\sin(2\pi x_j).
		\end{split}
		\end{align*}
		Thus,
		\begin{align}
		\label{VerifyGiEq3}
		\begin{split}
		&\left(\sum_{j=1}^{N}G_j(U)\right)-\left(\sum_{j=1}^{N}G_j(V)\right)-\left(\nabla\left(\sum_{j=1}^{N}G_j(V)\right)\right)^T\left(U-V\right)\\
		=&\frac{1}{2h^2}\left(\sum_{j=1}^{N}\left(u_j-u_{j-1}\right)^2-\sum_{j=1}^{N}\left(v_j-v_{j-1}\right)^2-\sum_{j=1}^{N}2\left(2v_j-v_{j-1}-v_{j+1}\right)\left(u_j-v_j\right)\right)\\
		=&\frac{1}{2h^2}\left(\sum_{j=1}^{N}\left(u_j-u_{j-1}\right)^2-\sum_{j=1}^{N}\left(v_j-v_{j-1}\right)^2-\sum_{j=1}^{N}2\left(v_j-v_{j-1}\right)\left(u_j-v_j\right)-\sum_{j=1}^{N}2\left(v_j-v_{j+1}\right)\left(u_j-v_j\right)\right).
		\end{split}
		\end{align}
		For the last two terms in \eqref{VerifyGiEq3}, we use the periodicity of $U$ again and get 
		\begin{align*}
		\begin{split}
		&\sum_{j=1}^{N}2\left(v_j-v_{j-1}\right)\left(u_j-v_j\right)+\sum_{j=1}^{N}2\left(v_j-v_{j+1}\right)\left(u_j-v_j\right)\\
		=&\sum_{j=1}^{N}2\left(v_j-v_{j-1}\right)\left(v_{j-1}-v_j+u_j-v_{j-1}\right)+\sum_{j=1}^{N}2\left(v_j-v_{j+1}\right)\left(u_j-v_j\right)\\
		=&-\sum_{j=1}^{N}2\left(v_j-v_{j-1}\right)^2+\sum_{j=1}^{N}2\left(v_j-v_{j-1}\right)\left(u_j-v_{j-1}\right)+\sum_{j=1}^{N}2\left(v_j-v_{j+1}\right)\left(u_j-v_j\right)\\
		=&-\sum_{j=1}^{N}2\left(v_j-v_{j-1}\right)^2+\sum_{j=1}^{N}2\left(v_j-v_{j-1}\right)\left(u_j-v_{j-1}\right)-\sum_{j=1}^{N}2\left(v_{j}-v_{j-1}\right)\left(u_{j-1}-v_{j-1}\right)\\
		=&-\sum_{j=1}^{N}2\left(v_j-v_{j-1}\right)^2+\sum_{j=1}^{N}2\left(v_j-v_{j-1}\right)\left(u_j-u_{j-1}\right).
		\end{split}
		\end{align*}
		Thus, we have
		\begin{align}
		\label{VerifyGiEq5}
		\begin{split}
		&\left(\sum_{j=1}^{N}G_j(U)\right)-\left(\sum_{j=1}^{N}G_j(V)\right)-\left(\nabla\left(\sum_{j=1}^{N}G_j(V)\right)\right)^T\left(U-V\right)\\
		=&\frac{1}{2h^2}\left(\sum_{j=1}^{N}\left(u_j-u_{j-1}\right)^2+\sum_{j=1}^{N}\left(v_j-v_{j-1}\right)^2-2\sum_{j=1}^{N}\left(v_j-v_{j-1}\right)\left(u_j-u_{j-1}\right)\right)\\
		=&\frac{1}{2}\sum_{j=1}^{N}\left(\frac{u_j-u_{j+1}}{h}-\frac{v_j-v_{j-1}}{h}\right)^2.
		\end{split}
		\end{align}
		Therefore, \eqref{VerifyGiEq1}, \eqref{VerifyGiEq5}, and the periodicity of $U$ give \eqref{NumericalPoincare1D}. 
	\end{proof}	
}
}
{To guarantee that $\sum\limits_{j=1}^{N}u_j$ is constant in the Hessian Riemannian flow, we require $G$ to be invariant by translation, as stated next.
	\begin{hyp}
		\label{hypumass}
		For any $0\leq j \leq N$, we have $G_j(U+s)=G_j(U)$, where $U=(u_1,\dots,u_N)\in\mathbb{R}^N, s\in\mathbb{R}$, and  $U+s=\{u_1+s,\dots,u_N+s\}$.
	\end{hyp}
	\begin{remark} 
		By the definition of $\mathcal{L}_U$ and  $G_j(U+s)=G_j(U)$, we know $\mathcal{L}_U I=0$, where $I\in\mathbb{R}^n$ of which all components are $1$. Then, for any $M\in\mathbb{R}^N$, we have
		$$\sum_{j=1}^{N}\left(\mathcal{L}_U^*M\right)_j=\langle \mathcal{L}_U^*M,I\rangle=\langle M, \mathcal{L}_UI\rangle=0.$$ Then, $\sum\limits_{j=1}^{N}u_j(t)$ is invariant for all $t>0$ in \eqref{DynDiscreteHessianFlow}. 
\end{remark}}
{Next, we show that the flow defined by \eqref{DynDiscreteHessianFlow} converges to the solution of  $\widetilde{F}(M,U)=0$. Here, we show the convergence in one dimension. A similar proof holds for higher dimensions. }
\begin{pro}
\label{ConvergDiscreteHessianFlow}
	Let $d=1$. {Let} $(M^*,U^*)$ be the  solution  {of \eqref{DMFG-EH}}, where $M^*=(m_1^*,\dots,m_N^*)$ and $U^*=(u_1^*,\dots,u_N^*)$. Under Assumptions \ref{hypGConvex}-\ref{hypumass}, we have, as $t\rightarrow \infty$,
	$$u_j(t)\rightarrow u_j^*\  \text{and}\  
	m_j(t)\rightarrow m^*_j.$$
\end{pro}
\begin{proof}
	Let
	\begin{equation*}
	\overline{\phi}(t)=\frac{1}{N}\sum_{j=1}^{N}\left(m_j^*\ln\frac{m^*_j}{m_j(t)}\right)+\frac{1}{2N}\sum_{j=1}^{N}\left(u_j(t)-u_j^*\right)^2.
	\end{equation*}
	Because of \eqref{positivityofPhi}, $\overline{\phi}\geq 0$. 
	According to \eqref{MonotonicityPhiBar}, we have
	\begin{align*}
	&\frac{d \overline{\phi}}{dt}+\frac{1}{N}\sum_{j=1}^{N}\left(\left(G_j(U)-G_j(U^*)-\left(\mathcal{L}_{U^*}\left(U-U^*\right)\right)_j\right)m^*_j\right)\\
	&+\frac{1}{N}\sum_{j=1}^{N}\left(\left(G_j(U^*)-G_j(U)-\left(\mathcal{L}_U\left(U^*-U\right)\right)_j\right)m_j\right)\\
	&+\frac{1}{kN}\sum_{j=1}^{N}\left(\left(\ln m_j-\ln m_j^*\right)\left(m_j-m_j^*\right)\right)= 0.
	\end{align*}
	Using Lemma \ref{ConvergenceOfL1} and Assumption \ref{hypGStrictlyConvexity}, we can find a sequence, $\{t_i\}$, such that
	\begin{align}
	\label{DuConvergence}
	\sum_{j=1}^{N}\left(\frac{u_j(t_i)-u_{j+1}(t_i)}{h}-\frac{u_j^*-u_{j+1}^*}{h}\right)^2\rightarrow 0
	\end{align}
	and 
	$$\sum_{j=1}^{N}\left(\ln m_j(t_i)-\ln m_j^*\right)\left(m_j(t_i)-m_j^*\right)\rightarrow 0.$$
	Under Assumption \ref{hypumass}, we have $\sum\limits_{j=1}^{N}\left(u_j(t_i)-u_j^*\right)=0$. Then, combining \eqref{DuConvergence}, and Lemma \ref{quotientrepresentative} below, we conclude that
	$$\lim_{i\rightarrow +\infty}\left(u_j(t_i)-u_j^*\right)^2=0.$$
	Thus, we have
	$$\overline{\phi}(t_i)\rightarrow 0.$$
	Since $\overline{\phi}$ is decreasing, we know
	$$\overline{\phi}(t)\rightarrow 0.$$
	Thus, 
	$$\sum_{j=1}^{N}\left(u_j(t)-u_j^*\right)^2\rightarrow 0$$
	and
	$$\sum_{j=1}^{N}m_j^*\ln \frac{m_j^*}{m_j(t)}\rightarrow 0.$$
	Thus, $$u_j(t)\rightarrow u^*_j$$
	and $$m_j(t)\rightarrow m_j^*.$$
\end{proof}
\begin{lemma}
\label{quotientrepresentative}
Let $\{a_j\}, 0\leq j\leq N,$ be a sequence in $\mathbb{R}^N$ such $\sum\limits_{j=1}^{N}a_j=0$. Assume that $a_{N+1}=a_1$. {Then, there exists a constant $C>0$ such that }
\begin{equation}
\label{quotientrepresentativeequivalence}
\sum_{j=1}^{N}a_j^2\leq C\sum_{j=1}^{N}\left(a_{j+1}-a_j\right)^2.
\end{equation}  
\end{lemma}
\begin{proof}
Assume that $a_{N+1}=a_1$. We consider the linear subspace  $$\mathcal{D}=\left\{a=\left(a_1,\dots,a_N\right)\in \mathbb{R}^N\mid \sum\limits_{j=1}^{N}a_j=0\right\}$$ equipped with the standard $l^2$-norm. Then, $\mathcal{D}$ is isomorphic to the quotient space $\mathbb{R}^N/\mathbb{R}$. We notice that $\mathbb{R}^N/\mathbb{R}$ has another norm given by 
$$|a|_\diamond=\sum_{j=1}^{N}\left(a_{j+1}-a_j\right)^2, a \in \mathbb{R}^N/\mathbb{R}.$$ In addition, because all norms in a finite-dimensional linear space are equivalent, we conclude that \eqref{quotientrepresentativeequivalence} holds. 
\end{proof}
Finally, we record the proof of Theorem \ref{discreteExistenceConvergence}.
\begin{proof}[Proof (of Theorem \ref{discreteExistenceConvergence})]
	The global existence is given by Proposition \ref{GlobalExistence} and the convergence follows from Proposition \ref{ConvergDiscreteHessianFlow}.
\end{proof}
	
\section{Newton's Method for Effective Hamiltonians}
\label{section4}
Here, we explore the  connection between the Hessian Riemannian flow and Newton's method and construct a numerical scheme that, in our numerical tests, improves substantially the speed of the Hessian Riemannian flow. To motivate our method, 
we begin by discretizing  \eqref{DynDiscreteHessianFlow} using the implicit Euler method. Let $(M^j,U^j)$ represent the result of the $j$-th literation and $(M^0,U^0)$ be the initial value. The implicit Euler method computes $\left(M^{j+1},U^{j+1}\right)$ implicitly using the equation
\begin{equation}
\label{BEM}
\begin{bmatrix}
M^{j+1}\\ U^{j+1}
\end{bmatrix}-\begin{bmatrix}
M^{j}\\ U^{j}
\end{bmatrix}=-\xi \overline{F}\begin{bmatrix}
M^{j+1}\\ U^{j+1}
\end{bmatrix},
\end{equation}
where $\xi$ is the step size. Adding
$\frac{\xi}{2}\overline{F}\left(M^{j+1}, U^{j+1}\right)-\frac{\xi}{2}\overline{F}\left(M^j, U^j\right)$ to both sides of \eqref{BEM}, we get
\begin{align*}
&\begin{bmatrix}
M^{j+1}\\ U^{j+1}
\end{bmatrix}+\frac{\xi}{2}\overline{F}\begin{bmatrix}
M^{j+1}\\ U^{j+1}
\end{bmatrix}-\left(\begin{bmatrix}
M^{j}\\ U^{j}
\end{bmatrix}+\frac{\xi}{2}\overline{F}\begin{bmatrix}
M^{j}\\ U^{j}
\end{bmatrix}\right)\\
&=-\xi\left(\frac{1}{2}\overline{F}\begin{bmatrix}
M^{j+1}\\ U^{j+1}
\end{bmatrix}+\frac{1}{2}\overline{F}\begin{bmatrix}
M^{j}\\ U^{j}
\end{bmatrix}\right).
\end{align*}
The prior identity is the Crank-Nicolson scheme with a step size $\xi$ for the following ODE:
\begin{align}
\label{VCBEM}
\frac{d}{dt}\left(\begin{bmatrix}
M\\ U
\end{bmatrix}+\frac{\xi}{2}\overline{F}\begin{bmatrix}
M \\ U
\end{bmatrix}\right)=-\overline{F}\begin{bmatrix}
M \\ U
\end{bmatrix}.
\end{align}
Let {$\nabla \overline{F}\left(M,U\right)$} be the {Jacobian} matrix of $\overline{F}$ at $(M,U)$. Then, \eqref{VCBEM} is equivalent to
{\begin{align}
\label{CNM}
\begin{bmatrix}
\dot{M}\\ \dot{U}
\end{bmatrix}=-\left(I+\frac{\xi}{2}\nabla \overline{F}\begin{bmatrix}
M\\ U
\end{bmatrix}\right)^{-1}\overline{F}\begin{bmatrix}
M \\ U
\end{bmatrix}.
\end{align}}Here, we fix
two non-negative parameters $\kappa$ and $\tau$ and consider the following generalization of
the explicit Euler method for \eqref{CNM}:
\begin{align}
\label{NewtonExplicitDiscretization}
\begin{bmatrix}
M^{j+1}\\ U^{j+1}
\end{bmatrix}=\begin{bmatrix}
M^{j}\\ U^{j}
\end{bmatrix}-\left(\tau I+\kappa \nabla\overline{F}{\begin{bmatrix}
	M^j\\ U^j
	\end{bmatrix}}\right)^{-1}\overline{F}\begin{bmatrix}
M^j\\ U^j
\end{bmatrix}.
\end{align}
When $\kappa=1$ and $\tau=0$, we obtain Newton's method. When $\kappa=0$ and $\tau>0$, 
we get the explicit Euler scheme for \eqref{DynDiscreteHessianRiemannianFlow}. {Here, we do not have a theoretical proof that \eqref{NewtonExplicitDiscretization} preserves the non-negativity and the mass of $M$. However, in the numerical experiments of Section \ref{section6}, $M$ in \eqref{NewtonExplicitDiscretization} is always positive and the mean of $M$ is invariant if \eqref{NewtonExplicitDiscretization} converges. 
 }

\section{Numerical results}
\label{section6}
In this section, we discuss {some} numerical results.  Our algorithms were implemented in Mathematica 10 on MacBook Air (CPU: 1.6 GHz Intel Core i5; Memory: 4 GB 1600 MHz DDR3). For the Hessian Riemannian flow, we use the built-in routine, NDSolve, of Mathematica to solve \eqref{DynDiscreteHessianFlow}. For Newton's method, we use the iteration given in \eqref{NewtonExplicitDiscretization}.

\begin{remark}
To improve the numerical stability, we  solve  \eqref{DynDiscreteHessianFlow} in the following equivalent form. Let $W=\left(w_1,\dots,w_N\right)=\left(\ln m_1,\dots,\ln m_N\right)$ and $e^W=\left(m_1,\dots,m_N\right)$. Denote the initial value by $(M^0,U^0)=\left(m_1^0,\dots,m_N^0,u_1^0, \dots, u_N^0\right)$.  We transform \eqref{DynDiscreteHessianFlow} into 
\begin{align*}
\begin{cases}
\begin{bmatrix}
\dot{W}\\\dot{U}
\end{bmatrix}=-\begin{bmatrix}
\left(-G_1(U)+\frac{\sum\limits_{i=1}^{N}\left(e^{w_i}G_i(U)-\frac{1}{k}e^{w_i}w_i\right)}{\sum \limits_{i=1}^{N}e^{w_i}}+\frac{1}{k}w_1\right)\\
\dots\\
\left(-G_N(U)+\frac{\sum\limits_{i=1}^{N}\left(e^{w_i}G_i(U)-\frac{1}{k}e^{w_i}w_i\right)}{\sum \limits_{i=1}^{N}e^{w_i}}+\frac{1}{k}w_N\right)\\
\left(\mathcal{L}_U^*e^{W}\right)_1\\
\dots\\
\left(\mathcal{L}_U^*e^{W}\right)_N
\end{bmatrix},\\
W(0)=\left(\ln m_1^0,\dots,\ln m_N^0\right), U(0)=U^0.
\end{cases}
\end{align*}
This formulation has the advantage that $M=e^{W}$ is automatically positive. While the Hessian Riemannian flow preserves positivity, small truncation errors sometimes give rise to negative values of $M$. Without the above transformation, this would cause serious numerical difficulties. 
\end{remark}

\subsection{One-dimensional case}
{For $d=1$, let $H$ be as in \eqref{QuadHami}.}
{In this case, the analytic value of the effective Hamiltonian is given in \cite{cacace2016generalized} by 
\begin{align*}
\overline{H}(P)=\begin{cases}
1 & \text{when} \, |P|\leq P_0,\\
c & \text{when} \  |P|>P_0,\  \text{and $c$ is given by}\  |P|=\int_{0}^{1}\left(\sqrt{2(\sin(2\pi s)+c)}\right) ds,
\end{cases}
\end{align*}
where $P_0=\int_{0}^{1}\left(\sqrt{2(\sin(2\pi s)+1)}\right)ds=\frac{4}{\pi}$.
}

{{Then, we approximate $H$ by $G$ given in \eqref{DefGi}.} We also use similar schemes for $H$ in other examples. {We see that $G_i$ satisfies Assumptions \ref{hypGConvex}-\ref{hypumass}}.

In the algorithms, we use the initial value, $(M^0,U^0)=(m_1^0,\dots,m_N^0,u_1^0,\dots,u_N^0)$, where $m_i^0=1+0.9\cos(2\pi x_i)$ and $u_i^0=0.2\cos(2\pi x_i)$. For Newton's method, {as observed in \eqref{NewtonExplicitDiscretization}, there is a balance between a Newton method and a simple explicit descent method. In particular, $I$ can be viewed as a regularizing term at the points where the Jacobian $\nabla \overline{F}$ is nearly singular. Thus, $\tau I+\kappa \nabla\overline{F}$ may not be invertible if $\tau$ is small. Here,} we choose $\tau =\kappa=1$. 

Figure \ref{ExplicitHbarApproximatedHbar} plots the effective Hamiltonians versus their approximated values calculated using the Hessian Riemannian flow (HRF) and  Newton's method (NM). 
Figure \ref{EvolutionOfTheHessianflow} shows the  evolution of  $\overline{\boldsymbol{H}}$, $\boldsymbol{m}$, and $\boldsymbol{u}$ for $P=0.5$ and {$k=10^4$}. In Figure \ref{ExplicitHbarApproximatedHbar}, we see that our method is extremely accurate away from the flat part of the effective Hamiltonian. In the flat part, the Mather measure  corresponding to the different values of $P$
is not strictly positive (see, for example, Figures \ref{2B} and \ref{2e} at the terminal time) and the logarithmic term seems to slow the convergence speed. 
{We also notice in Figure \ref{ExplicitHbarApproximatedHbar} that as $k$ increases, we get more accurate approximations to the flat part of the effective Hamiltonian. Meanwhile, in the experiments, we find that \eqref{NewtonExplicitDiscretization} preserves the mass and the non-negativity of $m$, which also holds for the remaining cases tested below.}
\begin{figure}[h]
	\begin{subfigure}[b]{0.3\textwidth}
		\includegraphics[width=\textwidth]{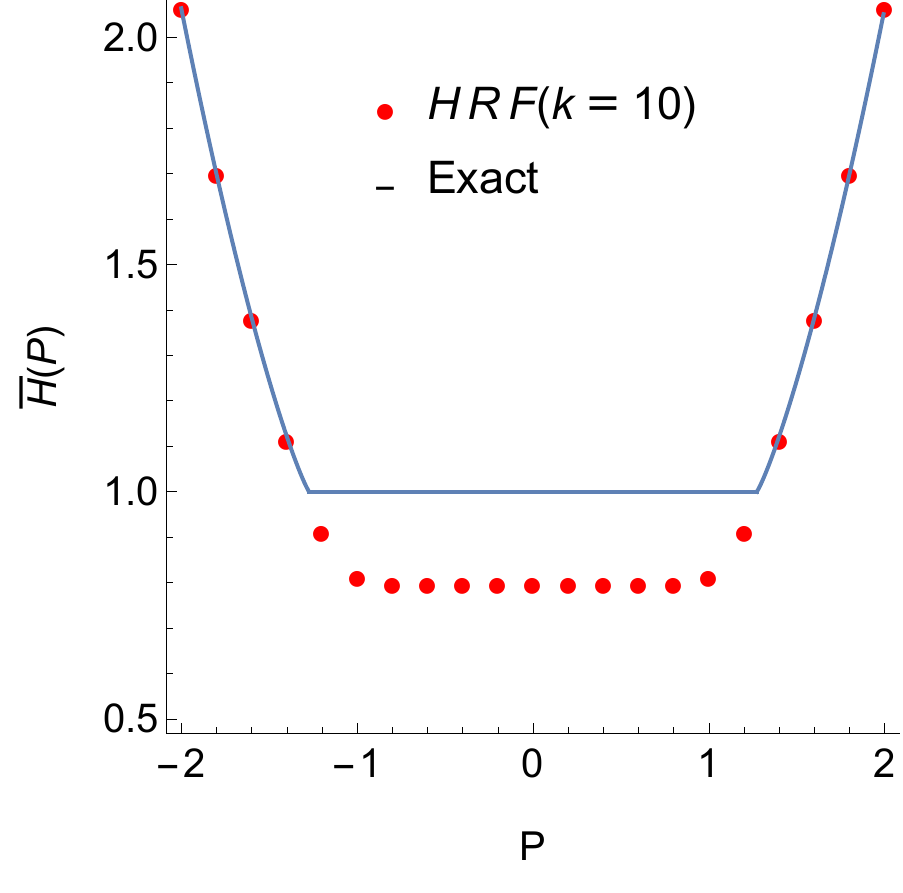}
		\caption{HRF ($k=10$)}
	\end{subfigure}
	\begin{subfigure}[b]{0.3\textwidth}
		\includegraphics[width=\textwidth]{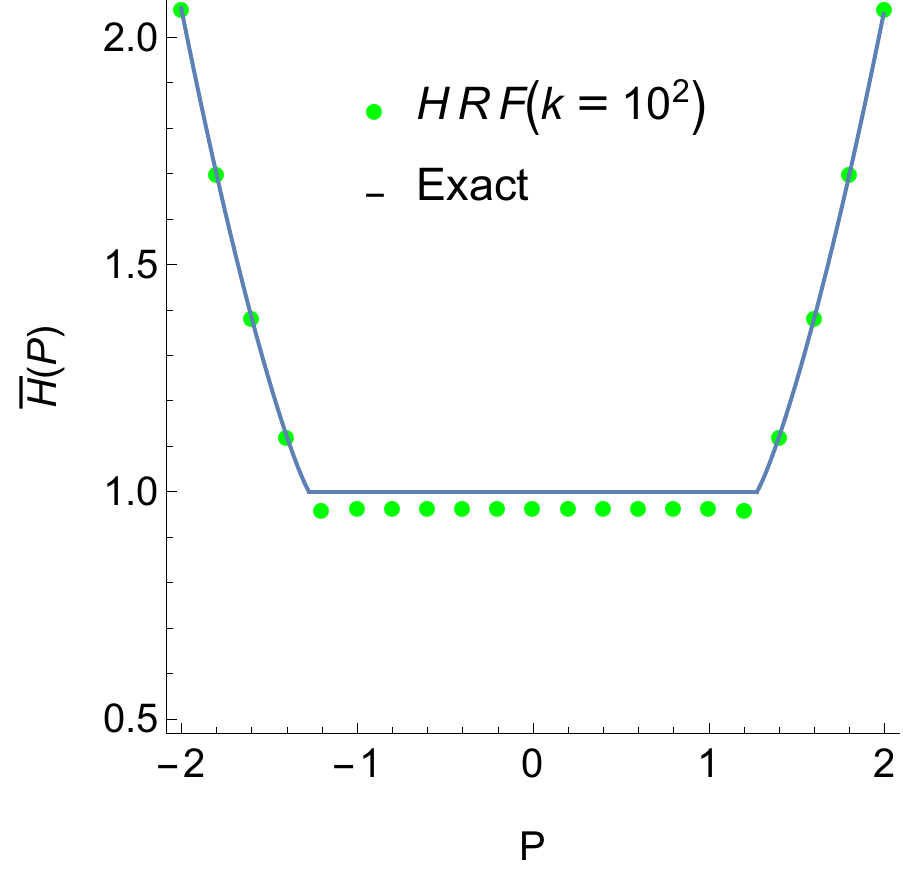}
		\caption{HRF ($k=10^2$)}
	\end{subfigure}
	\begin{subfigure}[b]{0.3\textwidth}
		\includegraphics[width=\textwidth]{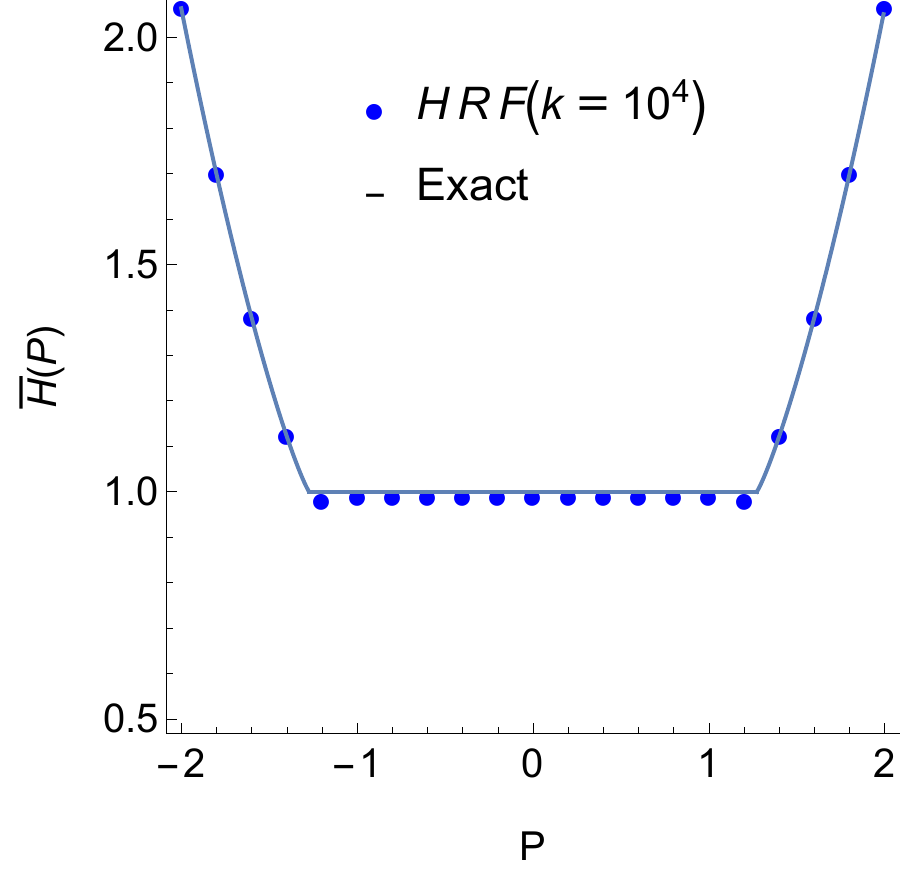}
		\caption{HRF ($k=10^4$)}
	\end{subfigure}
\\
	\begin{subfigure}[b]{0.3\textwidth}
		\includegraphics[width=\textwidth]{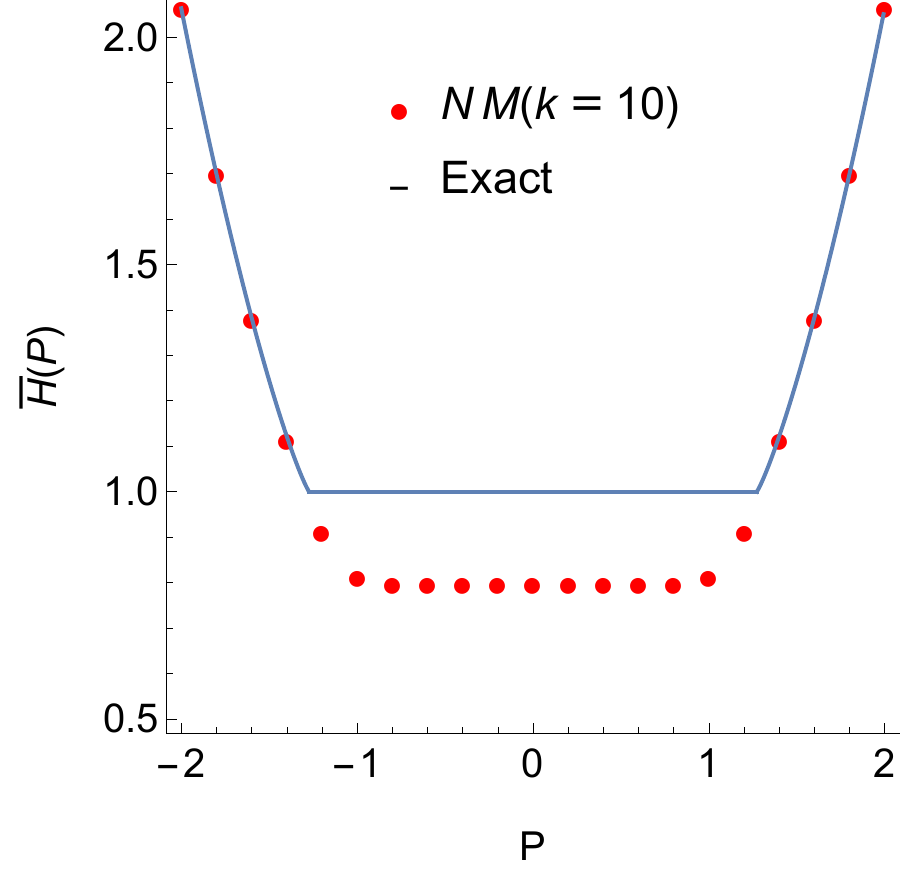}
		\caption{NM ($k=10$)}
	\end{subfigure}
	\begin{subfigure}[b]{0.3\textwidth}
		\includegraphics[width=\textwidth]{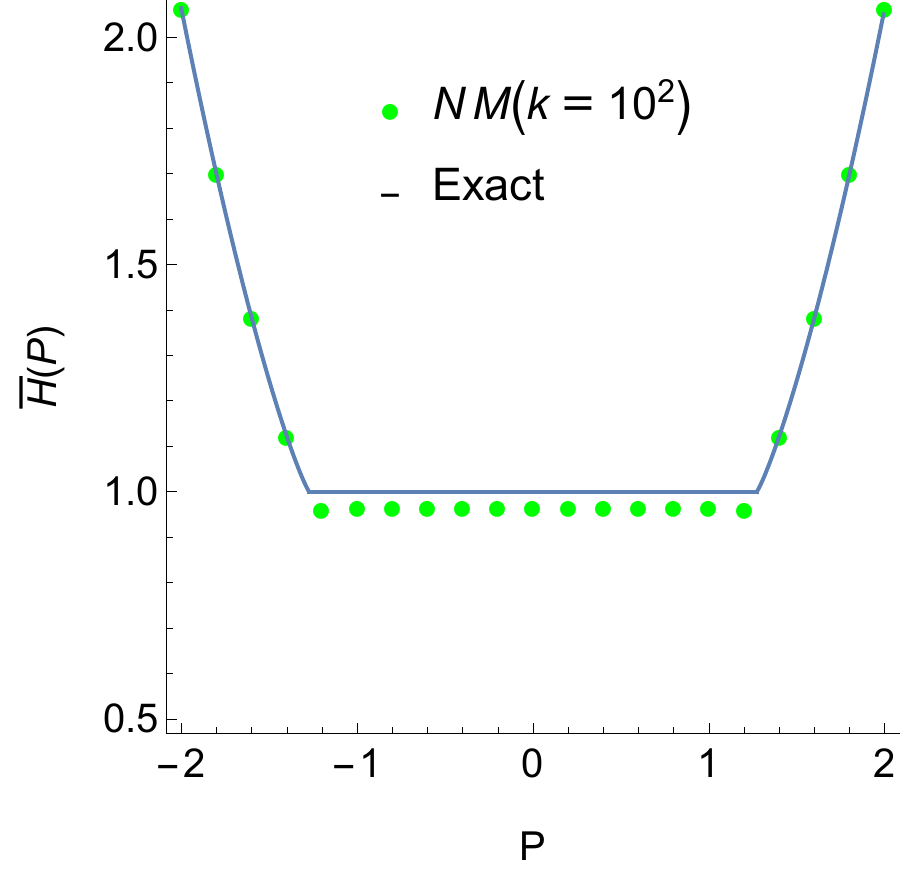}
		\caption{NM ($k=10^2$)}
	\end{subfigure}
	\begin{subfigure}[b]{0.3\textwidth}
		\includegraphics[width=\textwidth]{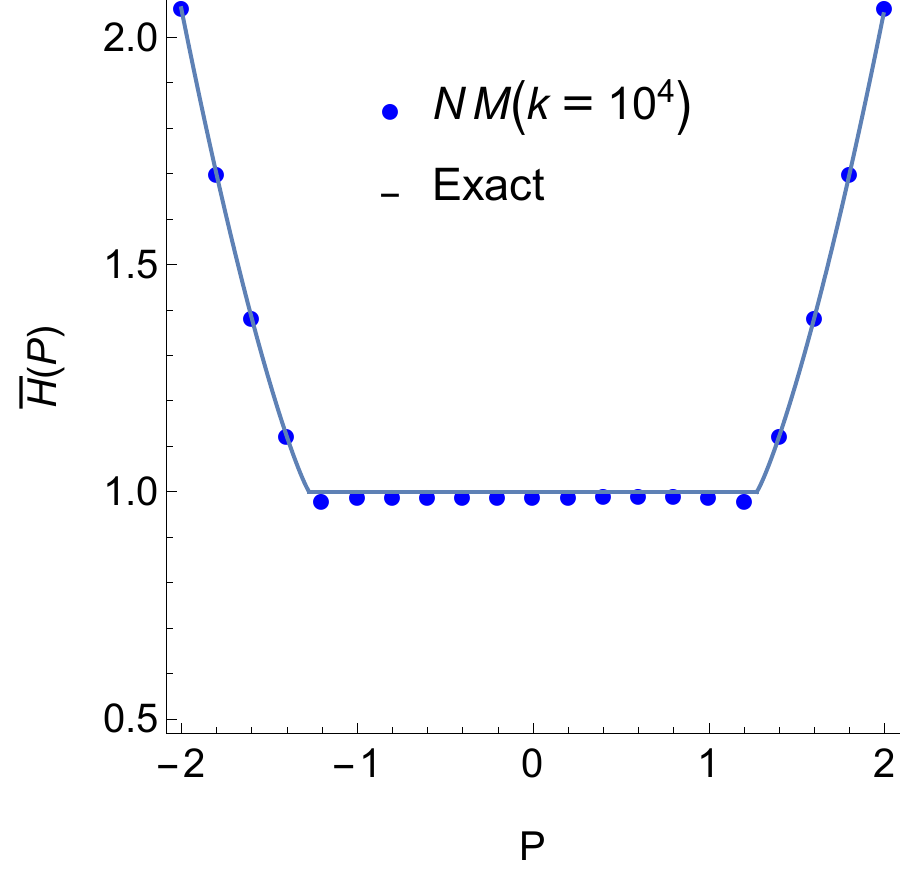}
		\caption{NM ($k=10^4$)}
	\end{subfigure}
	\caption{$\overline{H}$ vs. $\overline{\boldsymbol{H}}$, where $\overline{H}$ is the exact value of the effective Hamiltonian and $\overline{\boldsymbol{H}}$ its approximations given by the
		Hessian Riemannian flow (HRF),
		\eqref{ApproxHessianMonotoneFlow}, and Newton's method (NM), \eqref{NewtonExplicitDiscretization}.}
	\label{ExplicitHbarApproximatedHbar}
\end{figure}

\begin{figure}[h]
	\begin{subfigure}[b]{0.3\textwidth}
		\includegraphics[width=\textwidth]{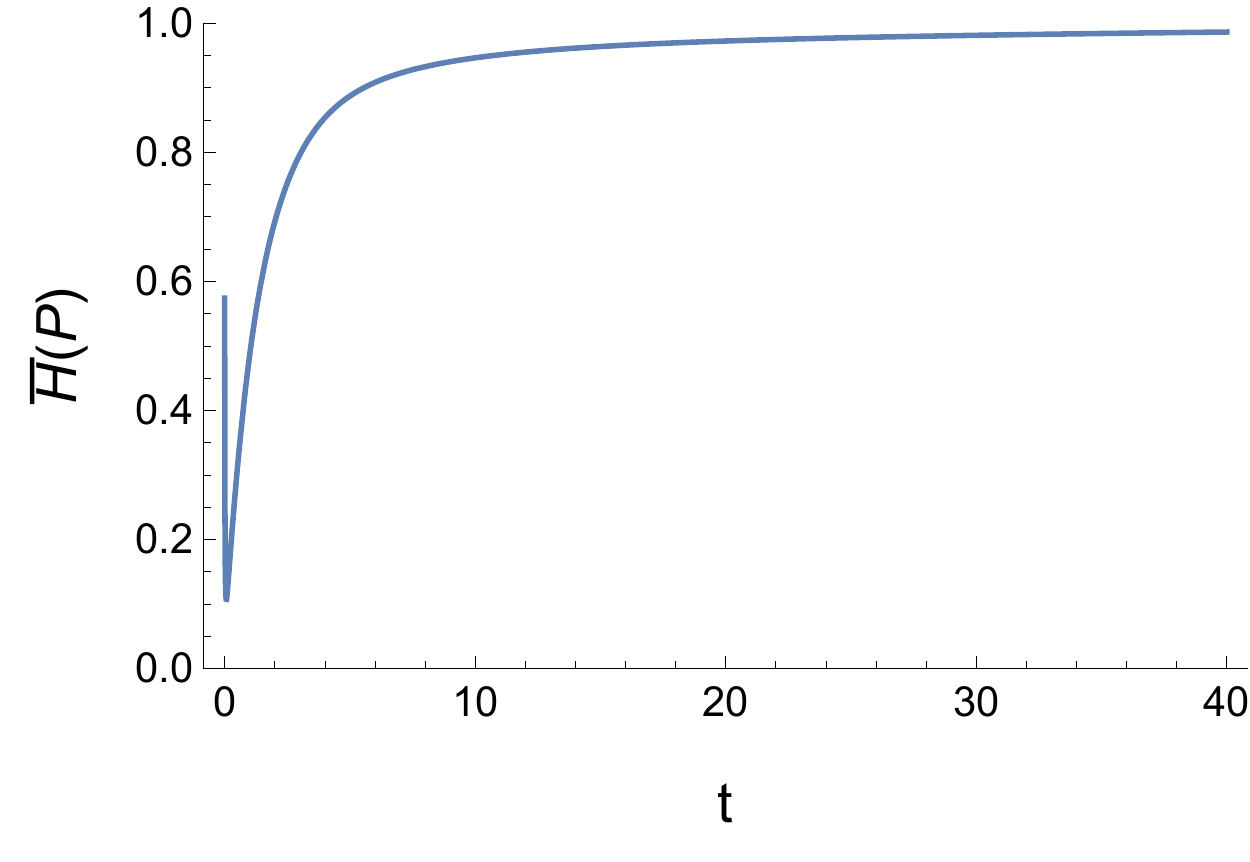}
		\caption{ $\overline{\boldsymbol{H}}(P)$ (HRF)}
	\end{subfigure}
	\begin{subfigure}[b]{0.3\textwidth}
		\includegraphics[width=\textwidth]{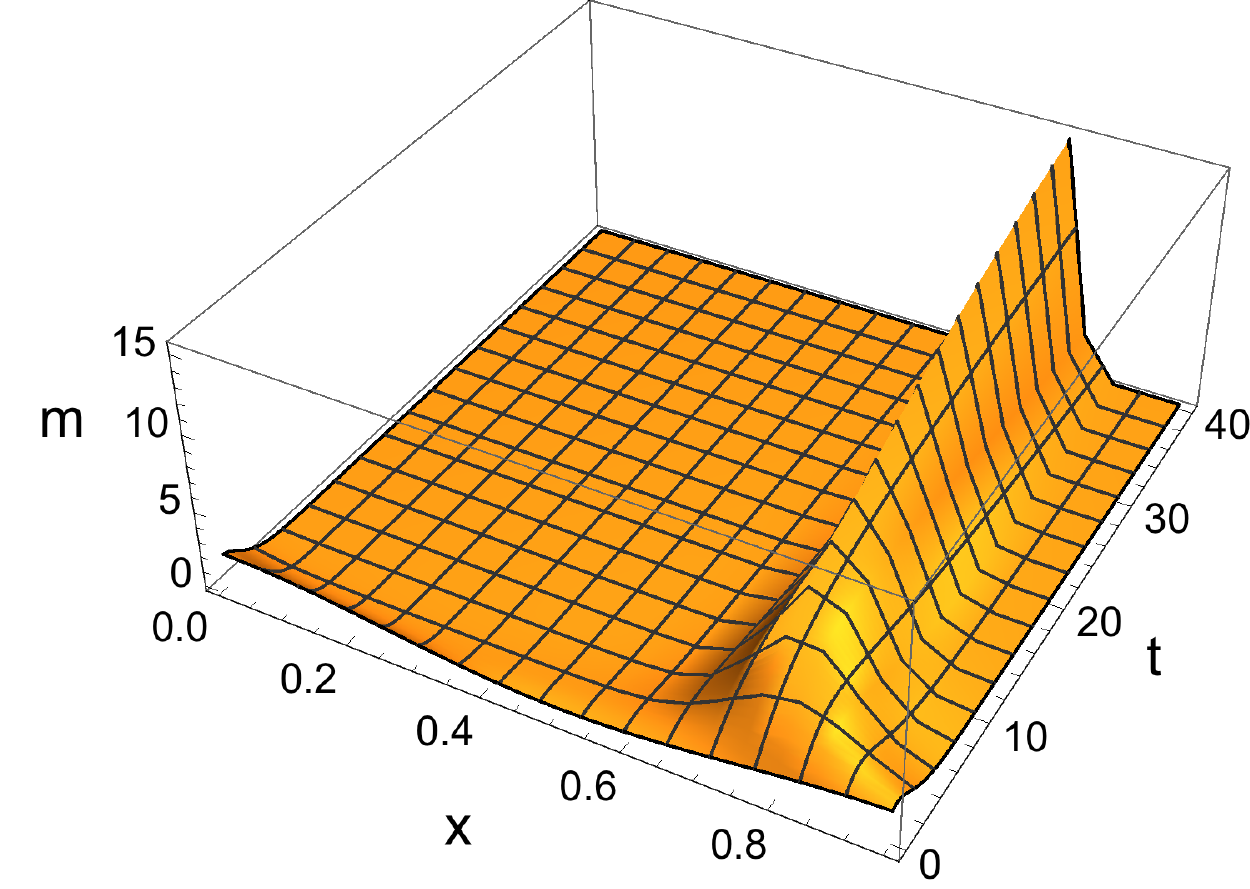}
		\caption{$\boldsymbol{m}$ (HRF)}
		\label{2B}
	\end{subfigure}
	\begin{subfigure}[b]{0.3\textwidth}
		\includegraphics[width=\textwidth]{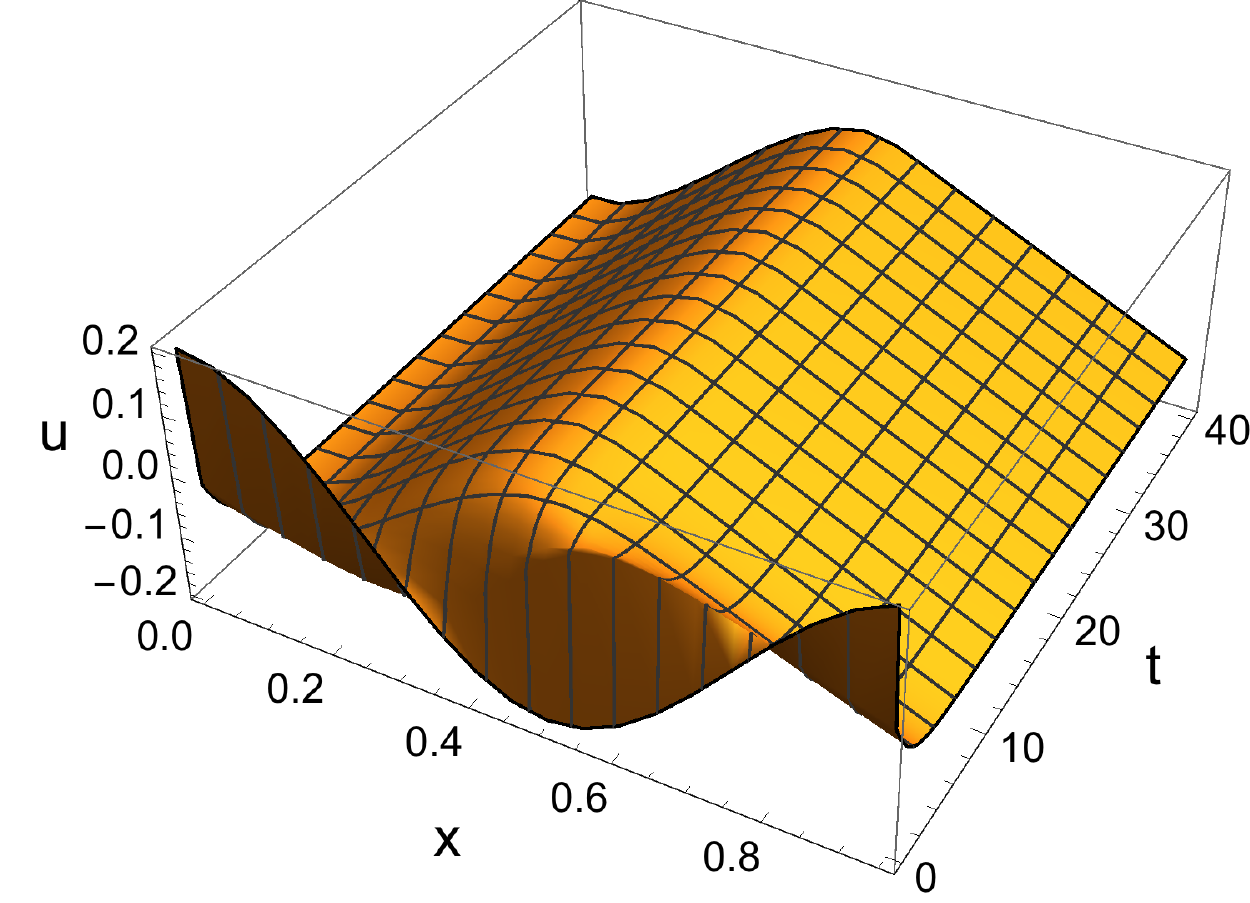}
		\caption{$\boldsymbol{u}$ (HRF)}
	\end{subfigure}
\\
	\begin{subfigure}[b]{0.3\textwidth}
		\includegraphics[width=\textwidth]{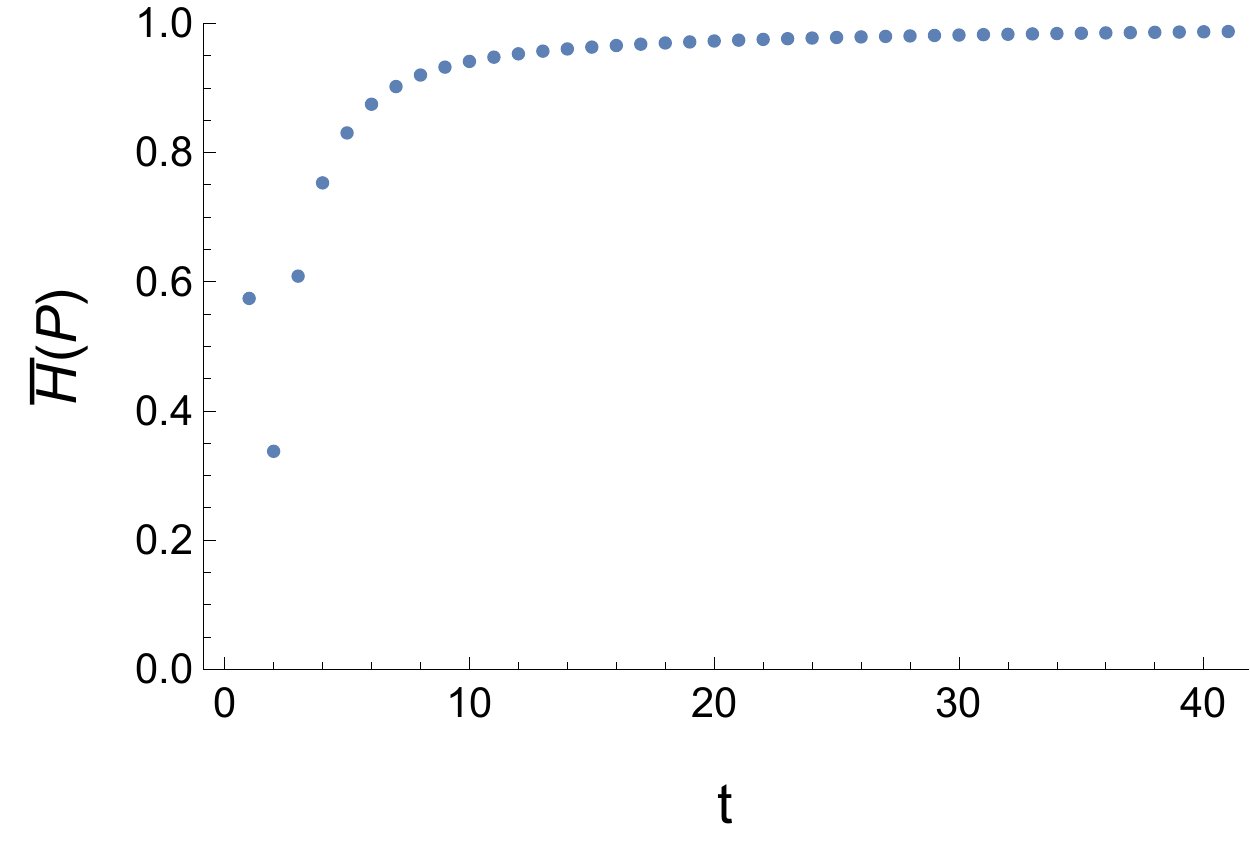}
		\caption{$\overline{\boldsymbol{H}}(P)$ (NM)}
	\end{subfigure}
	\begin{subfigure}[b]{0.3\textwidth}
		\includegraphics[width=\textwidth]{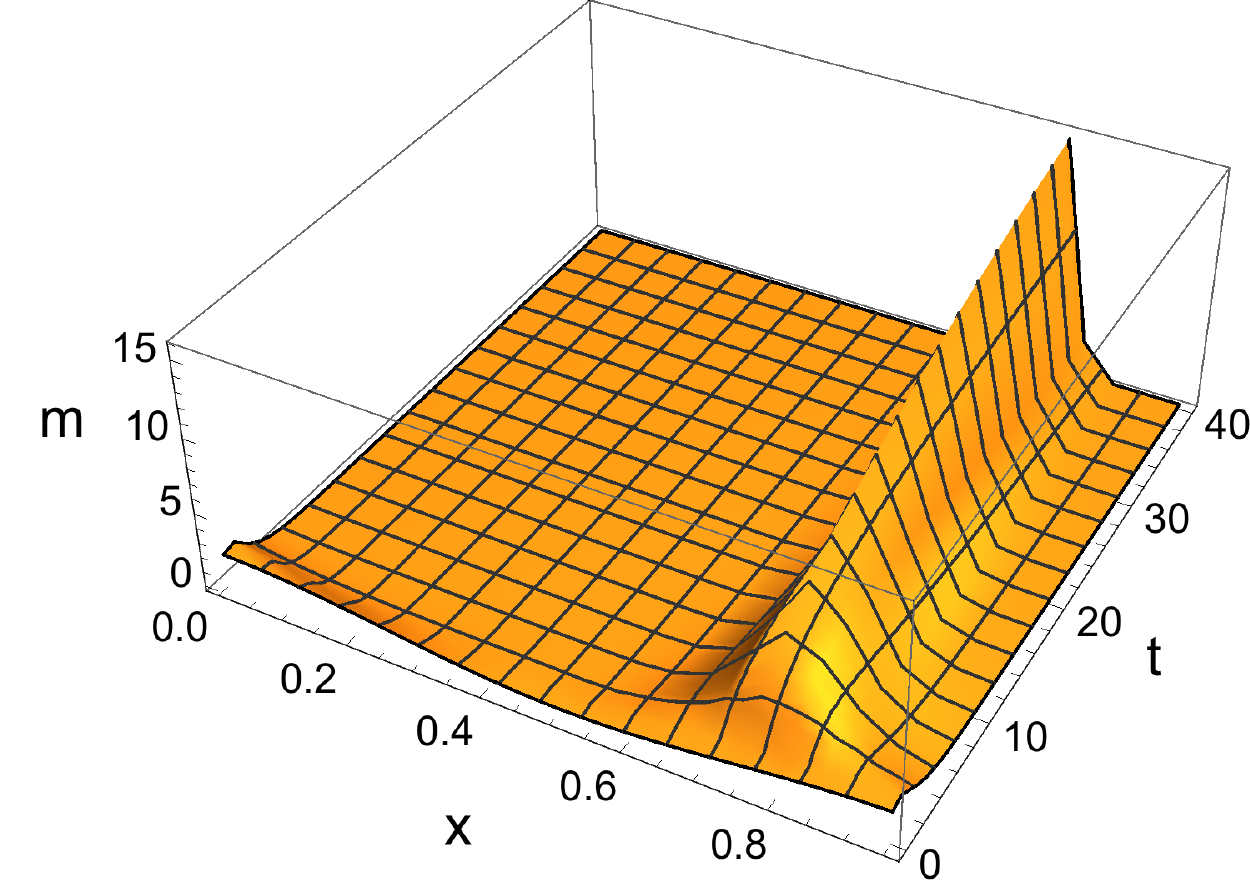}
		\caption{$\boldsymbol{m}$ (NM)}
		\label{2e}
	\end{subfigure}
	\begin{subfigure}[b]{0.3\textwidth}
		\includegraphics[width=\textwidth]{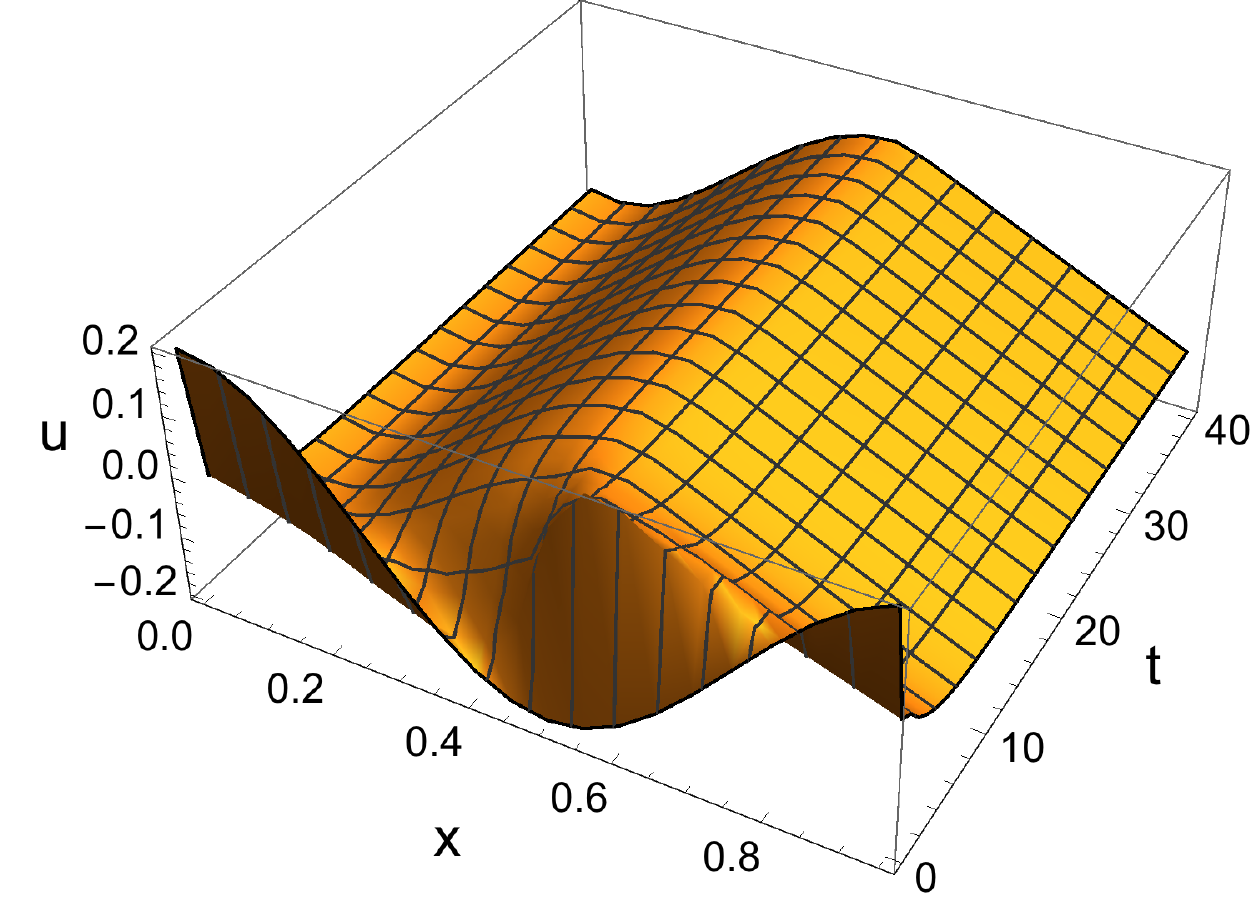}
		\caption{$\boldsymbol{u}$ (NM)}
	\end{subfigure}
	\caption{Numerical solutions of the
		Hessian Riemannian flow (HRF),
		\eqref{ApproxHessianMonotoneFlow}, and Newton's method (NM), \eqref{NewtonExplicitDiscretization}, for {$k=10^4$}.}
	\label{EvolutionOfTheHessianflow}
\end{figure}

To illustrate the convergence of our methods {in time}, we introduce error functions measuring the difference between the numerical result and the exact solution of \eqref{ApproxiMFG}. Let $(\boldsymbol{m}(t), \boldsymbol{u}(t), \overline{\boldsymbol{H}}(t))$ denote either the solution of the Hessian Riemannian flow or Newton's method. Besides, let $(m^*,u^*)$ {be} the  solution of \eqref{ApproxiMFG}  and $\overline{H}^*$ be the corresponding effective Hamiltonian. Inspired by Theorem  \ref{ConvergenceApproxiU}, we define errors:
\begin{align*}
u_{error}(t)=\int_{0}^{1}\left|\boldsymbol{u}(t)-u^*\right|^2dx,\\ 
m_{error}(t)=\int_{0}^{1}\left|\boldsymbol{m}(t)-m^*\right|dx,
\end{align*}
and 
\begin{align*}
\overline{H}_{error}(t)=\left|\overline{\boldsymbol{H}}(t)-\overline{H}^*\right|.
\end{align*}
Here, we use $\boldsymbol{u}(T), \boldsymbol{m}(T), \overline{\boldsymbol{H}}(T)$, where $T$ is the terminal time, to approximate $u^*, m^*, \overline{H}^*$. 
In the simulations, we choose {$k=10^4$, $P=0.5$, and $T=40$}. Figures \ref{HessianflowValidation2}  shows the evolution of the errors for the Hessian Riemannian flow and for  Newton's method. We see that the errors decrease exponentially.   

{Next, we study the convergence of  $u^k$ and $m^k$ as $k\rightarrow +\infty$.  For \eqref{ApproxiMFG} with $H$ given in \eqref{QuadHami}, Theorem 1.2 and Corollary 1.3 in  \cite{GIMY} show that $u^k$ converges and $m^k$ weakly converges if $P\not=P_0$. 	
However, the convergence is not known when $P=P_0$. Here, we examine numerically the convergence of $u^k$ and $m^k$ as $k\rightarrow +\infty$ when $P=P_0=\frac{4}{\pi}$. {We denote by  $(m^*,u^*,\overline{H}^*)$ the solution to \eqref{1dsyspr}.
	 According to Theorem 1.2 and Corollary 1.3 of  \cite{GIMY}, $m^*=\delta_{x=\frac{3}{4}}$, which is a Dirac delta concentrated in $x=\frac{3}{4}$, $u^*_x(x)=\sqrt{2\max\{C(P)+\sin(2\pi x),0\}^2}-P$ for any $x\in [0,1]$, where $C(P)$ is the unique number satisfying $P=\int_{0}^{1}\sqrt{2\max\{C(P)+\sin(2\pi x),0\}^2}dx$, and $\overline{H}^*=1$. Given a fixed $k$, we compute numerically the solution to \eqref{ApproxHessianMonotoneFlow} and denote by  $(\boldsymbol{m}^k(T),\boldsymbol{u}^k(T),\boldsymbol{\overline{H}}^k(T))$ the value of the numerical solution to \eqref{ApproxHessianMonotoneFlow} evaluating at time $T$. Here, we set $T=40$. Since $m^*$ is a Dirac delta, we only consider the errors for $\boldsymbol{u}^k(T)$ and $\boldsymbol{\overline{H}}^k(T)$. We compute and plot in Figure \ref{HessianflowKConv} the errors
	 \begin{align*}
	 u_{err}(k)=\int_0^1\left|\boldsymbol{u}^k(T)-u^*\right|^2dx
	 \end{align*}
	 and
	 \begin{align*}
	 \overline{H}_{err}(k)=\left|\boldsymbol{\overline{H}}^k(T)-\overline{H}^*\right|
	 \end{align*}
	 as $k$ increases and plot  $\boldsymbol{m}^k(T)$ for $T=40$ and $k=10^4$.  We see that $(\boldsymbol{u}^{k}(T), \boldsymbol{\overline{H}}^{k}(T))$ appears to converge as $k$ increases and $\boldsymbol{m}^k(T)$ behaves like $\delta_{x=\frac{3}{4}}$ when $T=40$ and $k=10^4$. 
 }

\begin{figure}
	\begin{subfigure}[b]{0.3\textwidth}
		\includegraphics[width=\textwidth]{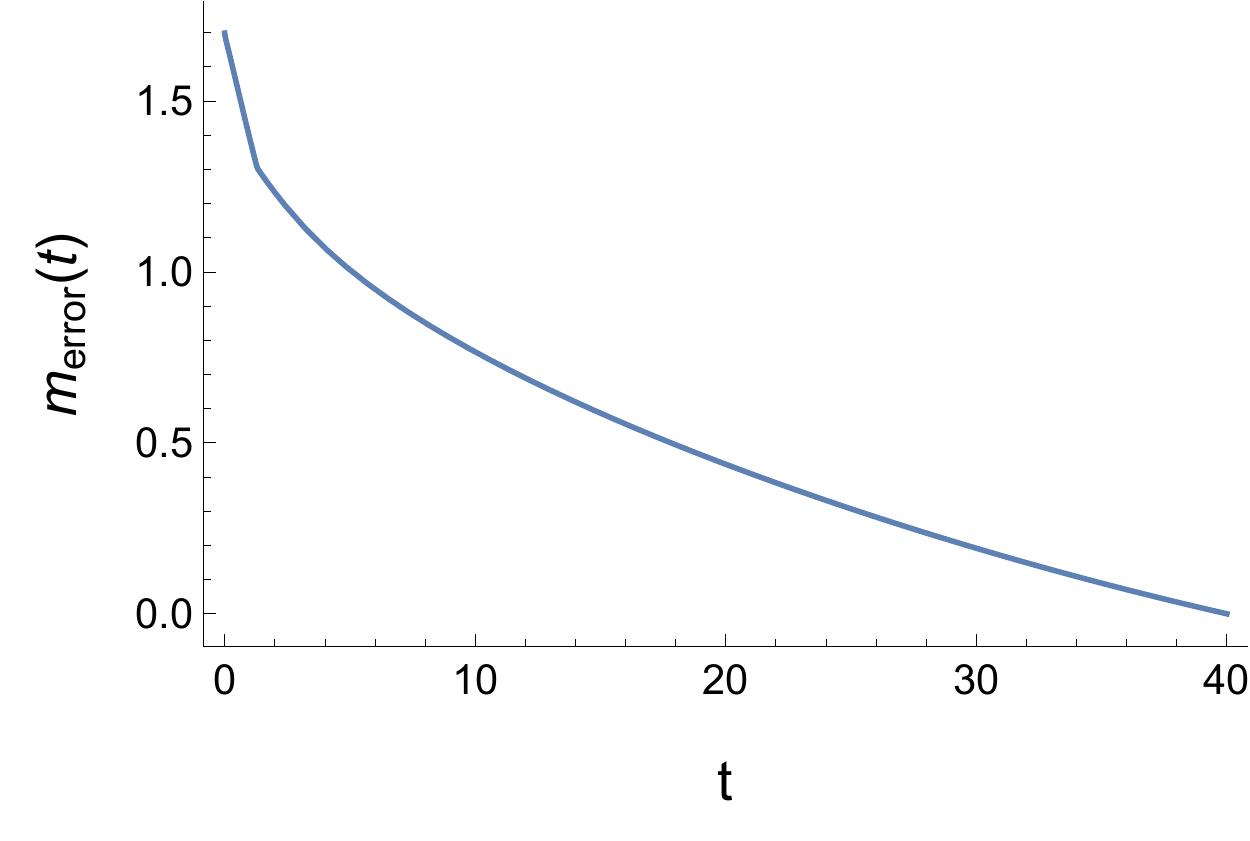}
		\caption{$m_{error}$ (HRF)}
	\end{subfigure}
	\begin{subfigure}[b]{0.3\textwidth}
		\includegraphics[width=\textwidth]{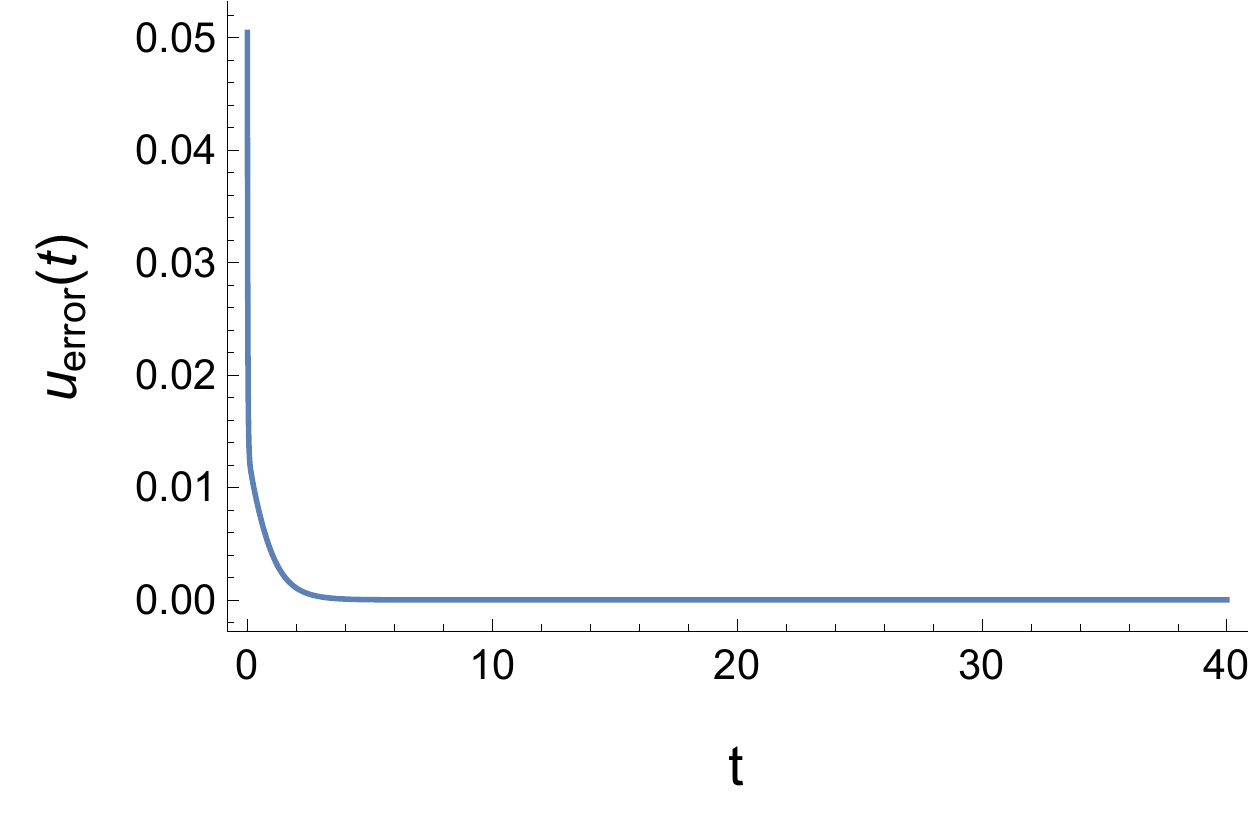}
		\caption{$u_{error}$ (HRF)}
	\end{subfigure}
	\begin{subfigure}[b]{0.3\textwidth}
		\includegraphics[width=\textwidth]{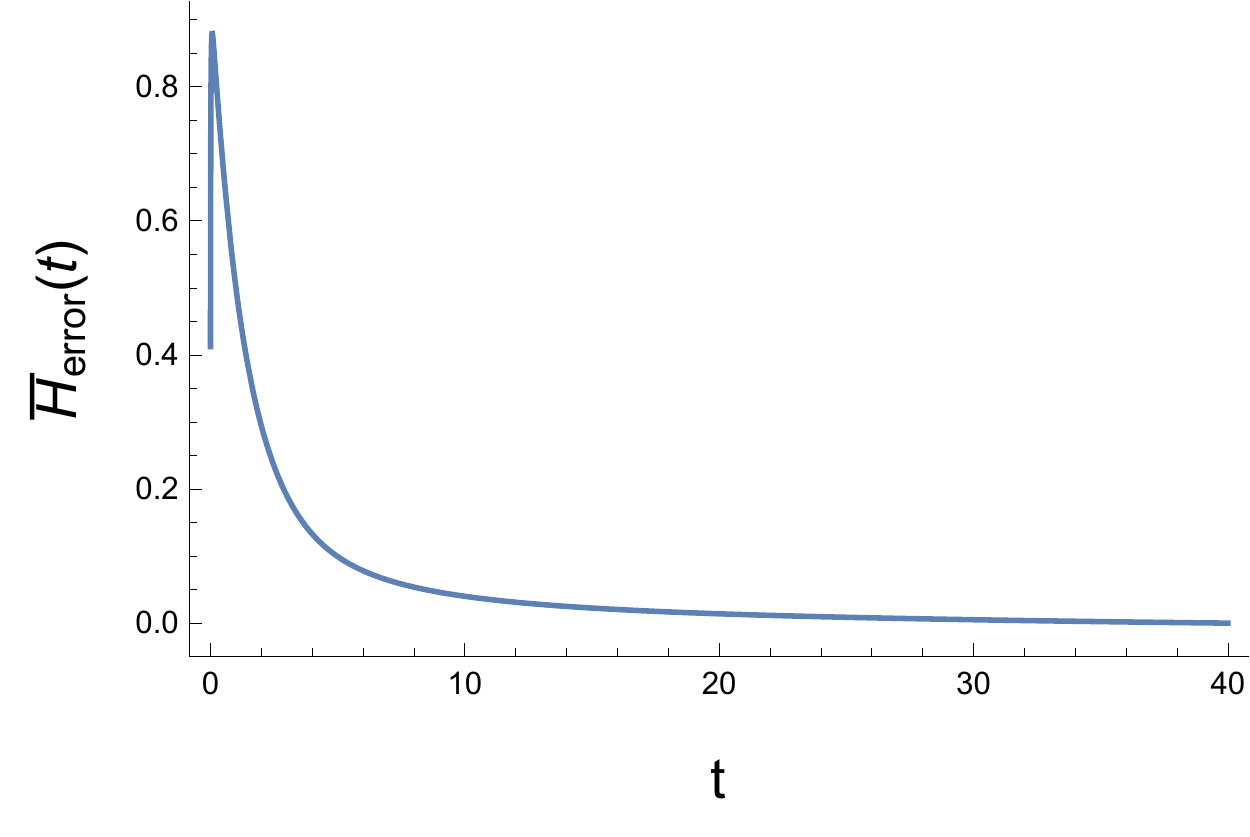}
		\caption{$\overline{H}_{error}$ (HRF)}
	\end{subfigure}
\\
	\begin{subfigure}[b]{0.3\textwidth}
		\includegraphics[width=\textwidth]{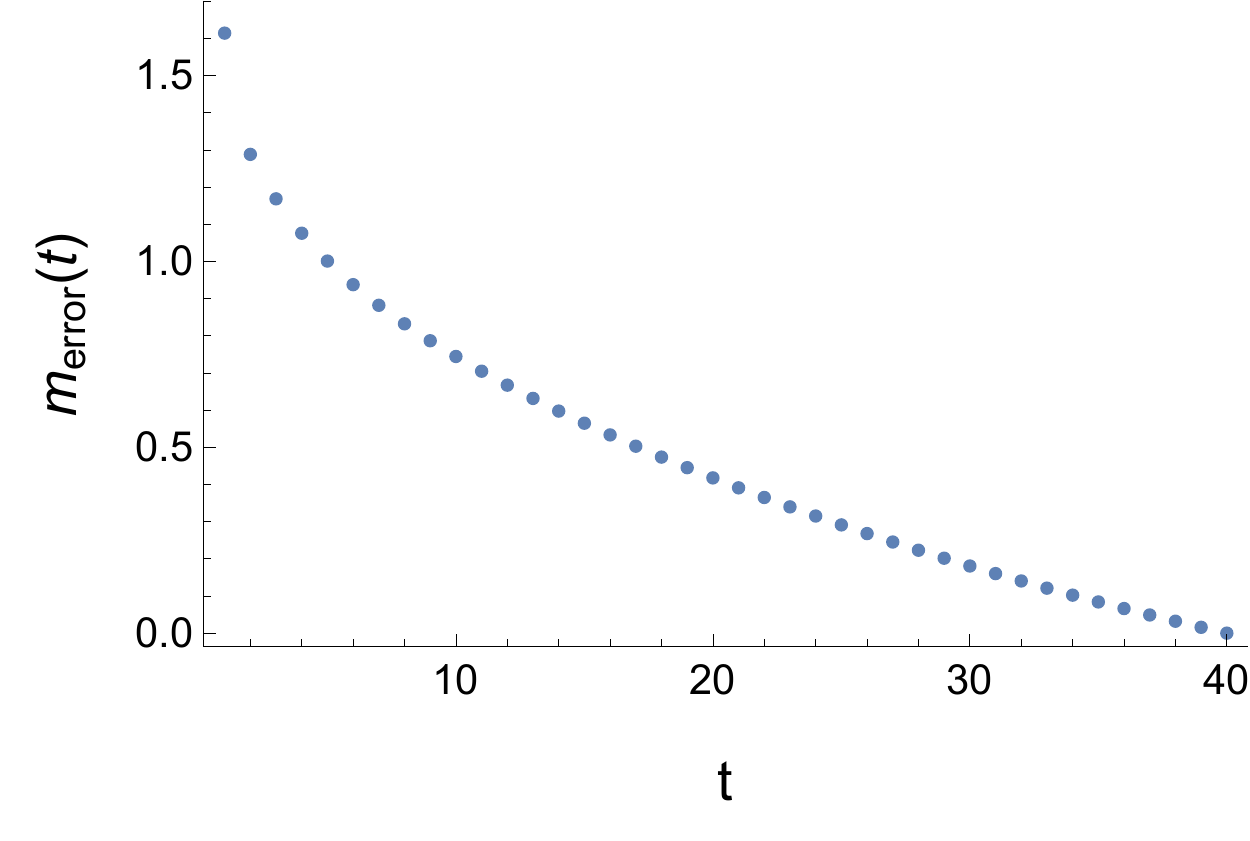}
		\caption{$m_{error}$ (NM)}
	\end{subfigure}
	\begin{subfigure}[b]{0.3\textwidth}
		\includegraphics[width=\textwidth]{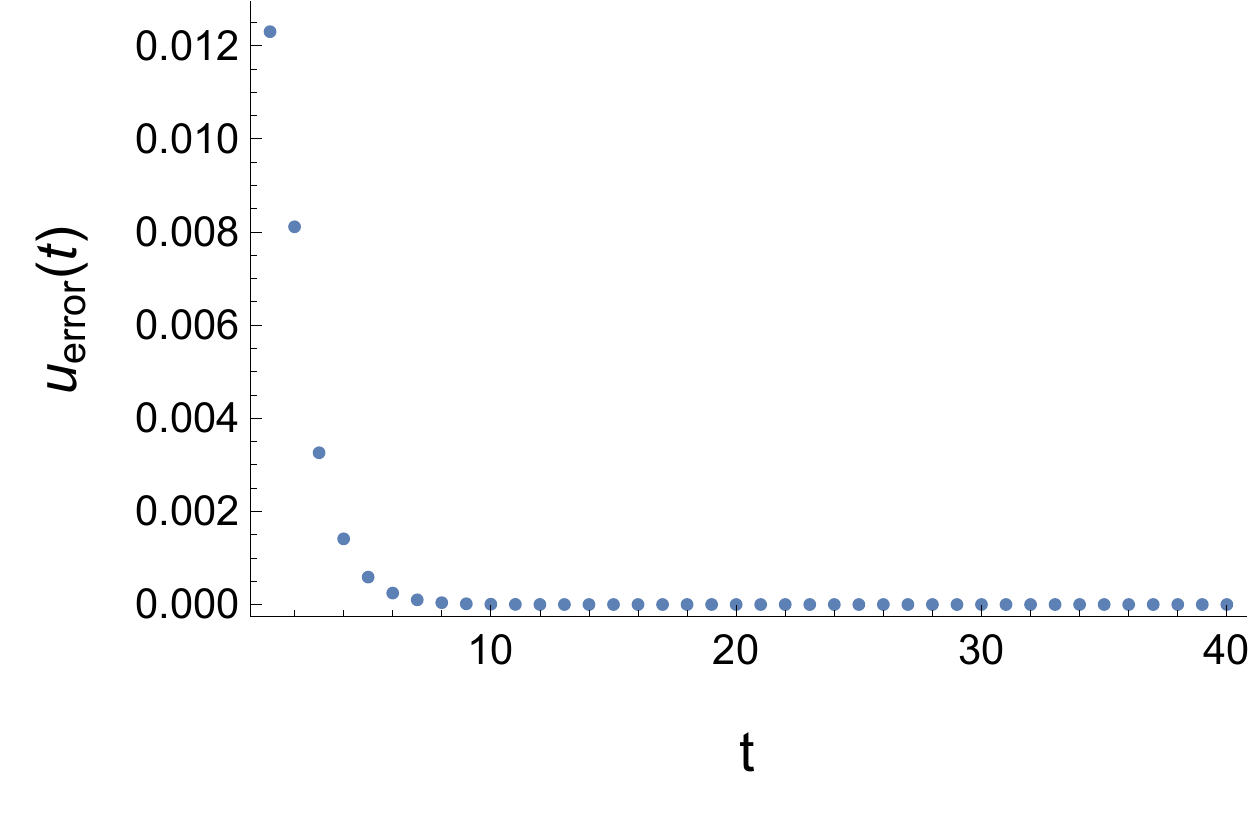}
		\caption{$u_{error}$ (NM)}
		\label{HessianflowValidation2:e}
	\end{subfigure}
	\begin{subfigure}[b]{0.3\textwidth}
		\includegraphics[width=\textwidth]{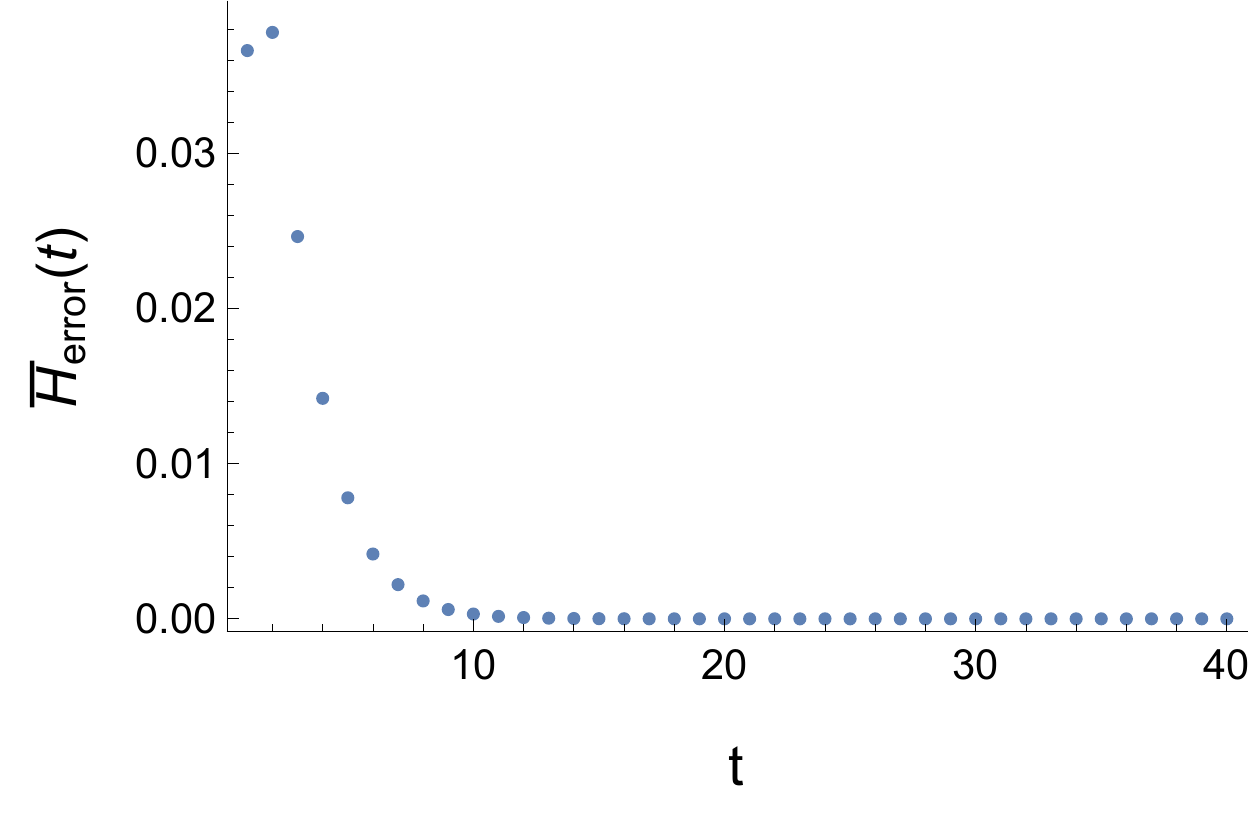}
		\caption{$\overline{H}_{error}$ (NM)}
	\end{subfigure}
	\caption{Evolutions of errors of the
		Hessian Riemannian flow (HRF),
		\eqref{ApproxHessianMonotoneFlow}, and Newton's method (NM), \eqref{NewtonExplicitDiscretization}.}
	\label{HessianflowValidation2}
\end{figure}

\begin{figure}
	\begin{subfigure}[b]{0.3\textwidth}
		\includegraphics[width=\textwidth]{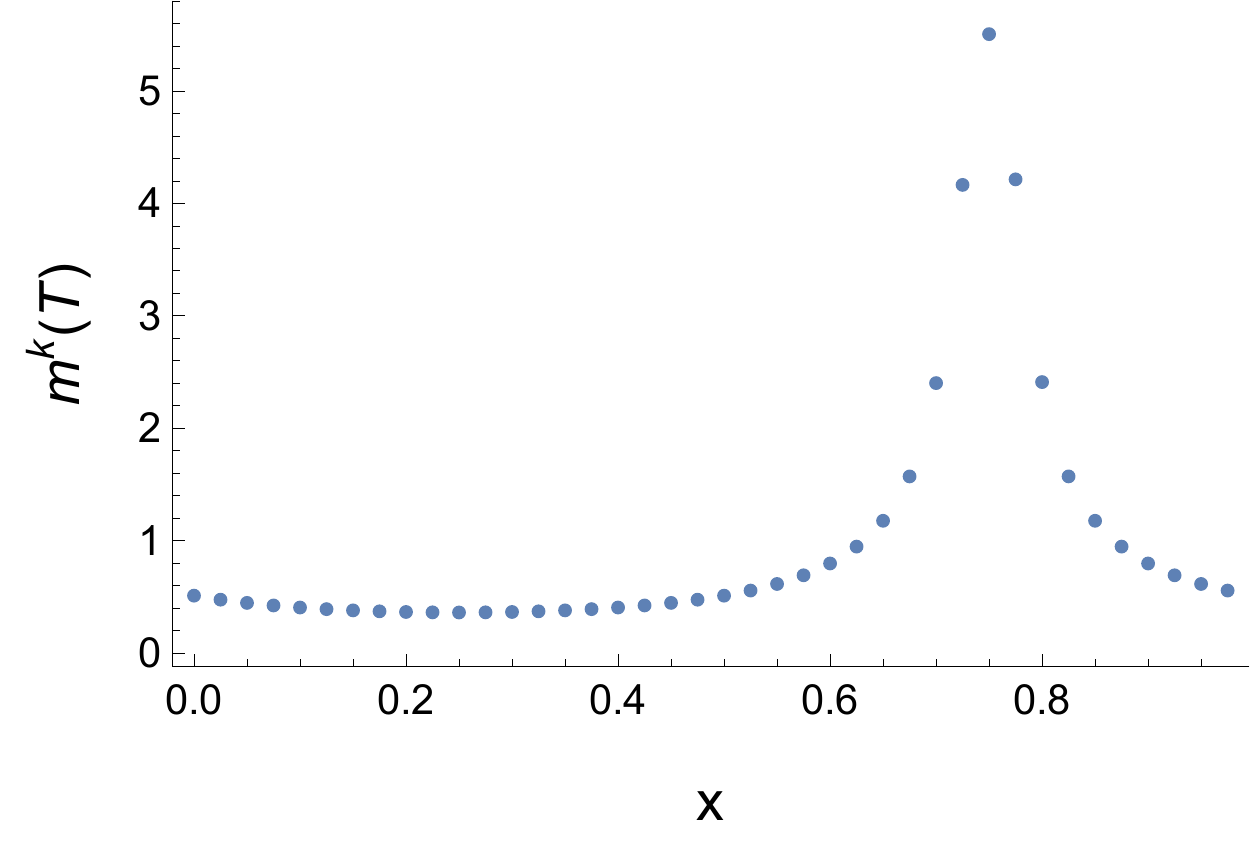}
		\caption{$\boldsymbol{m}^k(T)$ for $T=40$ and $k=10^4$.}
	\end{subfigure}
	\begin{subfigure}[b]{0.3\textwidth}
		\includegraphics[width=\textwidth]{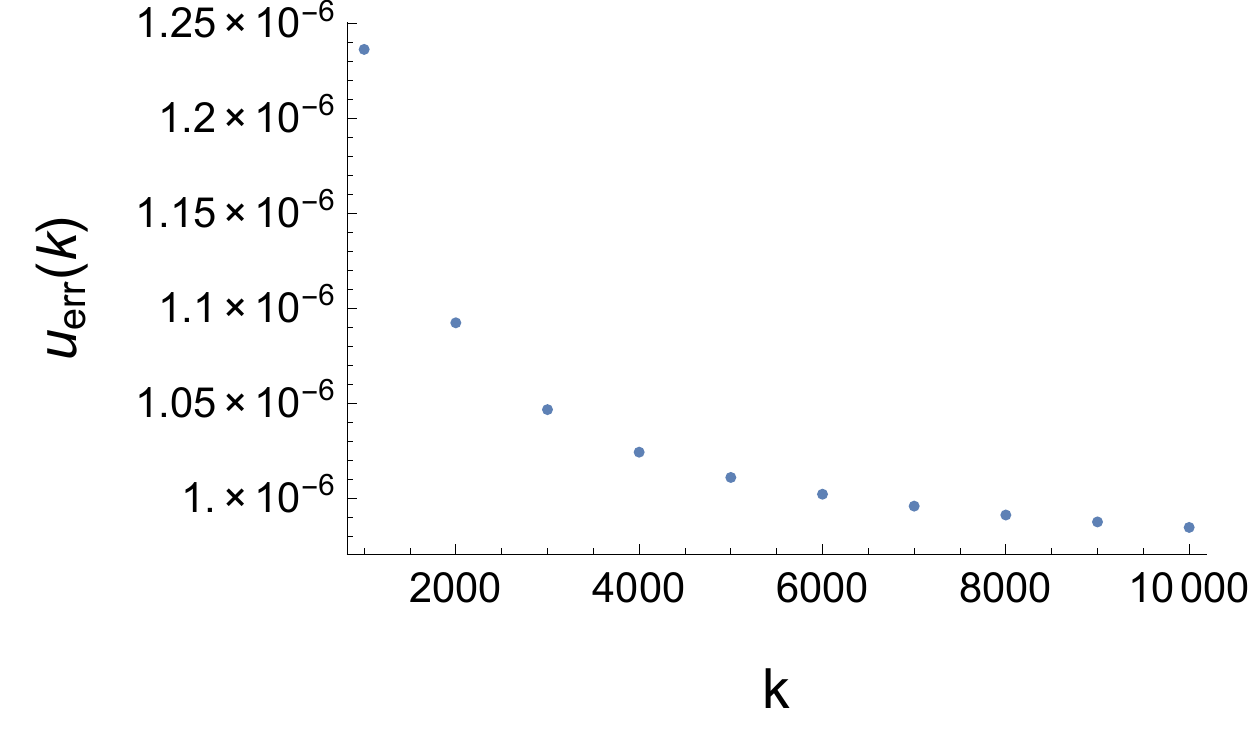}
		\caption{$u_{err}(k)$}
	\end{subfigure}
	\begin{subfigure}[b]{0.3\textwidth}
		\includegraphics[width=\textwidth]{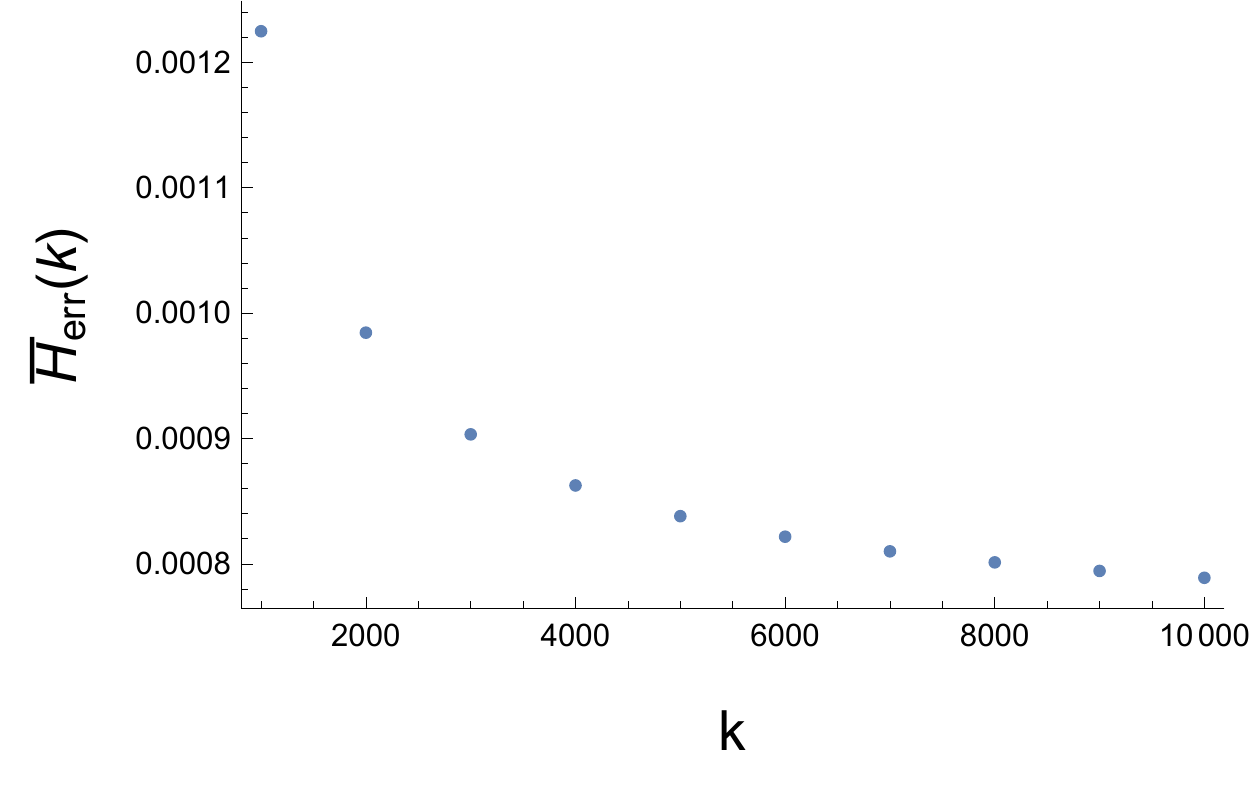}
		\caption{$\overline{H}_{err}(k)$}
	\end{subfigure}
	\caption{The convergence of solutions, approximated by \eqref{ApproxHessianMonotoneFlow}, of \eqref{ApproxiMFG} as $k$ increases when $d=1$ and $P=\frac{4}{\pi}$.}
	\label{HessianflowKConv}
\end{figure}

\subsection{{Two-dimensional case}}
\begin{figure}
	\begin{subfigure}[b]{0.3\textwidth}
		\includegraphics[width=\textwidth]{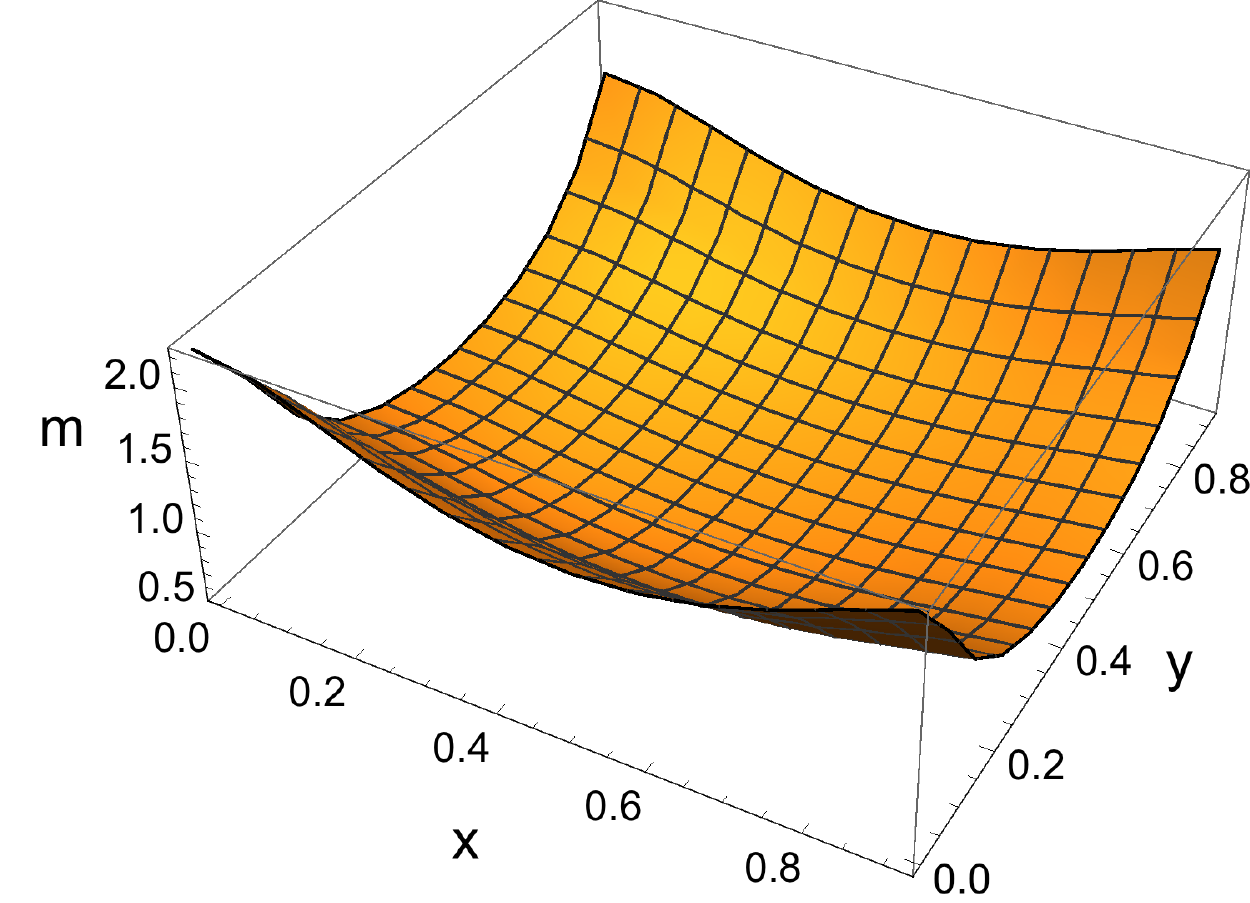}
		\caption{$m$ (HRF)}
	\end{subfigure}
	\begin{subfigure}[b]{0.3\textwidth}
		\includegraphics[width=\textwidth]{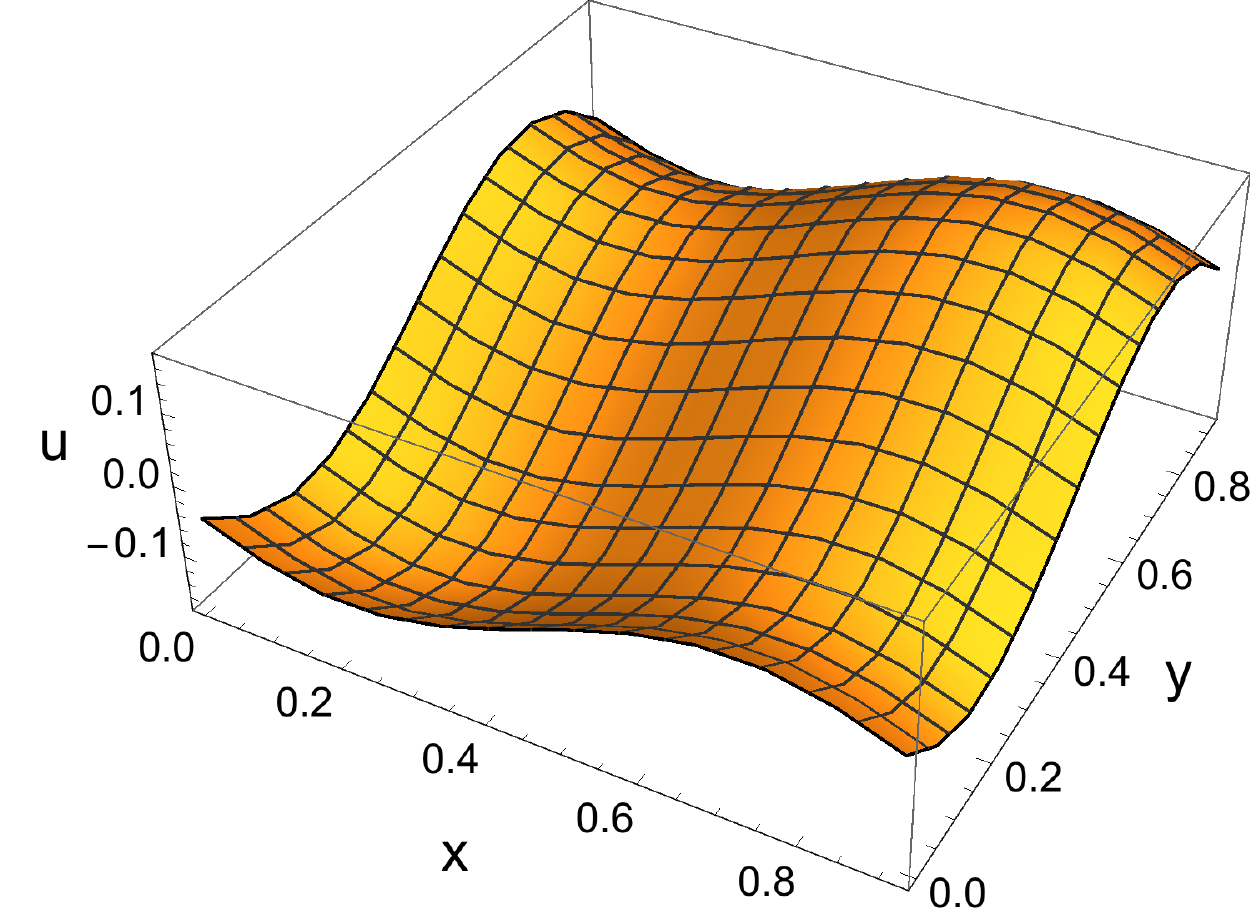}
		\caption{$u$ (HRF)}
	\end{subfigure}
\\
	\begin{subfigure}[b]{0.3\textwidth}
		\includegraphics[width=\textwidth]{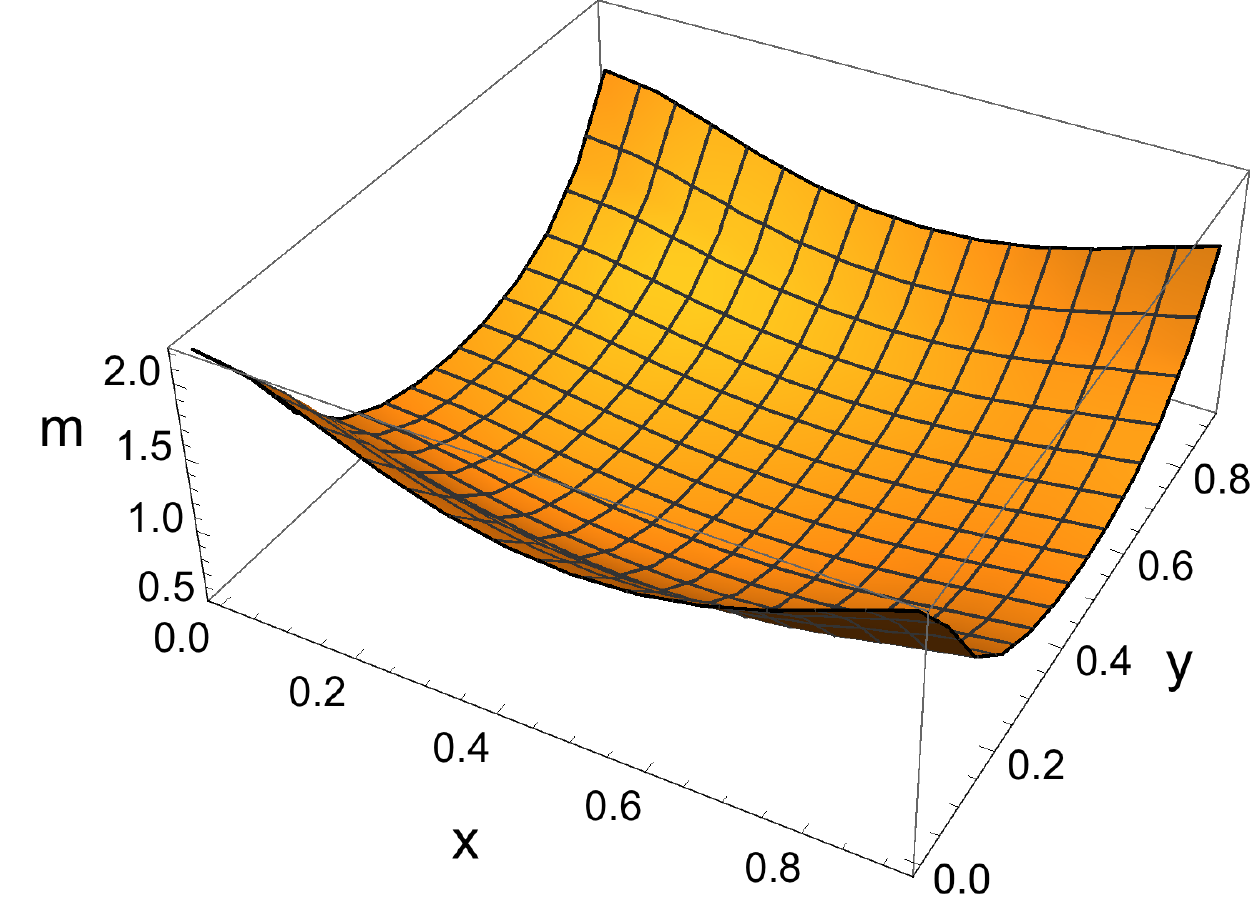}
		\caption{$m$ (NM)}
	\end{subfigure}
	\begin{subfigure}[b]{0.3\textwidth}
		\includegraphics[width=\textwidth]{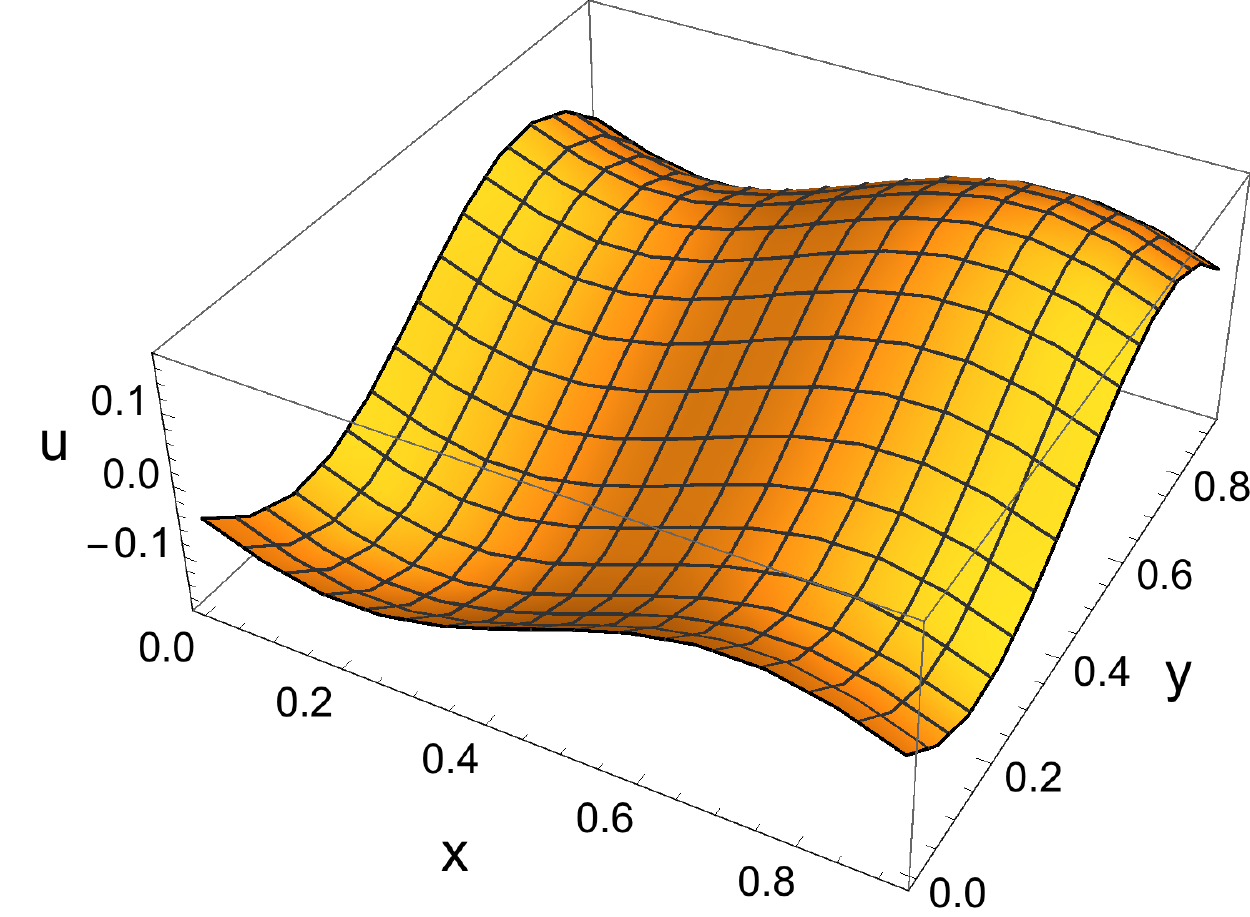}
		\caption{$u$ (NM)}
	\end{subfigure}
	\caption{Numerical solutions of
		Hessian Riemannian flow (HRF),
		\eqref{ApproxHessianMonotoneFlow}, and Newton's method (NM), \eqref{NewtonExplicitDiscretization}, at $t=14$ when $d=2$, $P=(1.5,2.5)$, and $H(x,p)=\frac{\left|p_1\right|^2}{2}+\frac{\left|p_2\right|^2}{2}+\cos\left(2\pi x_1\right)+\cos\left(2\pi x_2\right)$.}
	\label{TwoDimM}
\end{figure}
{
In higher dimensions, neither \cite{evans2003some} nor \cite{GIMY} give the convergence of measures in \eqref{ApproxiMFG} as $k\rightarrow+\infty$. Here, we study numerically the convergence in two dimensions. 
}
Let $d=2$,  $p=(p_1,p_1)\in \mathbb{R}^2$, and $x=(x_1, x_2)\in \mathbb{T}^2$. We consider the two-dimensional Hamiltonian discussed in \cite{gomes2004computing}: 
\begin{align}
\label{2DHmtn}
H(x,p)=\frac{\left|p_1\right|^2}{2}+\frac{\left|p_2\right|^2}{2}+\cos\left(2\pi x_1\right)+\cos\left(2\pi x_2\right).
\end{align}
{Let $P=(P_1,P_2)\in\mathbb{R}^2$. 
In this case, as pointed out in \cite{gomes2004computing}, 
\begin{equation*}
\overline{{H}}(P_1,P_2)=\overline{H}(P_1)+\overline{H}(P_2),
\end{equation*}
where $\overline{H}(P_1)$ and $\overline{H}(P_2)$ are one-dimensional effective Hamiltonians related to  
\begin{equation*}
{H}(x,p)=\frac{p^2}{2}+\cos(2\pi x). 
\end{equation*}
According to the discussion in the one-dimensional case, the critical values of $P$ satisfy $|P_1|=\frac{4}{\pi}$ and $|P_2|=\frac{4}{\pi}$. For $|P_1|,|P_2|>\frac{4}{\pi}$,} we choose $P=(1.5,2.5)$ for which $\overline{H}$=$4.4099660$ according to \cite{gomes2004computing}. For Newton's method, we set $\tau=2$ and $\kappa=1$. Fixing $N=144$, Table     \ref{TwoDimensions} shows computed $\overline{\boldsymbol{H}}(P)$ at $t=14$ for different values of $k$. We see that when $k=10^4$, we get a very accurate approximation for $\overline{H}$. Figure \ref{TwoDimM} plots $\boldsymbol{m}$ and $\boldsymbol{u}$ at $t=14$ when $k=10^4$.

{Similar to the corresponding one-dimensional case, one may wonder whether the convergence holds when $|P_1|\leq \frac{4}{\pi}$ and $|P_2|\leq \frac{4}{\pi}$.  Thus, we choose $P_1=1.273<\frac{4}{\pi}$ and $P_2=-1.2>-\frac{4}{\pi}$. {Here, we do not have an explicit solution to \eqref{HJFPTogether}.  To study the convergence, we first set $k'\in \mathbb{N}$ very large and compute $(\boldsymbol{m}^{k'}(T),\boldsymbol{u}^{k'}(T),\boldsymbol{\overline{H}}^{k'}(T))$. Then, we compute and plot in Figure  \ref{HessianflowKConv2D} the errors,
		\begin{align}
		\label{uerrordef}
		u_{error}(k)=\int_0^1\left|\boldsymbol{u}^k(T)-\boldsymbol{u}^{k'}(T)\right|^2dx,\\
		\label{merrordef}
		m_{error}(k)=\int_0^1\left|\boldsymbol{m}^k(T)-\boldsymbol{m}^{k'}(T)\right|dx,
		\end{align}
		and
		\begin{align}
		\label{Herrordef}
		\overline{H}_{error}(k)=\left|\boldsymbol{\overline{H}}^k(T)-\boldsymbol{\overline{H}}^{k'}(T)\right|,
		\end{align}
		as $k$ increases.} We see from Figure \ref{HessianflowKConv2D} that the approximated solutions appear to converge as $k$ increases.  
}

\begin{figure}
	\begin{subfigure}[b]{0.3\textwidth}
		\includegraphics[width=\textwidth]{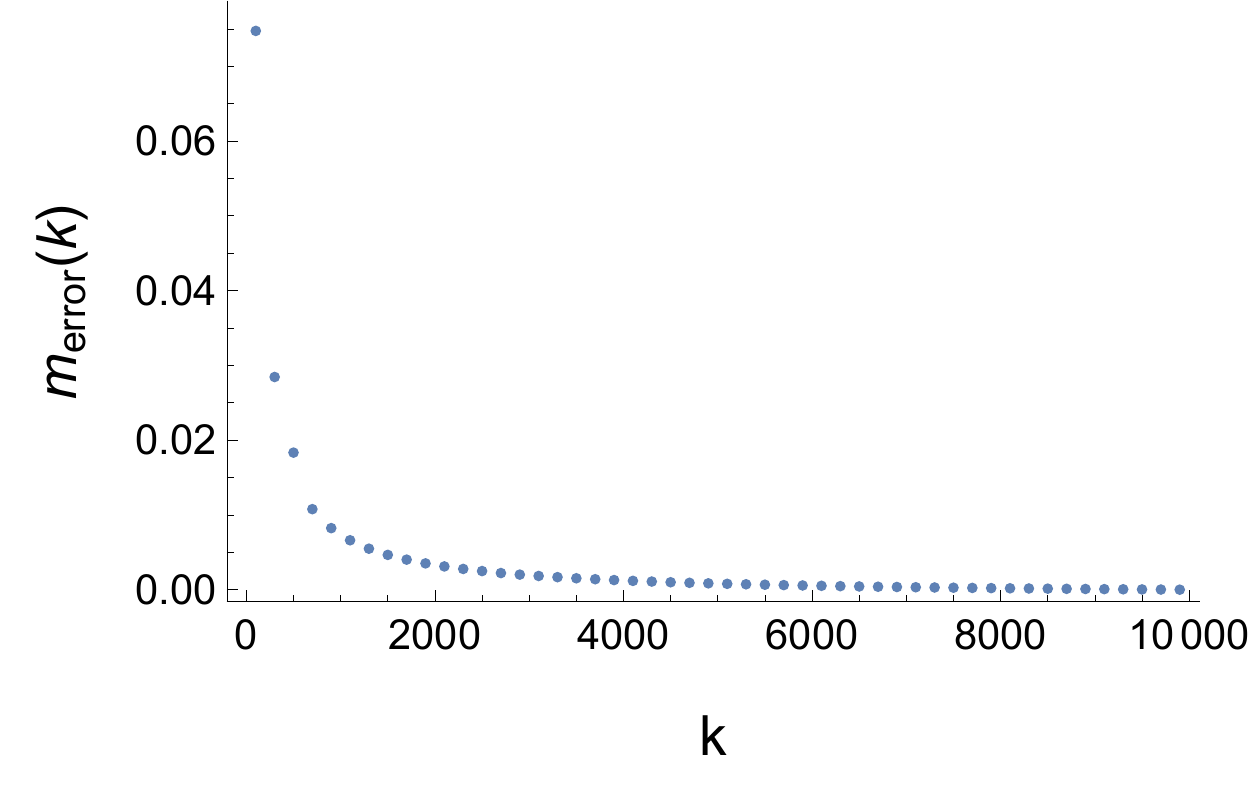}
		\caption{$m_{error}(k)$}
	\end{subfigure}
	\begin{subfigure}[b]{0.3\textwidth}
		\includegraphics[width=\textwidth]{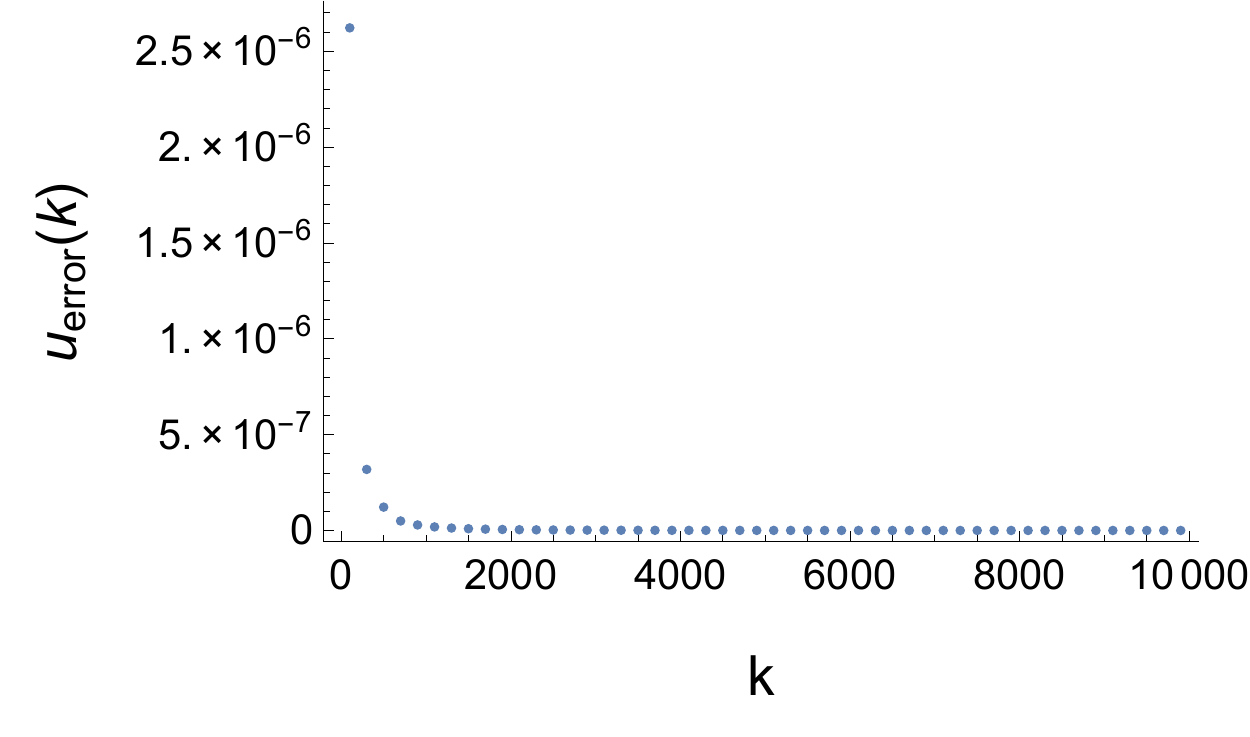}
		\caption{$u_{error}(k)$}
	\end{subfigure}
	\begin{subfigure}[b]{0.3\textwidth}
		\includegraphics[width=\textwidth]{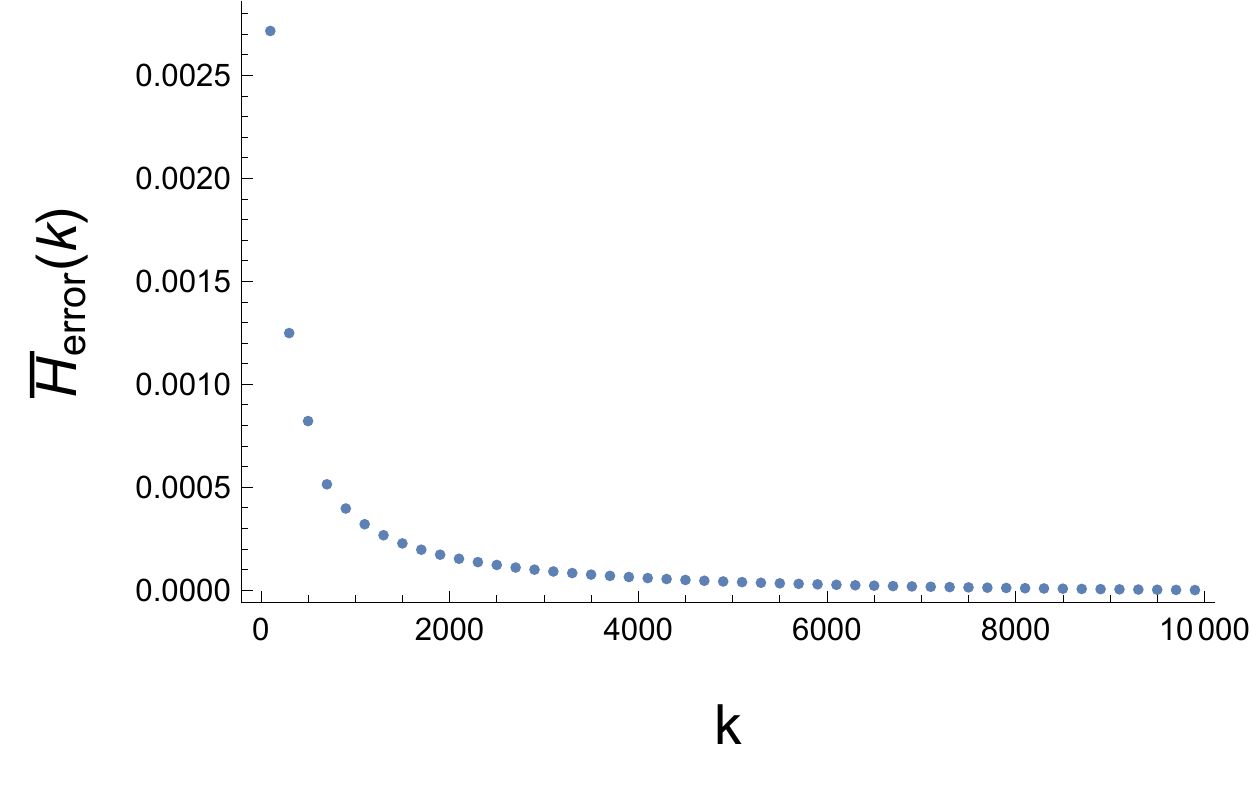}
		\caption{$\overline{H}_{error}(k)$}
	\end{subfigure}
	\caption{The convergence of solutions, approximated by \eqref{ApproxHessianMonotoneFlow}, of \eqref{ApproxiMFG} as $k$ increases when $d=2$, $P=(1.273,-1.2)$, and $H(x,p)=\frac{\left|p_1\right|^2}{2}+\frac{\left|p_2\right|^2}{2}+\cos\left(2\pi x_1\right)+\cos\left(2\pi x_2\right)$.}
	\label{HessianflowKConv2D}
\end{figure} 

\begin{table}[h!]
	\begin{tabular}{||c c c c c||} 
		\hline
		k & 10 & $10^2$ & $10^3$ & $10^4$\\ 
		\hline
		$\overline{\boldsymbol{H}}(P)$ (HRF) & 4.40251 & 4.40916 & 4.40989 & 4.40996 \\ 
		\hline
		$\overline{\boldsymbol{H}}(P)$ (NM) & 4.40935 & 4.40994 & 4.40996 & 4.40996 \\ 
		\hline
	\end{tabular}
	\caption{$\overline{\boldsymbol{H}}(P)$ for different values of $k$ when $d=2$, $P=(1.5,2.5)$, and $H(x,p)=\frac{\left|p_1\right|^2}{2}+\frac{\left|p_2\right|^2}{2}+\cos\left(2\pi x_1\right)+\cos\left(2\pi x_2\right)$.}
	\label{TwoDimensions}
\end{table}
{
Next, we study a non-separable Hamiltonian. For $x=(x_1,x_2)\in \mathbb{T}^2$ and $p=(p_1,p_2)\in \mathbb{R}^2$,  We consider
\begin{align*}
H(x,p)=\frac{\left|p_1\right|^2}{2}+\frac{\left|p_2\right|^2}{2}+\sin(2\pi x_1)\sin(2\pi x_2).
\end{align*}
In this case, we do not know the explicit solution of $\overline{H}$. In the numerical experiments, we fix $N=144$. Let $P=(P_1,P_2)\in \mathbb{R}^2$. We plot $\overline{\boldsymbol{H}}(P)$ at $t=14$ for $k=10^4$ and $P\in [-2,2]\times[-2,2]$ in Figure \ref{NonSepHRFHBar2D}.   For $P=(1.5,2.5)$, we plot in Figure \ref{NonSepTwoDimM} the approximated results for $\boldsymbol{m}$ and $\boldsymbol{u}$ at $t=14$ when $k=10^4$. Meanwhile, in Figure \ref{NonSepHessianflowKConv2D}, we plot the errors defined in  \eqref{uerrordef}, \eqref{merrordef}, and \eqref{Herrordef}.  
\begin{figure}
	\begin{subfigure}[b]{0.3\textwidth}
		\includegraphics[width=\textwidth]{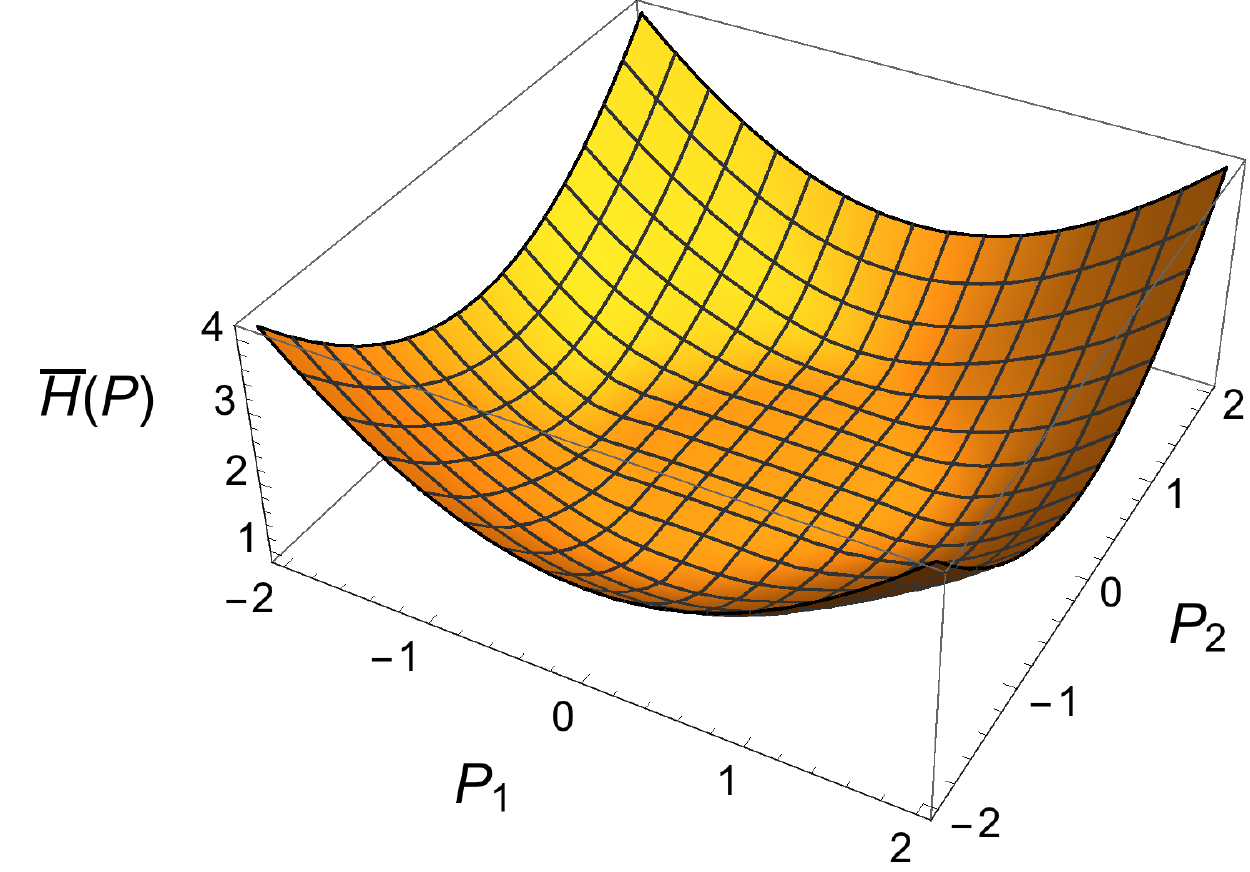}
	\end{subfigure}
	\caption{$\overline{\boldsymbol{H}}(P)$ computed by \eqref{ApproxHessianMonotoneFlow} when $d=2$ and $H(x,p)=\frac{\left|p_1\right|^2}{2}+\frac{\left|p_2\right|^2}{2}+\sin(2\pi x_1)\sin(2\pi x_2)$.} 
	\label{NonSepHRFHBar2D}
\end{figure}
\begin{figure}
	\begin{subfigure}[b]{0.3\textwidth}
		\includegraphics[width=\textwidth]{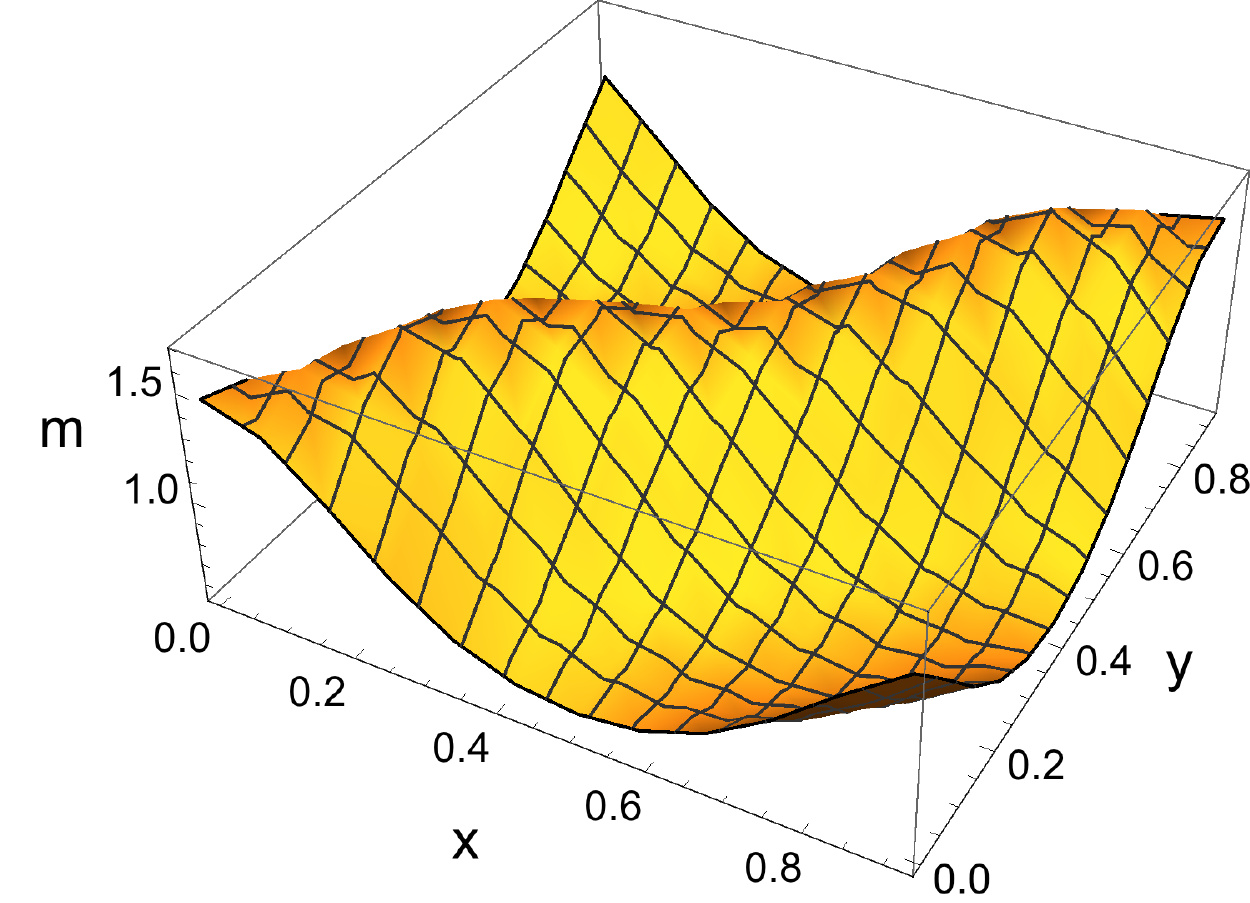}
		\caption{$m$ (HRF)}
	\end{subfigure}
	\begin{subfigure}[b]{0.3\textwidth}
		\includegraphics[width=\textwidth]{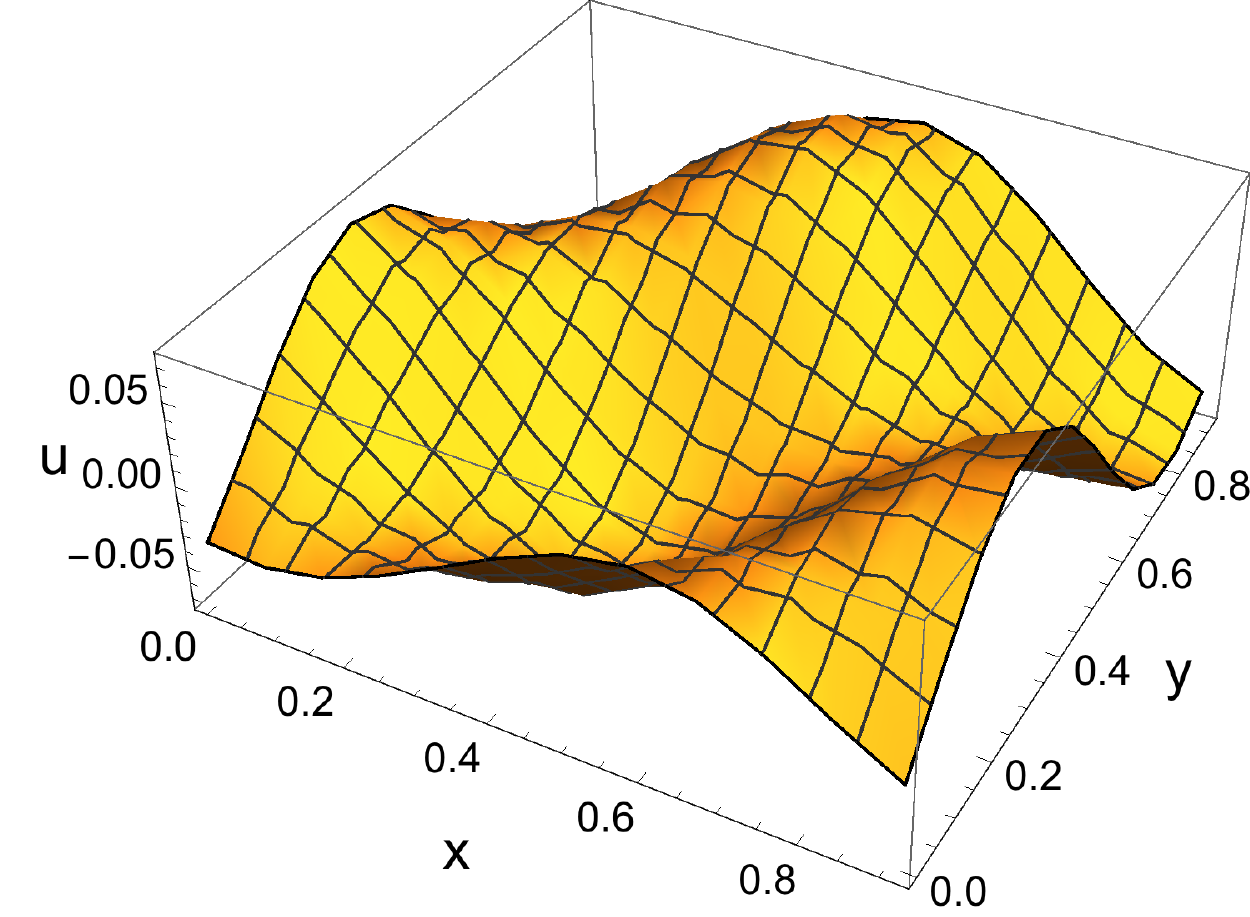}
		\caption{$u$ (HRF)}
	\end{subfigure}
	\\
	\begin{subfigure}[b]{0.3\textwidth}
		\includegraphics[width=\textwidth]{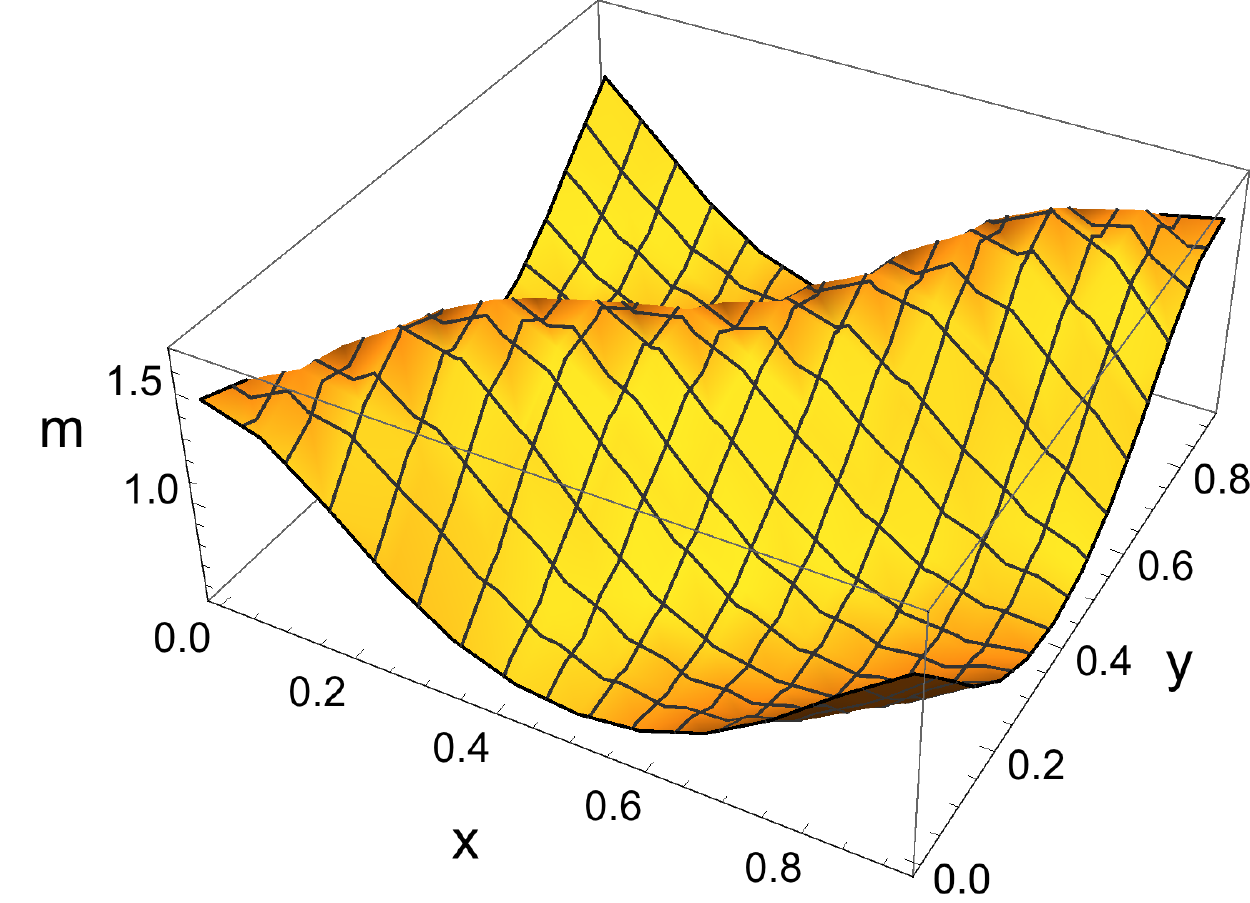}
		\caption{$m$ (NM)}
	\end{subfigure}
	\begin{subfigure}[b]{0.3\textwidth}
		\includegraphics[width=\textwidth]{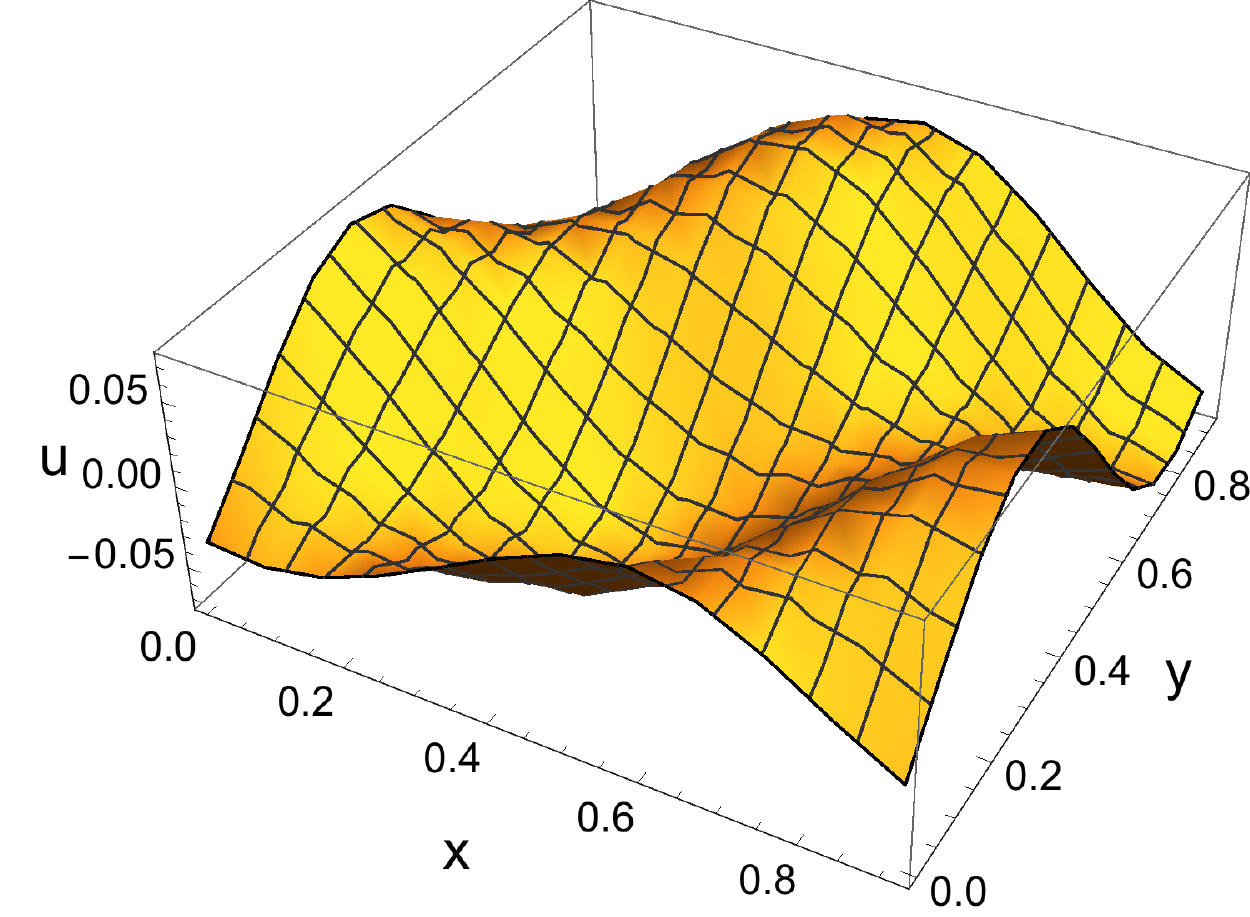}
		\caption{$u$ (NM)}
	\end{subfigure}
	\caption{Numerical solutions of
		Hessian Riemannian flow (HRF),
		\eqref{ApproxHessianMonotoneFlow}, and Newton's method (NM), \eqref{NewtonExplicitDiscretization}, at $t=14$ when $d=2$, $P=(1.5, 2.5)$, and $H(x,p)=\frac{\left|p_1\right|^2}{2}+\frac{\left|p_2\right|^2}{2}+\sin(2\pi x_1)\sin(2\pi x_2)$.}
	\label{NonSepTwoDimM}
\end{figure}
\begin{figure}
	\begin{subfigure}[b]{0.3\textwidth}
		\includegraphics[width=\textwidth]{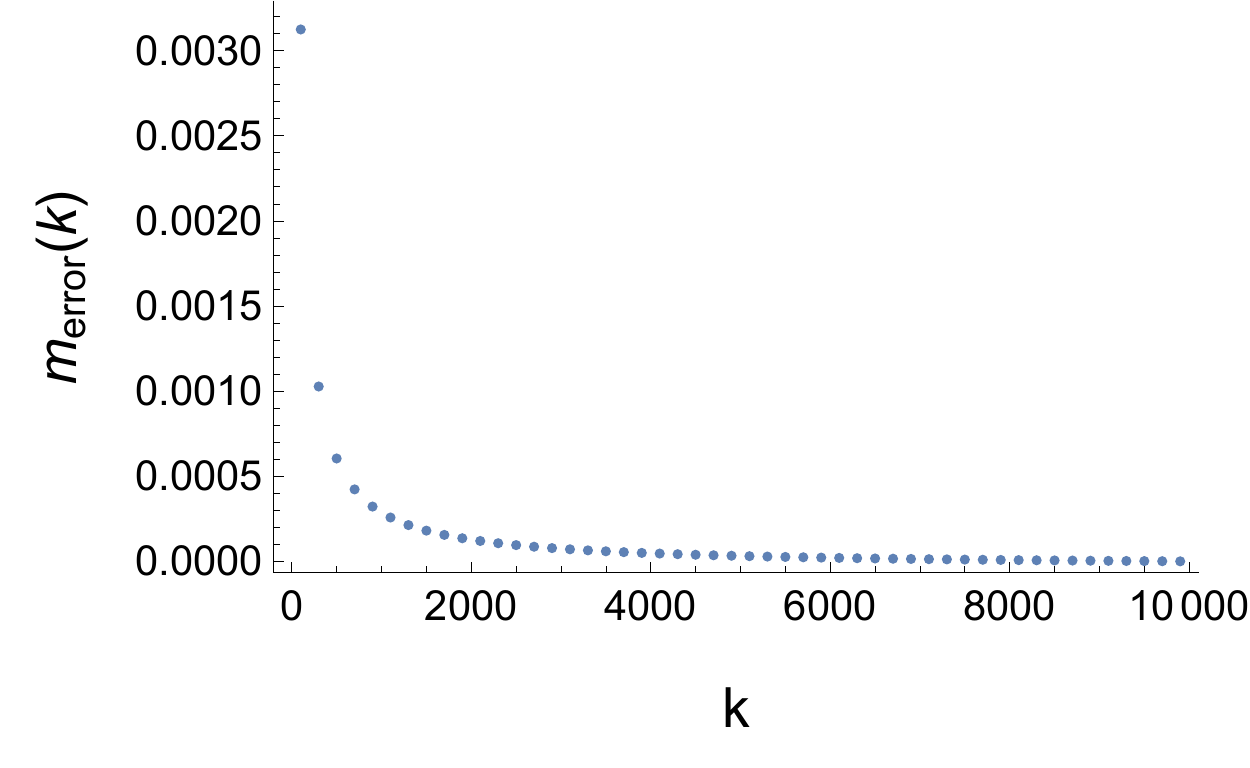}
		\caption{$m_{error}(k)$}
	\end{subfigure}
	\begin{subfigure}[b]{0.3\textwidth}
		\includegraphics[width=\textwidth]{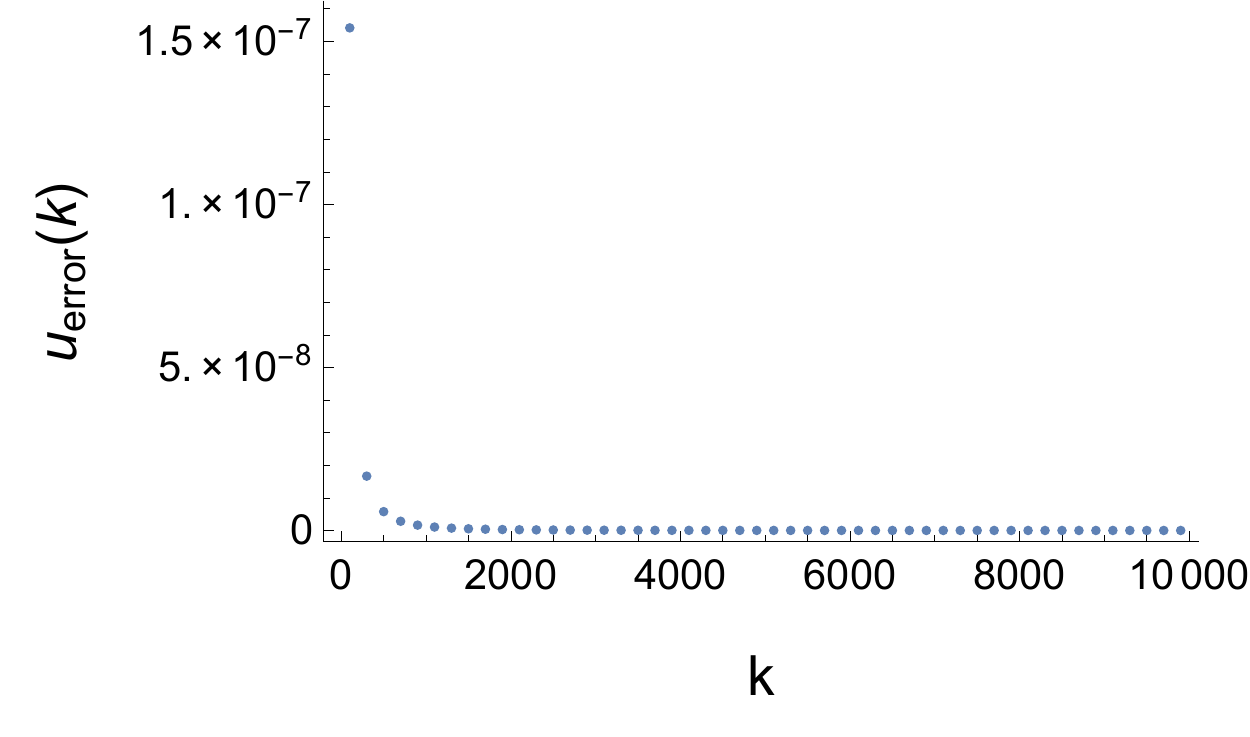}
		\caption{$u_{error}(k)$}
	\end{subfigure}
	\begin{subfigure}[b]{0.3\textwidth}
		\includegraphics[width=\textwidth]{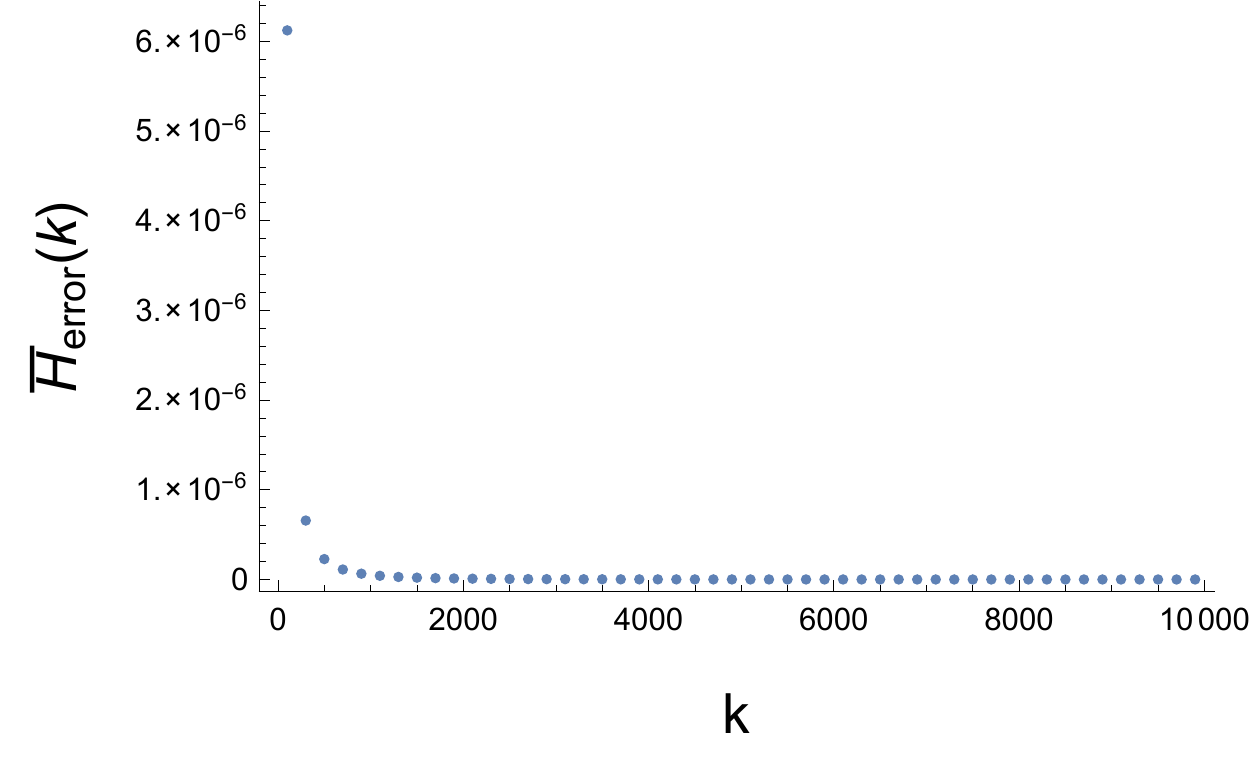}
		\caption{$\overline{H}_{error}(k)$}
	\end{subfigure}
	\caption{The convergence of solutions, approximated by \eqref{ApproxHessianMonotoneFlow}, of \eqref{ApproxiMFG} as $k$ increases when $d=2$, $P=(1.5, 2.5)$, and $H(x,p)=\frac{\left|p_1\right|^2}{2}+\frac{\left|p_2\right|^2}{2}+\sin(2\pi x_1)\sin(2\pi x_2)$.}
	\label{NonSepHessianflowKConv2D}
\end{figure} 
}

\subsection{Non-monotonicity of $\overline{F}$}
\label{NonMonotonicityOfHessianFlow}
Lemma \ref{Monotonicity} implies that $\widetilde{F}$ is monotone. However, the operator, $\overline{F}$, in \eqref{overlineF} may not be monotone. 

To illustrate the non-monotonicity of $\overline{F}$, we choose $$H(x,p)=\frac{|p|^2}{2}-10\cos(2\pi x)-10\sin(2\pi x).$$ In the simulation, we set $P=0.5$, $k=100$, and $N=20$.  Here, we compute two trajectories generated by \eqref{DynDiscreteHessianFlow} from two sets of initial values, $(M^0,U^0)=(m_1^0,\dots,m_N^0,u^0_1,\dots,u_N^0)$ and $(\widetilde{M}^0,\widetilde{U}^0)=\left(\widetilde{m}_1^0,\dots,\widetilde{m}_N^0,\widetilde{u}^0_1,\dots,\widetilde{u}^0_N\right)$, where $u^0_i=\cos(2 \pi x_i)$, $\widetilde{u}^0_i=\sin(2\pi x_i), m^0_i=1 + 0.2\cos(2 \pi x_i)$, $\widetilde{m}^0_i=1+0.7\cos(2\pi x_i)$. We represent the solutions corresponding to $(M^0,U^0)$ and $(\widetilde{M}^0,\widetilde{U}^0)$ by $$(M(t),U(t))=(m_1(t),\dots,m_N(t), u_1(t),\dots,u_N(t))$$ and  $$(\widetilde{M}(t),\widetilde{U}(t))=(\widetilde{m}_1(t),\dots,\widetilde{m}_N(t),\widetilde{u}_1(t),\dots,\widetilde{u}_N(t)),$$ respectively.
If $\overline{F}$ were monotone, we would have
\begin{align*}
&\frac{d}{dt}\left(\sum_{i=1}^{N}\left(\left(m_i(t)-\widetilde{m}_i(t)\right)^2+\left(u_i(t)-\widetilde{u}_i(t)\right)^2\right)\right)\\
=&-\left\langle \overline{F}(M(t),U(t))-\overline{F}(\widetilde{M}(t),\widetilde{U}(t)),(M(t),U(t))-(\widetilde{M}(t),\widetilde{U}(t)) \right\rangle\leq 0.
\end{align*}
Hence, we plot the values of $\sum \limits_{i=1}^{N}\left(\left(m_i(t)-\widetilde{m}_i(t)\right)^2+\left(u_i(t)-\widetilde{u}_i(t)\right)^2\right)$ versus time in Figure \ref{nonmotonicity}, which shows that the curve is not strictly decreasing. Thus, $\overline{F}$ fails to be monotone. In contrast, we plot $\overline{\phi}$ defined in \eqref{PhiBar} (see Figure \ref{motonicity}), which shows that $\overline{\phi}(t)$ is decreasing, as expected.

\begin{figure}
	\begin{subfigure}[b]{0.3\textwidth}
		\includegraphics[width=\textwidth]{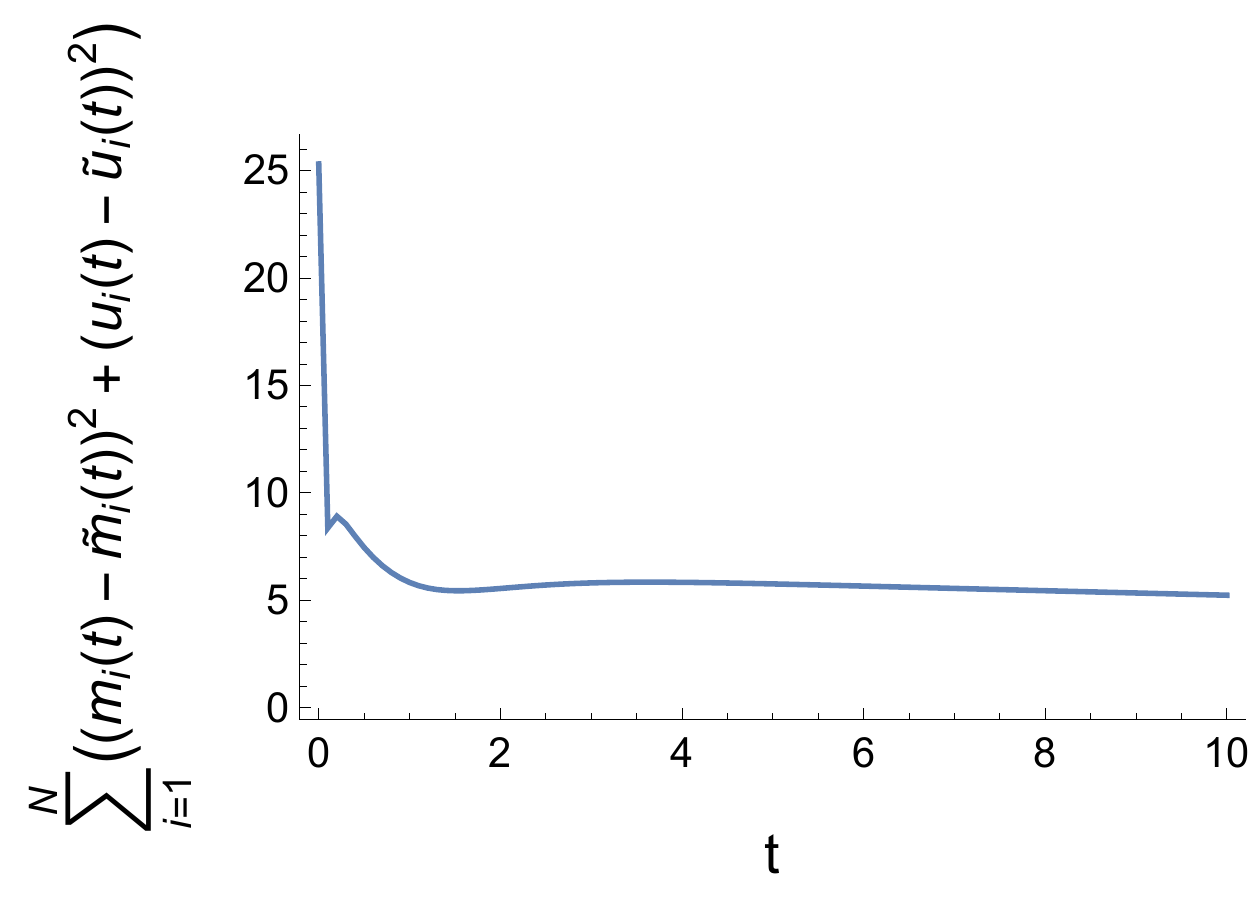}
		\caption{Non-monotonicity of $\overline{F}$.}
		\label{nonmotonicity}
	\end{subfigure}
	\begin{subfigure}[b]{0.3\textwidth}
		\includegraphics[width=\textwidth]{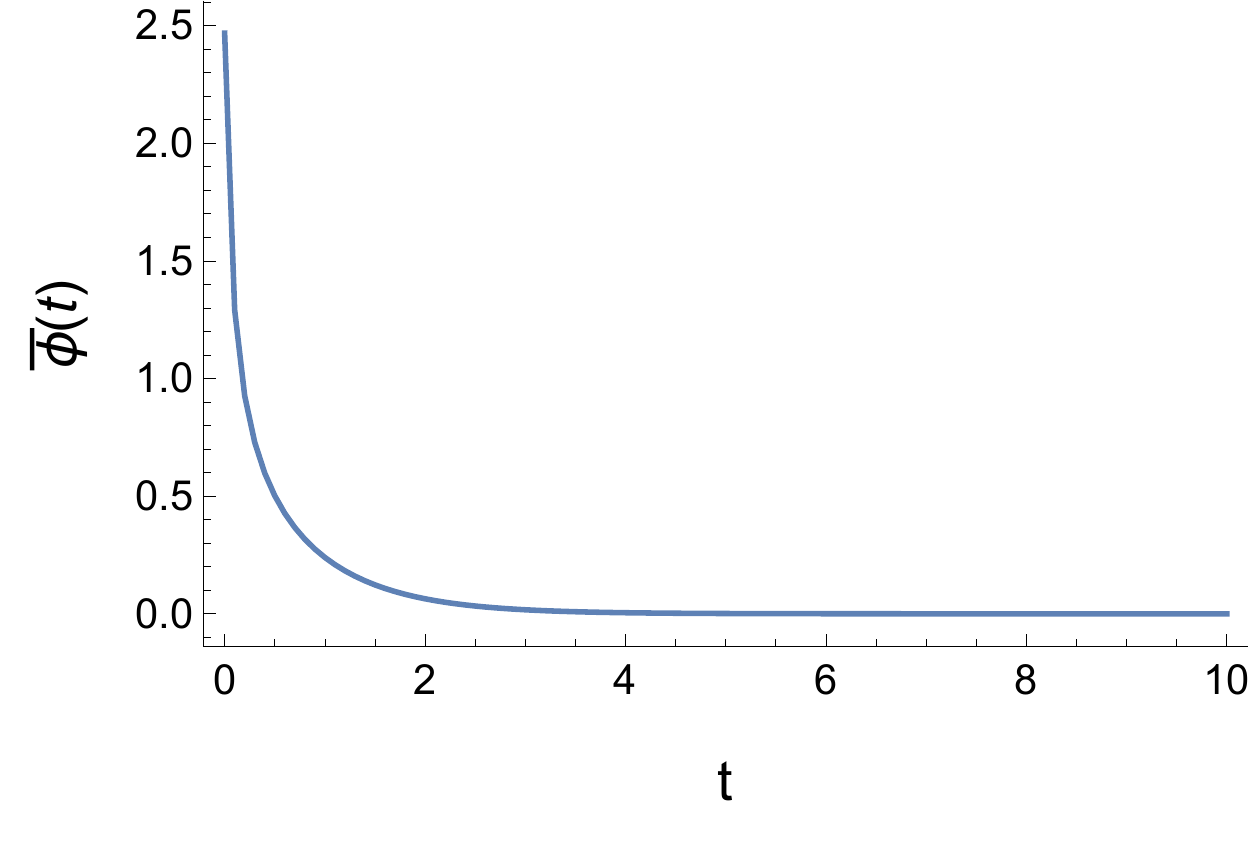}
		\caption{Monotonicity of $\overline{\phi}(t)$.}
		\label{motonicity}
	\end{subfigure}
	\caption{Monotonicity test}
	\label{MonotonicityTest}
\end{figure}

\subsection{Speed comparison between the Hessian Riemannian flow and Newton's  method}
Here, we compare the speed of the Hessian Riemannian flow in \eqref{DynDiscreteHessianFlow}, which is solved by NDSolve, with the speed of  Newton's method in \eqref{NewtonExplicitDiscretization}. 

We consider the Hamiltonian, $$H(x,p)=\frac{|p|^2}{2}-\sin(2\pi x).$$ 
In the numerical experiment, we set $P=0.5$ and $k=10^2$. In this case,  $\overline{H}(P)=1$. The initial point is given by  $(M^0,U^0)=(m_1^0,\dots,m_N^0,u_1^0,\dots,u_N^0)$, where $m_i^0=1+0.9\cos(2\pi x_i)$ and $u_i^0=0.2\cos(2\pi x_i)$. For Newton's method, we choose $\tau=\kappa=1$. For each value of $N$, we compute $\overline{\boldsymbol{H}}$ by { the Hessian Riemannian flow for a large time $T^\diamond$ and use it as a benchmark, named $\overline{\boldsymbol{H}}^\diamond$.} Then, we use the Hessian Riemannian flow and  Newton's method to compute $\overline{\boldsymbol{H}}(T)$ such that  $|\overline{\boldsymbol{H}}-\overline{\boldsymbol{H}}^\diamond|<\epsilon$. {To solve the Hessian Riemannian flow, we use the built-in function, NDSolve, in Mathematica and use two different methods, named LSODA and BDF, separately. Then, we record $T$, the interior number of iterations of NDSolve, and the corresponding CPU time (measured in seconds) in Table \ref{CPUtimeComparison}}. Here, we choose $T^\diamond=50$ and $\epsilon=0.001$. We see that  Newton's method is substantially faster than the Hessian Riemannian flow { solved by LSODA and BDF.} 
{ 
\begin{table}[h!]
	\begin{tabular}{||c c c c c||} 
		\hline
		N & 15 & 30 & 60 & 120\\ 
		\hline
		$\overline{\boldsymbol{H}}^\diamond$ & 0.964609 & 0.964754 & 0.96476 & 0.96476 \\ 
		\hline
		$T$ (HRF, BDF) & 44 & 44 & 44 & 44 \\
		\hline
		$T$ (HRF, LSODA) & 44 & 44 & 44 & 44 \\
		\hline
		$T$ (NM) & 44 & 43 & 43 & 43 \\
		\hline
		Num. of iterations (HRF, BDF) & 720 & 829 & 763 & 639 \\
		\hline
		Num. of iterations (HRF, LSODA) & 1473 & 1629 & 3366 & 1025 \\
		\hline
		Num. of iterations (NM) & 44 & 43  & 43  & 43 \\
		\hline
		CPU time (HRF, BDF) & 1.941897 & 11.285296  & 67.163746 & 2378.225083 \\
		\hline
		CPU time (HRF, LSODA) & 1.514499 & 8.101890 & 48.836245 & 1841.271556 \\
		\hline
		CPU time (NM) & 0.009373 & 0.040518 & 0.092378  & 0.364952 \\
		\hline
	\end{tabular}
	\caption{The Hessian Riemannian flow (HRF) vs.  Newton's method (NM).}
	\label{CPUtimeComparison}
\end{table}
}
\subsection{Stability of the Hessian Riemannian flow and  Newton's  method}

\begin{figure}[h]
	\begin{subfigure}[b]{0.3\textwidth}
		\includegraphics[width=\textwidth]{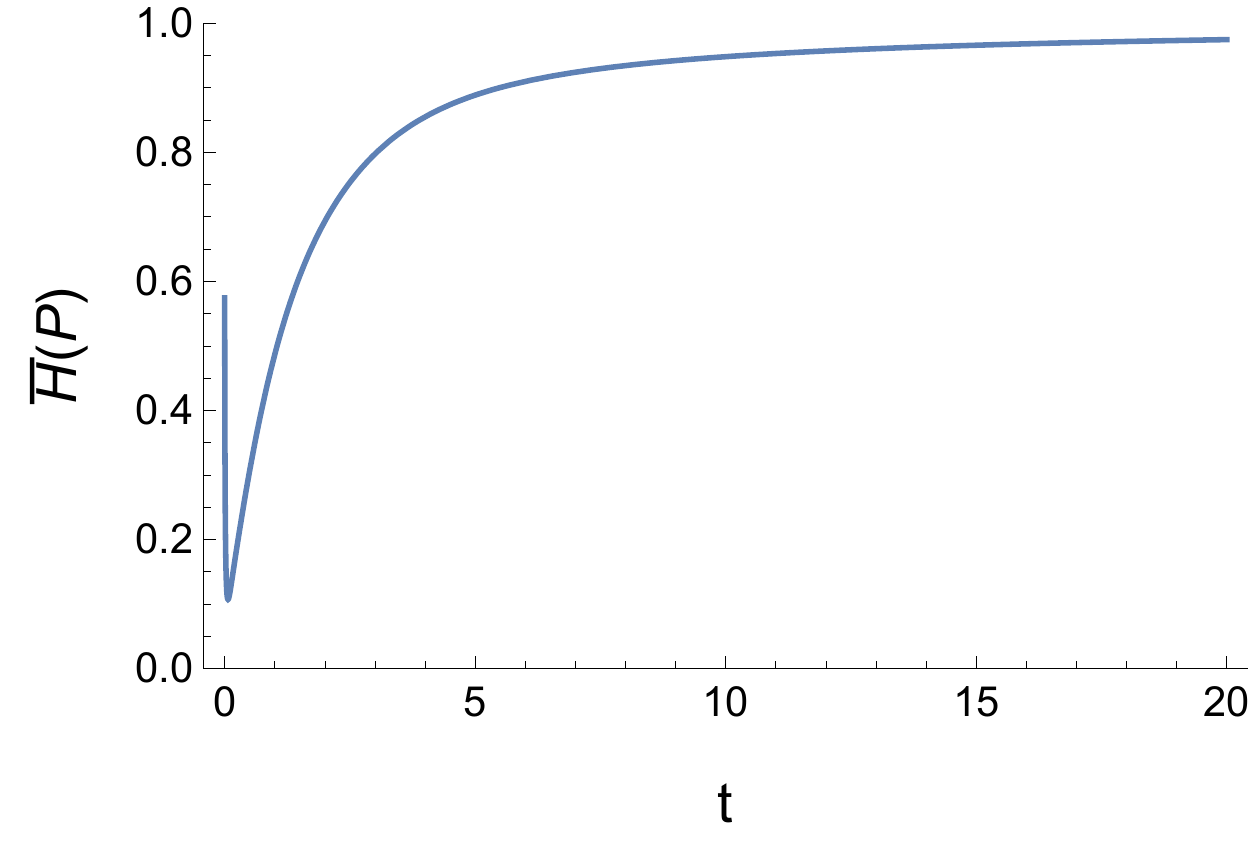}
		\caption{$\overline{\boldsymbol{H}}(P)$ (HRF)}
	\end{subfigure}
	\begin{subfigure}[b]{0.3\textwidth}
		\includegraphics[width=\textwidth]{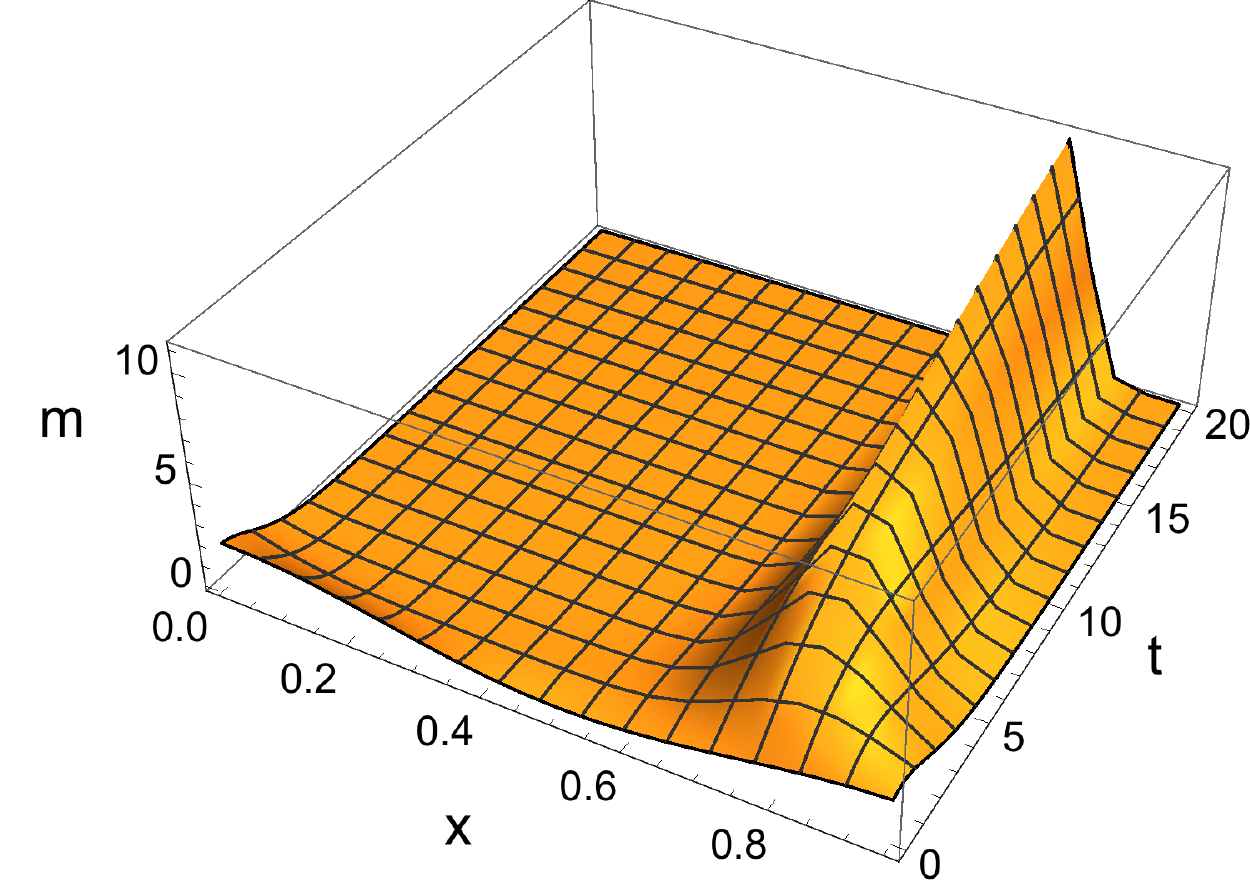}
		\caption{$\boldsymbol{m}$ (HRF)}
	\end{subfigure}
	\begin{subfigure}[b]{0.3\textwidth}
		\includegraphics[width=\textwidth]{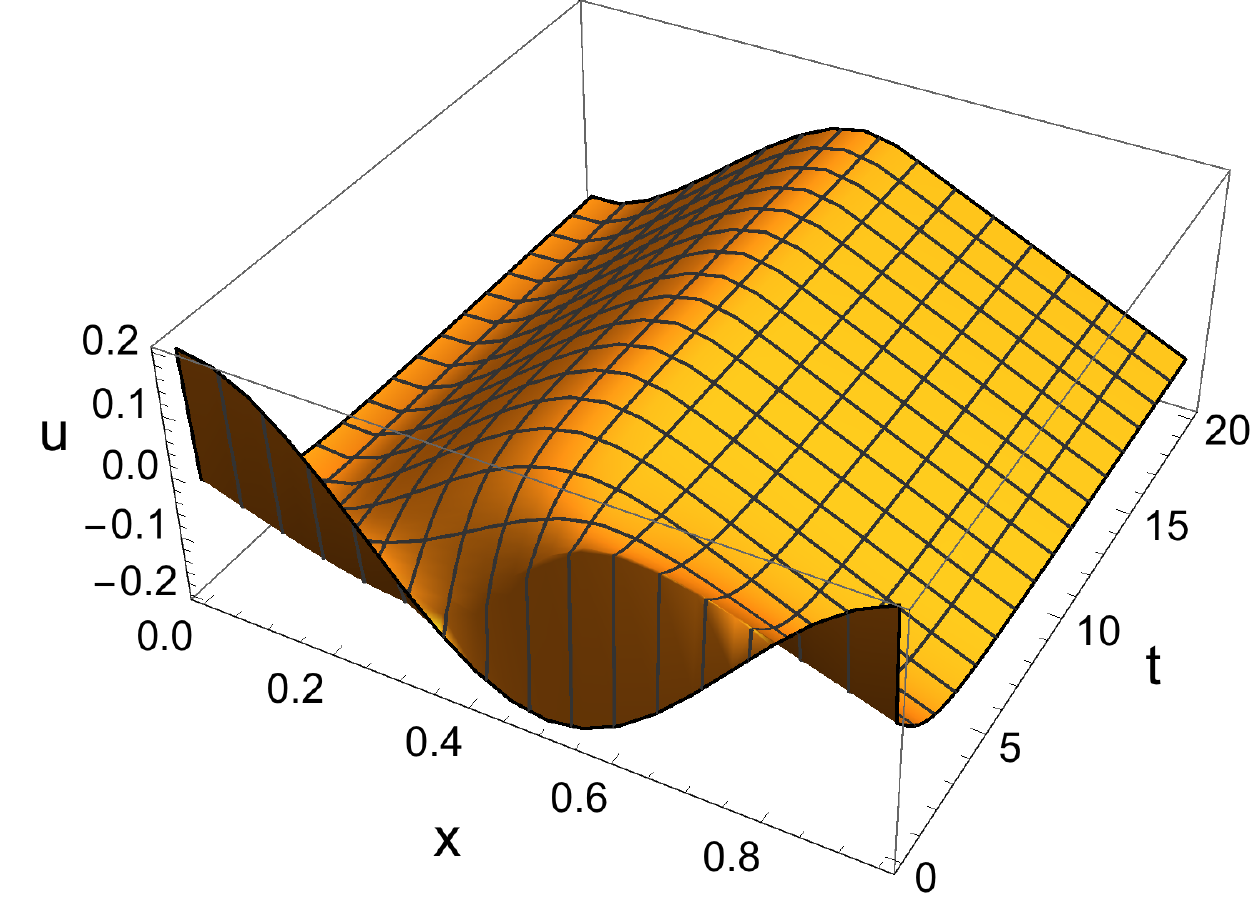}
		\caption{$\boldsymbol{u}$ (HRF)}
	\end{subfigure}
	\\
	\begin{subfigure}[b]{0.3\textwidth}
		\includegraphics[width=\textwidth]{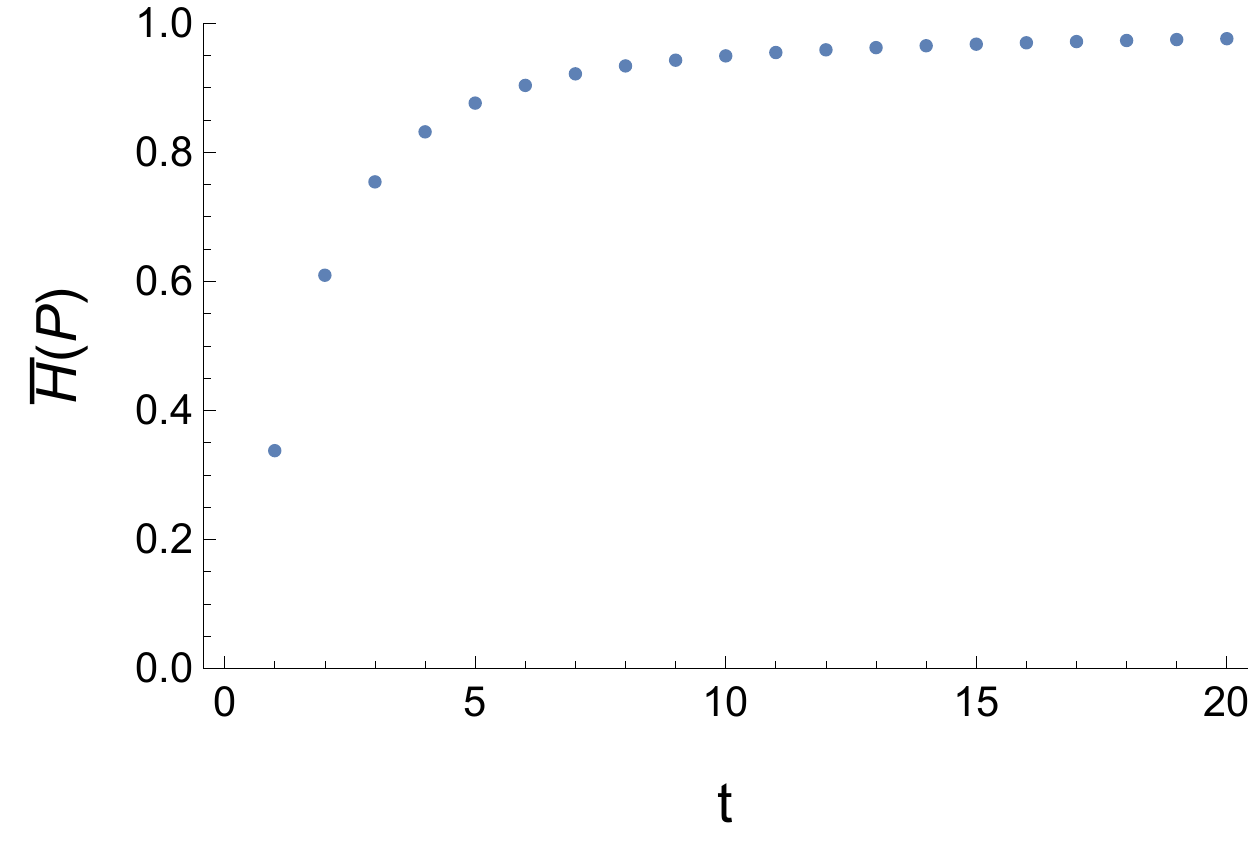}
		\caption{$\overline{\boldsymbol{H}}(P)$ (NM)}
	\end{subfigure}
	\begin{subfigure}[b]{0.3\textwidth}
		\includegraphics[width=\textwidth]{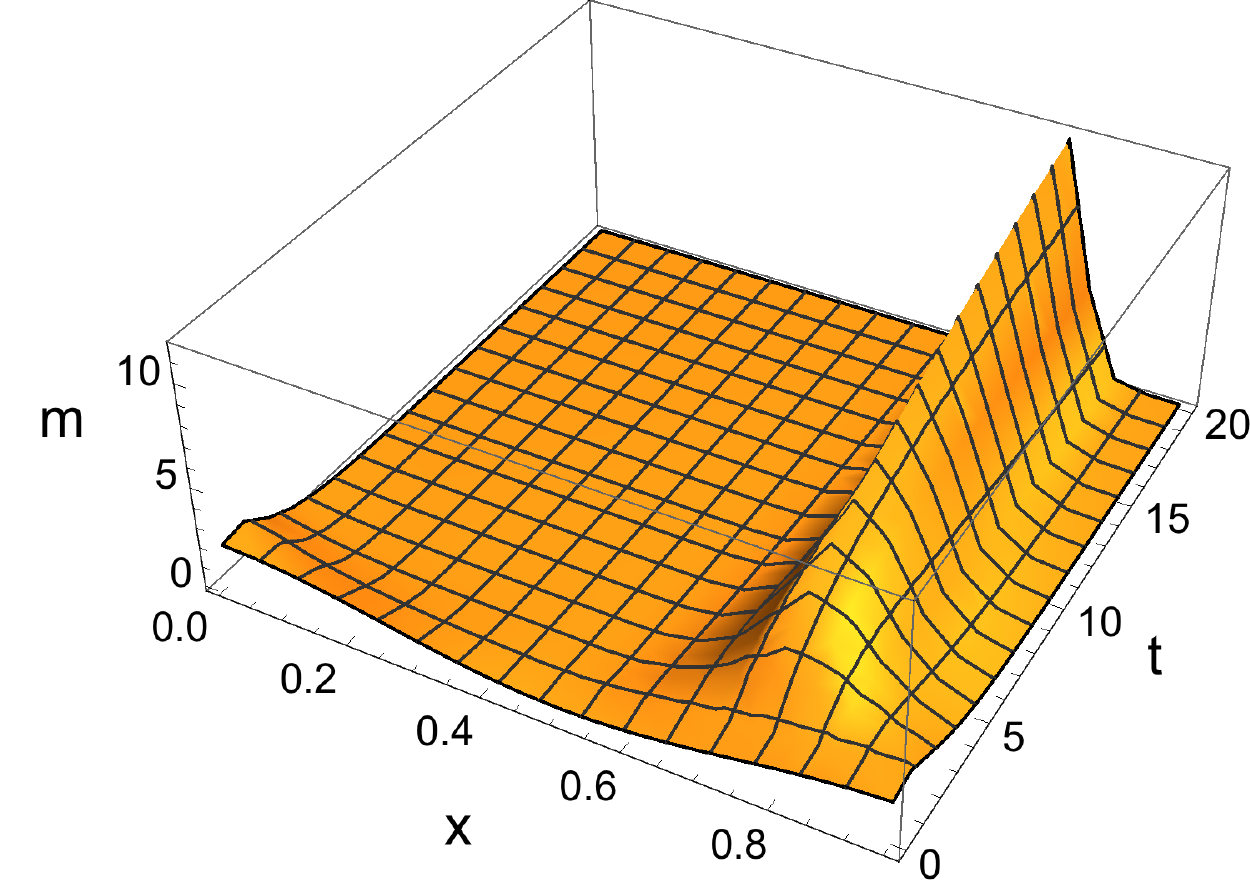}
		\caption{$\boldsymbol{m}$ (NM)}
	\end{subfigure}
	\begin{subfigure}[b]{0.3\textwidth}
		\includegraphics[width=\textwidth]{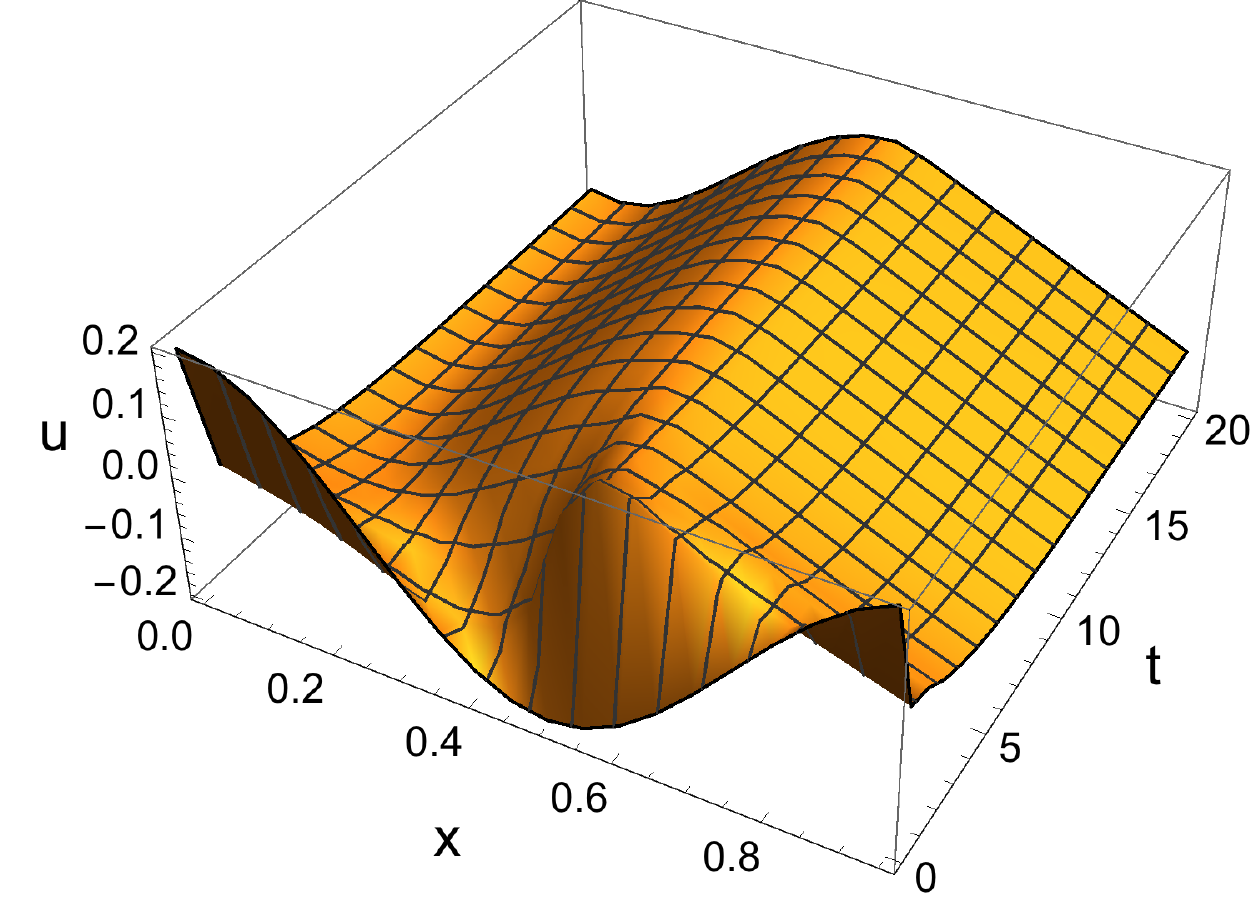}
		\caption{$\boldsymbol{u}$ (NM)}
	\end{subfigure}
	\caption{Numerical solutions of
		 Hessian Riemannian flow (HRF),
		\eqref{ApproxHessianMonotoneFlow}, and Newton's method (NM), \eqref{NewtonExplicitDiscretization}, for $k=10^5$.}
	\label{StabilityOfTheHessianflow}
\end{figure}

Though not stated explicitly in  \cite{falcone2008variational}, the algorithm described there computes both $\overline{H}$ and the projected Mather measure. However, as stated in that paper, that scheme becomes unstable if {$k$ is too large compared to the mesh size $N$.} Here, we show that the Hessian Riemannian flow and  Newton's method overcome this issue. 

To illustrate the stability of our methods, we consider $$H(x,p)=\frac{|p|^2}{2}-\sin(2\pi x).$$ For the implementation, we choose $P=0.5$, $k=10^5$, and $N=20$.  We use the initial value $(M^0,U^0)=(m_1^0,\dots,m_N^0,u_1^0,\dots,u_N^0)$, where $m_i^0=1+0.9\cos(2\pi x_i)$ and $u_i^0=0.2\cos(2\pi x_i)$. Figure  \ref{StabilityOfTheHessianflow} shows  the evolution of $\overline{\boldsymbol{H}}(P)$, $\boldsymbol{m}$ and $\boldsymbol{u}$ by the Hessian Riemannian flow and  Newton's method, which
illustrates that the Hessian Riemannian flow and  Newton's method are stable for nearly singular equations, corresponding to a large value of $k$.

\section{Conclusion}
\label{section7}
In this paper, {to calculate simultaneously the effective Hamiltonian and the Mather measure, we suggested two methods, the Hessian Riemannian flow and  Newton's method, to compute the approximated system in \eqref{ApproxiMFG}.} {We proved the convergence of the Hessian Riemannian flow in the continuous setting and gived both the existence and the  convergence of the Hessian Riemannian flow in the discrete setting.} We showed that this method guarantees the non-negativity of $m$. Besides, we pointed out the relation between the implicit discretization of the Hessian Riemannian flow and  Newton's method.  In our numerical experiments,  Newton's method is faster than the Hessian Riemannian flow. Both methods preserve the positivity of the Mather measure. Moreover, the Hessian Riemannian flow and  Newton's method seem to be stable for large $k$, a case where the variational method in \cite{falcone2008variational} faces difficulties. {Our experiments show that the solution to \eqref{ApproxiMFG} seems to converge as $k\rightarrow+\infty$ even when the values of $P$ are in areas where convergence has not yet been proven. We stress that our algorithm can also be used to solve stationary mean-field games. }

\def\cprime{$'$}

\end{document}